\documentclass[11pt,reqno,oneside]{amsart}

\usepackage[utf8]{inputenc}
\usepackage[T1]{fontenc}
\usepackage[top=3cm,bottom=3cm,left=3cm,right=3cm]{geometry}

\PassOptionsToPackage{dvipsnames}{xcolor}

\usepackage{orcidlink}

\usepackage{hyperref}
\usepackage{mathtools}
\usepackage{thmtools}
\usepackage[capitalize, nameinlink]{cleveref}

\usepackage{xcolor}
\definecolor{red}{rgb}{1,0,0}

\hypersetup{
    colorlinks,
    linkcolor={red!50!black},
    citecolor={blue!50!black},
    urlcolor={blue!80!black}
}

\usepackage{amsfonts}
\usepackage{amssymb}
\usepackage{mathrsfs}
\usepackage{textcomp}
\usepackage{longtable}
\usepackage{booktabs}
\usepackage{graphicx}
\usepackage{scalerel}
\usepackage{setspace}
\usepackage{nicefrac}
\usepackage{indentfirst}
\usepackage{enumerate}
\usepackage[normalem]{ulem}
\usepackage{paralist}
\usepackage{etoolbox}
\usepackage{mathdots}
\usepackage{wrapfig}
\usepackage{floatflt}
\usepackage{tensor} 
\usepackage{parskip}
\usepackage{lscape}
\usepackage{float}
\usepackage{subcaption}

\usepackage{enumitem}

\usepackage{tikz}
\usetikzlibrary{arrows}
\usetikzlibrary{matrix}
\usetikzlibrary{positioning}
\usetikzlibrary{arrows.meta}
\usetikzlibrary{calc}
\usetikzlibrary{intersections}

\usepackage[all,cmtip]{xy} 
\usepackage[labelfont=bf, font={small}]{caption}[2005/07/16]

\newcommand{\xto}[1]{\xrightarrow{\phantom{a}{#1}{\phantom{a}}}}

\newcommand{\vvirg}{ , \dots , }
\newcommand{\ootimes}{ \otimes \ldots \otimes }
\newcommand{\ooplus}{ \oplus \ldots \oplus }
\newcommand{\ttimes}{ \times \ldots \times }

\newcommand{\Wedge}{\bigwedge}

\newcommand{\bfe}{\mathbf{e}}

\newcommand{\bfp}{\mathbf{p}}

\newcommand{\bft}{\mathbf{t}}

\newcommand{\calI}{\mathcal{I}}

\newcommand{\calL}{\mathcal{L}}

\newcommand{\calO}{\mathcal{O}}

\newcommand{\calQ}{\mathcal{Q}}

\newcommand{\calS}{\mathcal{S}}

\newcommand{\calU}{\mathcal{U}}

\newcommand{\calZ}{\mathcal{Z}}

\newcommand{\scrC}{\mathscr{C}}

\newcommand{\bbA}{\mathbb{A}}

\newcommand{\bbC}{\mathbb{C}}

\newcommand{\bbF}{\mathbb{F}}

\newcommand{\bbP}{\mathbb{P}}
\newcommand{\bbQ}{\mathbb{Q}}

\newcommand{\bbZ}{\mathbb{Z}}

\newcommand{\frakm}{\mathfrak{m}}

\newcommand{\rmN}{\mathrm{N}}

\newcommand{\bfepsilon}{\boldsymbol{\epsilon}}

\newcommand{\bfeta}{\boldsymbol{\eta}}

\renewcommand{\phi}{\varphi}
\newcommand{\eps}{\varepsilon}

\newcommand{\dashto}{\dashrightarrow}

\renewcommand{\tilde}[1]{\widetilde{#1}}
\renewcommand{\hat}[1]{\widehat{#1}}
\renewcommand{\bar}[1]{\overline{#1}}

\newcommand{\rank}{\mathrm{rank}}
\newcommand{\im}{\mathrm{im}}  
  
\newcommand{\image}{\mathrm{image}}  
\DeclareMathOperator{\codim}{codim}

\DeclareMathOperator{\Hom}{Hom}

\newcommand{\sat}{\mathrm{sat}}

\newcommand{\SO}{\mathrm{SO}}
\newcommand{\GL}{\mathrm{GL}}
\newcommand{\SL}{\mathrm{SL}}

\newcommand{\Bl}{\mathrm{Bl}}

\newcommand{\TC}{\mathrm{TC}}

\newcommand{\reg}{\mathrm{reg}}

\newcommand{\rmHF}{\mathrm{HF}}

\newcommand{\CH}{\mathrm{CH}}
\newcommand{\Mat}{\mathrm{Mat}}

\newtheorem{theorem}{Theorem}[section]

\newtheorem{lemma}[theorem]{Lemma}
\newtheorem{proposition}[theorem]{Proposition}
\newtheorem{corollary}[theorem]{Corollary}

\newtheorem{thmx}{Main Theorem}

\newtheorem*{Greenthm*}{Green's Theorem}
\newtheorem*{liftlemma*}{Lifting Lemma}

\theoremstyle{definition}
\newtheorem{definition}[theorem]{Definition}
\newtheorem{remark}[theorem]{Remark}
\newtheorem{example}[theorem]{Example}
\newtheorem{notation}[theorem]{Notation}

\crefname{theorem}{Theorem}{Theorems}
\crefname{lemma}{Lemma}{Lemmas}
\crefname{proposition}{Proposition}{Propositions}
\crefname{corollary}{Corollary}{Corollaries}
\crefname{definition}{Definition}{Definitions}
\crefname{remark}{Remark}{Remarks}
\crefname{example}{Example}{Examples}
\crefname{notation}{Notation}{Notations}
\crefname{conjecture}{Conjecture}{Conjectures}

\numberwithin{equation}{section}
\numberwithin{figure}{section}

\usepackage[textwidth=0.9in]{todonotes}

\newcommand{\PP}{\mathbb{P}}

\usepackage{comment}

\newcommand{\CC}{\mathbb{C}}

\newcommand{\Gr}{\mathrm{Gr}}
\newcommand{\CQ}{\mathit{CQ}}
\newcommand{\CCol}{\mathit{CC}}

\newcommand{\N}{\mathrm{N}}
\DeclareMathOperator{\YL}{\mathnormal{Y_L}}
\newcommand{\rmleft}{\mathrm{left}}
\newcommand{\rmright}{\mathrm{right}}
\newcommand{\Diag}{\mathrm{Diag}}
\newcommand{\Seg}{\mathrm{Seg}}

\newcommand{\WedgeScale}{1.1}
\newcommand{\WedgeVShift}{-0.5pt}
\renewcommand{\Wedge}{\smash{\raisebox{\WedgeVShift}{\scalebox{\WedgeScale}{\scalerel*{\bigwedge}{\Lambda}}}}}

\newcommand{\drawchain}[8]{
    \def\startX{#1}
    \def\N{#2}
    \def\colors{#3}
    \def\lineLabels{#4}
    \def\interLabels{#5}
    \def\twists{#6}
    \def\twisttype{#7}
    \def\lengths{#8}

    \foreach \n [evaluate=\n as \arrayindex using int(\n-1)] in {1,...,\N} {
        
        \pgfmathparse{\colors[\arrayindex]}
        \let\colorkey\pgfmathresult
        \if\colorkey B \def\drawcolor{blue!50!black} \fi
        \if\colorkey G \def\drawcolor{green!50!black} \fi
        \if\colorkey R \def\drawcolor{red!80!black}  \fi
        \if\colorkey K \def\drawcolor{black!50!black}  \fi

        \pgfmathparse{\lineLabels[\arrayindex]} \let\lineL\pgfmathresult
        \pgfmathparse{\interLabels[\arrayindex]} \let\interL\pgfmathresult

        \pgfmathsetmacro{\rotangle}{mod(\n, 2) == 1 ? 30 : -30}
        \pgfmathtruncatemacro{\isodd}{mod(\n,2)}
        \ifnum\isodd=0 \def\anchorval{east} \else \def\anchorval{west} \fi
        \pgfmathsetmacro{\xshiftval}{\isodd == 1 ? -5 : 5}

        \pgfmathparse{\lengths[\arrayindex]} \let\halfL\pgfmathresult
        \pgfmathparse{\twists[\arrayindex]} \let\twistamt\pgfmathresult
        \pgfmathsetmacro{\ctrloneX}{\startX + \twistamt}
        \pgfmathsetmacro{\ctrloneY}{\n - \halfL/3}
        \if\twisttype S
            \pgfmathsetmacro{\ctrltwoX}{\startX - \twistamt}
        \else
            \pgfmathsetmacro{\ctrltwoX}{\startX + \twistamt}
        \fi
        \pgfmathsetmacro{\ctrltwoY}{\n + \halfL/3}
        
        \ifnum\isodd=0 \def\labelyshift{\drawchainEvenYshift} \else \def\labelyshift{0pt} \fi
        \draw[thick, \drawcolor, rotate around={\rotangle:(\startX,\n)}, name path=seg\n] 
            (\startX, \n - \halfL) .. controls (\ctrloneX, \ctrloneY) and (\ctrltwoX, \ctrltwoY) .. (\startX, \n + \halfL)
            node[pos=0.5, anchor=south \anchorval, xshift=\xshiftval pt, yshift=\labelyshift, font=\tiny, text=black] {\lineL};
                }
}

\providecommand{\drawchainEvenYshift}{0pt}

\newcommand{\drawhchain}[3]{%
    \pgfmathtruncatemacro{\drawhcNval}{#1}%
    \ifnum\drawhcNval>0
        \foreach \drawhcal [count=\drawhck] in {#2}
            {\expandafter\xdef\csname drawhcA\drawhck\endcsname{\drawhcal}}%
        \foreach \drawhcbl [count=\drawhck] in {#3}
            {\expandafter\xdef\csname drawhcB\drawhck\endcsname{\drawhcbl}}%
        \foreach \drawhcn in {1,...,\drawhcNval}{%
            \pgfmathsetmacro{\drawhccx}{(\drawhcn-1)*0.7}%
            \pgfmathsetmacro{\drawhcrot}{mod(\drawhcn,2)==1 ? 30 : -30}%
            \draw[thick, rotate around={\drawhcrot:(\drawhccx,0)}]
                (\drawhccx-0.45, 0) -- (\drawhccx+0.45, 0);%
            \node[above, font=\tiny] at (\drawhccx, 0.35) {\csname drawhcA\drawhcn\endcsname};%
            \node[below, font=\tiny] at (\drawhccx, -0.35) {\csname drawhcB\drawhcn\endcsname};%
        }%
    \fi%
}

\title{Schubert compactifications of matrix pencils}

\author{Vincenzo Galgano, Fulvio Gesmundo and Hanieh Keneshlou}

\address[V. Galgano \orcidlink{0000-0001-8778-575X}]{Max Planck Institute of Molecular Cell Biology and Genetics, Dresden, Germany; Faculty of Mathematics, Technische Universität Dresden, Germany}
\email{galgano@mpi-cbg.de}

\address[F. Gesmundo \orcidlink{0000-0001-6402-021X}]{Institut de Mathématiques de Toulouse; UMR5219 -- Université de Toulouse; CNRS -- UPS, F-31062 Toulouse Cedex 9, France}
\email{fgesmund@math.univ-toulouse.fr}

\address[H. Keneshlou \orcidlink{0009-0008-4910-8422}]{Faculty of Mathematics and Computer Science, Universität Würzburg, Germany}
\email{hanieh.keneshlou@uni-wuerzburg.de}

\usepackage[new]{old-arrows}

\begin{document}

\begin{abstract} 
The \emph{Schubert compactification of a linear space of matrices} $L\subseteq \Mat_{n\times m}$ is a natural compactification of $L$ in the Grassmannian $\Gr(n, n+m)$. This construction generalises classical Schubert varieties and matroid Schubert varieties, and it turns invariants of linear matrix spaces into projective and intersection-theoretic invariants of the associated variety. We develop this geometry completely in the case of matrix pencils, that is $\dim L=2$. The resulting Schubert surfaces are singular analogues of rational surfaces obtained as blow-ups of $\bbP^2$. We show that they are normal with quotient singularities, compute their degree, cohomology ring, and describe the distinguished rational curves and their intersection pairing in terms of the Kronecker invariants of the matrix pencil. Remarkably, this dictionary can be reversed: the invariants of the pencil are read off from intersection numbers and other invariants of the variety. In particular, we prove that the Schubert surface determines the regular component of the pencil, up to the natural action of $\mathrm{GL}_2 \times \mathrm{GL}_n \times\mathrm{GL}_m$. Finally, we prove that the ideal of Schubert surfaces of matrix pencils in the Pl\"{u}cker embedding is generated in degree $2$ and satisfies Green's property $N_p$ in a certain range.
\end{abstract}

\maketitle

\tableofcontents

\section{Introduction}

A recurring theme in algebraic geometry is that different compactifications of interesting spaces retain and highlight different kinds of information about the space from which they are constructed. The points added \emph{at infinity} arise from degenerations of the original points of the space: this boundary, its stratification, its singularities, and the intersection theory of its components often reflect the structure of interesting problems taking place in the original space. Schubert calculus is the most classical and successful manifestation of this principle.

The Grassmannian $\Gr(k,N)$ of $k$-dimensional subspaces in $\bbC^N$ is covered by affine spaces of matrices. Incidence conditions with a fixed complete flag in $\bbC^N$ impose a prescribed pattern on these matrices, and the closures of the resulting affine loci are the classical Schubert varieties. Several aspects of this construction, including the multiplication of the classes in the cohomology of the ambient Grassmannian are governed by algebraic combinatorics, and cohomological identities translate combinatorial problems into enumerative incidence problems among Schubert varieties. Even more intrinsically, any individual Schubert variety is stratified by smaller Schubert cells, its singularities are constrained by the corresponding Bruhat interval, and its intersection cohomology is governed by Kazhdan--Lusztig theory.

An analogous picture arises from matroids. A $\bbC$-representable matroid of rank (at most) $n$ is encoded by the combinatorics of the intersection of a linear space $L \subseteq \bbC^n$ with the coordinate subspaces of $\bbC^n$. The compactification $\bbC^n \hookrightarrow (\bbP^1)^{\times n}$ makes $L$ into a subvariety $Y_L^{\mathrm{mat}}$ in $(\bbP^1)^{\times n}$, known as the \emph{matroid Schubert variety} of $L$. In \cite{ardila2016closure} it was shown that much of the geometry of $Y_L$ is controlled combinatorially by the matroid. On the other hand, cohomology and singularities of matroid Schubert varieties are central to some of the principal applications of geometry to matroid theory. They are part of the broader Hodge theory of matroids initiated in \cite{HuhKatz12}, and culminating in \cite{AHK18,BHMPW20}: this theory was used to prove the log-concavity of the characteristic polynomial and of the numbers of independent sets of matroids, and led to a Kazhdan--Lusztig theory for arbitrary matroids. 

The overall message that we take from these works is that, once a combinatorial invariant is realized as an intersection number in some suitable algebraic variety, one can use positivity, duality, and other tools from algebraic geometry to study it. It is natural to further generalise this picture, to include more general invariants in other moduli spaces. This idea is at the foundation of the ``Bodensee program'' \cite{MicBodensee} and leads to the introduction of a generalisation of Schubert varieties: for a given linear space of matrices $L \subseteq \Mat_{n \times m}$, the \textit{Schubert compactification}, or \textit{Schubert variety}, of $L$ is a projective variety $Y_L \subseteq \Gr(n,n+m)$ compactifying an open embedding of $L$: we refer to \Cref{sec: preliminaries and motivations} for the precise definition. In this perspective, classical Schubert varieties are the Schubert varieties of a linear space $L$ defined by a pattern of zero entries in the matrices, governed by intersection conditions with a fixed flag. At the same time, matroid Schubert varieties are the Schubert varieties of a linear space $L$ of diagonal matrices. Many invariants of the linear space $L$ can be encoded into the geometry of $Y_L$.

In this work, we study Schubert varieties of linear matrix spaces of dimension $2$, classically called \textit{matrix pencils}. Matrix pencils form one of the few geometric settings in which a complete orbit classification is available: they are completely classified by \emph{Kronecker invariants} under the action of left-right multiplication of two copies of the general linear group. This classification gives sufficiently explicit control of the associated compactification to make a complete description possible. The associated Schubert surfaces are singular generalisations of Del Pezzo surfaces, and their geometry is already rich enough to exhibit interesting singularities, resolutions, and nontrivial intersection theory, while remaining sufficiently concrete to be described explicitly.

\begin{thmx}\label{thm: main forward}
Let $L \subseteq \Mat_{n \times m}$ be a two-dimensional linear space of matrices, and let $Y_L \subseteq \Gr(n, n+m)$ be its Schubert variety. Then:
\begin{enumerate}
\item $Y_L$ is normal with quotient singularities (see \Cref{thm: rational singularities});
\item the degree of $Y_L$ in the Pl\"{u}cker embedding is computed in \Cref{theorem: degree of YL};
\item the Schubert class of $\YL$ in the Chow ring of $\Gr(n,n+m)$ is computed in \Cref{thm: schubert class};
\item the intersection cohomology ring of $Y_L$ is completely described in \Cref{thm: intersections general};
\item the ideal of $Y_L$ in the Pl\"{u}cker embedding is generated in degree $2$ (see \Cref{thm: quadratic generation});
\item the ideal of $Y_L$ satisfies Green's property $N_p$ in the range given in \Cref{corol: Np property}.
\end{enumerate} 
\end{thmx}
\Cref{thm: main forward} shows that the geometry of Schubert varieties of matrix pencils is extremely well-behaved, similarly to higher dimensional classical Schubert varieties and matroid Schubert varieties. However, several of these good geometric properties do not persist for higher dimensional Schubert varieties of linear matrix spaces. For instance, in general they are neither normal nor linearly normal (see \Cref{example: counterexample normality} and \Cref{example: counterexample linear normality}), nor cut out by quadrics, even set-theoretically (see \Cref{example: counterexample quadratic ideal}). Thus matrix pencils provide \textit{both} the first case in which the geometry can be worked out completely \textit{and} a distinguished class whose behaviour is substantially better than one might expect in general.

Since the orbit-classification for matrix pencils in terms of Kronecker invariants is complete, we address the \textit{inverse problem} of recovering the invariants of the pencil $L$ from the geometry of the Schubert variety $\YL$.
More generally, linear matrix spaces of dimension $a$ may be regarded as tensors in $\bbC^a \otimes \bbC^n \otimes \bbC^m$: the linear space is identified with the image of the linear map $M: {\bbC^a}^* \to   \bbC^n \otimes \bbC^m$. The natural action of $\GL_a\times\GL_n \times\GL_m$ on the tensor space defines a notion of equivalence of tensors, and a great deal of research is devoted to understanding invariants separating interesting orbits and their closures: see, e.g., \cite{Ng95,LanOtt13,BCLNWZ25} for some methods. 

It is therefore natural to ask to what extent the Schubert variety $Y_L$ determines the linear space $L$. Equivalent tensors under the action of $\GL_a \times \GL_n \times \GL_m$ give rise to linearly isomorphic Schubert varieties: the $\GL_a$-factor changes the parameterisation of $L$, whereas the $\GL_n \times \GL_m$-action induces the corresponding linear action on the ambient Grassmannian. The converse is less evident, since passing from $L$ to the closure $Y_L$ may forget part of the linear-algebraic data of $L$.

Analogous reconstruction questions have natural answers in the classical
setting and have been studied in the matroid setting. For classical Schubert varieties, the cohomology class in the Grassmannian determines the indexing partition and hence the corresponding zero pattern, up to a natural change of coordinates. For matroid Schubert varieties, the answer is more subtle: \cite{ardila2016closure} showed that the multidegree of the embedded variety in $(\mathbb{P}^1)^{\times n}$ records the bases of the matroid and \cite{HuhWang17} related an affine stratification of the matroid Schubert variety to the lattice of flats of the matroid.

We establish an analogous recovery result for matrix pencils, that is the case $a=2$. The Kronecker invariants of matrix pencils come in two types, and identify two \emph{components} of the matrix pencil, called the \emph{regular} and \emph{singular} components. We prove that $Y_L$ uniquely determines the regular component of the pencil. The reconstruction is built on geometric and intersection theoretic properties in $Y_L$.

\begin{thmx}\label{thm: reconstruction of invariants}
Let $L_1, L_2$ be two matrix pencils of the same size. Let $Y_{L_1}, Y_{L_2} \subseteq \Gr(n, n+m) \subseteq \bbP \Wedge^{n} (\bbC^{n} \oplus \bbC^m)$ be the associated Schubert varieties, in the Pl\"ucker embedding of $\Gr(n,n+m)$. If $Y_{L_1}$ and $Y_{L_2}$ are projectively equivalent, then the regular components of $L_1$ and $L_2$ are the same up to the action of $\GL_2 \times \GL_n \times \GL_m$.
\end{thmx}

In fact, the proof of \Cref{thm: reconstruction of invariants} provides an algorithmic procedure (see \Cref{fig: reconstruction-flowchart}) to reconstruct the regular Kronecker invariants of the pencil from the geometry of the associated Schubert variety. We will highlight that certain information regarding the singular invariants can be recovered as well, although there exist pencils with distinct singular invariants but the same associated Schubert variety.

\subsection{Context}  

The study of compactifications of spaces of matrices dates back to the nineteenth century. Algebraic geometers such as Chasles \cite{chasles1864determination}, Giambelli \cite{giambelli1903problema}, and Schubert \cite{schubert1979kalkul} introduced compact parameter spaces in which limits of geometric objects could be used to formulate and solve classical enumerative problems; some of the most celebrated examples involve the enumeration of quadrics tangent to a prescribed collection of linear spaces. These questions led to the introduction of the spaces of \emph{complete quadrics} $\CQ_n$, compactifying the space of full-rank symmetric matrices of size $n$, and the space of \emph{complete collineations} $\CCol_n$, compactifying the space of full-rank matrices of size $n \times n$. Initially, these spaces were described rather informally as parameter spaces containing all possible degenerations of the corresponding linear maps. A systematic study of their geometry was undertaken later in \cite{severi1940fondamenti,semple1952complete}. Their modern construction appears in \cite{vainsencher1982schubert}, in terms of iterative blow-ups
\[
    \CQ_n \longrightarrow \mathbb{P}(S^2\mathbb{C}^n)
    \qquad\text{and}\qquad
    \CCol_n \longrightarrow \mathbb{P}(\Mat_n)
\]
of spaces of matrices along the (strict transforms of) determinantal loci of bounded rank.

In \cite{deconcini1982complete}, the variety of complete quadrics was reinterpreted as the \emph{wonderful compactification} of the symmetric space $\SL_n/\SO_n$, whereas \cite{thaddeus1998complete} described the variety of complete collineations as the \emph{Chow quotient} of $\Gr(n,2n)$ by a natural $\mathbb{C}^{\times}$-action. An alternative compactification of $\GL_n$ was introduced in \cite{kausz2000modular}. It is obtained by iteratively blowing up $\mathbb{P}(\mathbb{C}\oplus\Mat_n)$ along the closures of the determinantal varieties of bounded rank contained in the affine chart $\Mat_n$. More recently, \cite{fang2025canonical} related this compactification to the space of complete collineations through canonical blow-ups of Grassmannians associated with a torus action. This approach has also been extended to Lagrangian and orthogonal Grassmannians in \cite{fangmassarenti2026canonical}. 

Schubert varieties of linear matrix spaces belong to the same circle of ideas, but shift the emphasis from compactifying an entire space of matrices or a rank stratum to compactifying a prescribed linear subspace of matrices inside a Grassmannian. From this perspective, one may ask how the linear-algebraic and orbit-theoretic invariants of the original subspace are reflected in the geometry of its compactification.

A parallel development comes from matroid theory. In \cite{ardila2016closure}, matroid Schubert varieties were introduced by compactifying an affine linear subspace $L\subseteq \bbC^n$ in the coordinate-wise compactification $\mathbb{C}^n \hookrightarrow (\mathbb{P}^1)^{\times n}$. When $L$ represents a matroid $M$, the geometry of the matroid Schubert variety $Y_L^{\mathrm{mat}}$ is controlled to a remarkable extent by the combinatorics of $M$. Its defining equations, multidegree, and multigraded Betti numbers are determined by the matroid; moreover, its multidegree records the bases of $M$ and hence determines the labelled matroid \cite{ardila2016closure}. We refer to \Cref{subsec: matroid schubert varieties} for details on this construction. The variety $Y_L^{\mathrm{mat}}$ admits an affine stratification indexed by the flats of the matroid: the resulting distinguished basis in cohomology and its combinatorics was used in \cite{HuhWang17} to settle the ``top-heavy conjecture'' from \cite{DW74}. In \cite{EPW16}, the authors introduce Kazhdan--Lusztig polynomials of matroids and connect their behaviour to the intersection cohomology of certain varieties related to the matroid Schubert varieties. In \cite{stellahedral}, this construction was placed in the broader geometry of the stellahedral toric variety $\mathrm{Stell}_n$. This variety can be obtained from $(\mathbb{P}^1)^{\times n}$ by a sequence of torus-equivariant blow-ups, and the closure of a realisation of a matroid lifts to its augmented wonderful variety. This work identifies the valuative group of matroids with the cohomology ring of $\mathrm{Stell}_n$ and shows that valuative, homological, and numerical equivalence for matroids coincide. Thus the stellahedral variety provides a single intersection-theoretic framework in which a broad collection of matroid invariants can be studied.

These developments all sit in a more general framework, whose aim is to associate projective varieties to increasingly general classes of linear matrix spaces and tensors, extending classical Schubert calculus, discrete geometry of matroids and algebraic combinatorics \cite{MicBodensee}. Schubert varieties of linear matrix spaces form one of the basic constructions in this unifying program. The present work provides a first detailed test of this philosophy. In the case of matrix pencils, we show that the resulting compactifications are sufficiently explicit to admit a precise geometric analysis, while at the same time retaining substantial information about the original object. Our results therefore give initial evidence that Schubert varieties can serve as effective geometric tools for the study of linear matrix spaces and, more generally, of tensors.

\subsection{Applications.} 
A central problem in the study of tensors is to construct invariants that distinguish tensors up to change of coordinates. In contrast with the case of matrices, where the only invariant under left-right multiplication is the rank of the matrix, tensors do not admit a finite classification by normal forms, except in very few formats. It is therefore natural to study numerical, algebraic, or geometric invariants that retain meaningful information about the orbit structure in the tensor spaces.

Tensor rank and border rank are among the most investigated invariants of this form. They were studied already in \cite{Sylvester,Segre20}, and play an important role in areas ranging from algebraic complexity theory to statistics and quantum information. At the same time, there are strong hardness results for computing tensor rank and other invariants of the same form \cite{NP-hardTensors}. This provides one motivation for constructing auxiliary invariants that can either distinguish tensor orbits or yield obstructions to low rank. More generally, such invariants may be used to organise tensor spaces into finite geometrically interesting strata, even when a complete orbit classification is unavailable.

One approach is to associate intersection-theoretic invariants to linear matrix spaces arising from tensors. A tensor $T \in \bbC^a \otimes \bbC^n \otimes \bbC^n $ is identified with an $a$-dimensional linear space of matrices $T( {\bbC^a}^*) \subseteq \bbC^n \otimes \bbC^n$. The natural map to the space of complete collineations defines a subvariety $\CCol(T) \subseteq \CCol_n$ which retains information about the original tensor. Inspired by the corresponding construction in matroid theory, where the permutahedral variety resolves the classical Cremona map \cite{HuhKatz12}, one may define the \emph{volume polynomial} of a tensor
\[
    P_T(x_1,\ldots,x_{n-1}):= \CCol(T) \cdot \bigl(x_1H_1+\cdots+x_{n-1}H_{n-1}\bigr)^{a-1}
\]
where $H_1 \vvirg H_{n-1}$ are the distinguished classes in the cohomology of $\CCol_n$. This was introduced in \cite{conner2026characteristic}: the coefficients of $P_T$ are called the \emph{characteristic numbers} of $T$ and they form a collection of intersection-theoretic invariants. Among them is the degree of the reciprocal variety of $T$, obtained by applying matrix inversion to the invertible elements of $\mathbb{P}(T)$ and taking the closure. This construction, as well as similar ones for the complete quadrics $\CQ_n$, appears in broad applications, including invariant theory \cite{AKRS24}, matroid theory \cite{LMW25}, the theory of hypergeometric systems \cite{MT26} and algebraic statistics \cite{manivel2024complete}. 

Characteristic numbers provide computable invariants that may separate different tensor orbits and orbit closures. However, they are numerical in nature and therefore necessarily retain only part of the geometry of the tensor. The approach of \cite{GeKe:collineations} proposed replacing individual intersection numbers by the full geometric information of the projective varieties from which they arise. This motivated the introduction of the $k$-th collineation variety of a tensor that is the closure of the image of $\mathbb{P}(T({\bbC^a}^*))$ under the rational map defined by the $k\times k$ minors, or equivalently the projection of $\CCol(T)$ to the $k$-th factor. The dimensions, degrees, equations, and geometric features of these varieties provide invariants of tensor orbits and orbit closures and induce corresponding stratifications of tensor spaces. It was observed that collineation varieties provide richer invariants than characteristic numbers alone. However, in some cases, and notably in the case of pencils of matrices, even collineation varieties are far from distinguishing elements within continuous families of orbits. 

The Schubert compactifications introduced in this work provide a substantially finer geometric invariant: \Cref{thm: reconstruction of invariants} shows that the regular component of a matrix pencil, which is the one depending on the continuous parameters, is uniquely determined by the Schubert surface of the pencil. This provides a first indication that the broader geometric constructions proposed in the Bodensee program \cite{MicBodensee} can produce effective and increasingly refined invariants of tensor orbits and orbit closures.

\subsection{Outline.} 

We summarise the main results of the paper in the flowchart \Cref{fig: flowchart forward}, showing how they contribute to the proof of \Cref{thm: main forward}. Similarly, \Cref{fig: reconstruction-flowchart} in \Cref{sec: recovering} outlines the reconstruction procedure yielding the proof of \Cref{thm: reconstruction of invariants}.

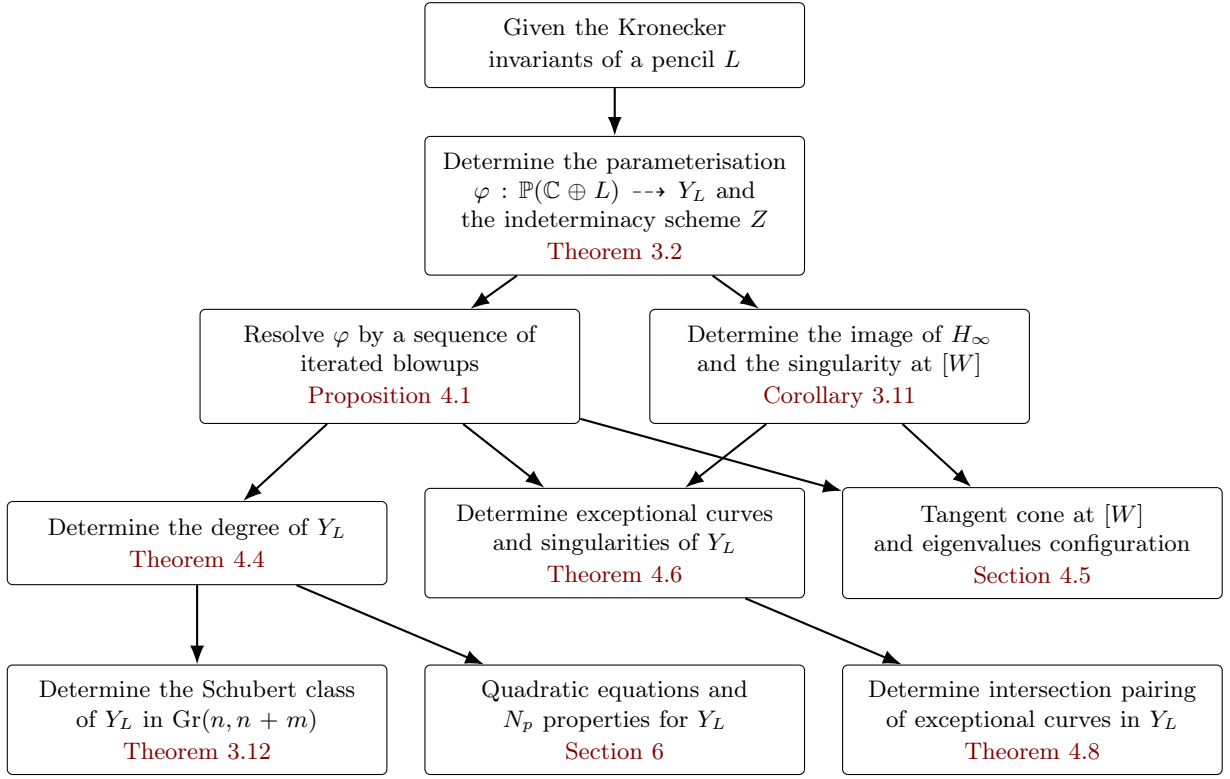
\begin{figure}[H]
\centering
\begin{tikzpicture}[
    node distance=1cm and 2cm,
    every node/.style={
        font=\small,
        align=center
    },
    box/.style={
        rectangle,
        draw,
        rounded corners=2pt,
        minimum width=5cm,
        minimum height=1.1cm,
        text width=4.6cm,
        inner sep=6pt
    },
    result/.style={
        rectangle,
        draw,
        rounded corners=2pt,
        minimum width=4cm,
        minimum height=1.2cm,
        text width=5.1cm,
        inner sep=6pt
    },
    arrow/.style={
        -{Latex[length=2.5mm]},
        thick
    },
    label/.style={
        font=\small,
        fill=white,
        inner sep=2pt
    },
    scale=0.85
]

\node[box] (start) at (0,0) {\footnotesize
    Given the Kronecker \\ invariants of a pencil $L$
};

\node[box] (parameterisation) at (0,-2.5) {\footnotesize
    Determine the parameterisation
    $\varphi:\bbP ( \bbC \oplus L) \dashrightarrow Y_L$ and \\
    the indeterminacy scheme $Z$\\
    \Cref{theorem: linear series of YL}
};

\node[box] (resolution) at (-3.5,-5) {\footnotesize
    Resolve $\phi$ by a sequence of \\
    iterated blowups \\
    \Cref{prop: resolution of phi}
};

\node[box] (infinityW) at (3.5,-5) {\footnotesize
    Determine the image of $H_\infty$ \\
    and the singularity at $[W]$ \\
    \Cref{prop: W singularity}
};

\node[box] (degree) at (-6.5,-7.75) {\footnotesize
    Determine the degree of $Y_L$ \\
    \Cref{theorem: degree of YL}
};

\node[box] (singularities) at (0,-7.75) {\footnotesize
    Determine exceptional curves \\ 
    and singularities of $Y_L$ \\
    \Cref{thm: rational singularities}
};

\node[box] (tangent cone) at (6.5,-7.75) {\footnotesize
    Tangent cone at $[W]$ \\ 
    and eigenvalues configuration \\
    \Cref{subsec: eigenvalue configuration}
};

\node[box] (Schubertclass) at (-6.5,-10.5) {\footnotesize
    Determine the Schubert class \\
    of $Y_L$ in $\Gr(n,n+m)$ \\
    \Cref{thm: schubert class}
};

\node[box] (quadrics) at (0,-10.5) {\footnotesize
    Quadratic equations and \\
    $N_p$ properties for $Y_L$ \\
    \Cref{sec: quadrics}
};

\node[box] (intersections) at (6.5,-10.5) {\footnotesize
    Determine intersection pairing \\
    of exceptional curves in $Y_L$ \\
    \Cref{thm: intersections general}
};

\draw[arrow] (start) -- (parameterisation);

\draw[arrow] (parameterisation) -- (resolution);

\draw[arrow] (parameterisation) -- (infinityW);

\draw[arrow] (resolution) -- (degree);

\draw[arrow] (resolution) -- (singularities);

\draw[arrow] (resolution) -- (tangent cone);

\draw[arrow] (infinityW) -- (singularities);

\draw[arrow] (infinityW) -- (tangent cone);

\draw[arrow] (degree) -- (Schubertclass);

\draw[arrow] (degree) -- (quadrics);

\draw[arrow] (singularities) -- (intersections);

\end{tikzpicture}
\caption{Flowchart of the results and the contribution to \Cref{thm: main forward}.}
\label{fig: flowchart forward}
\end{figure}

\section{Preliminaries and motivations}\label{sec: preliminaries and motivations}
In this section, we introduce \emph{Schubert varieties of linear matrix spaces} as compactifications of linear matrix spaces in suitable Grassmannians. We show that this framework indeed generalises both classical Schubert varieties and matroid Schubert varieties.

Let $V,W$ be two vector spaces with $n = \dim V$ and $m = \dim W$. Let $\Gr(n,V \oplus W)$ be the Grassmannian of $n$-dimensional linear subspaces of $V \oplus W$. Denote by $[E]$ the point of $\Gr(n,V\oplus W)$ defined by the subspace $E \subset V\oplus W$. The tensor space $V^\ast \otimes W = \Hom(V,W)$ is identified with the affine open subset in $\Gr(n,V\oplus W)$
\[ 
\mathcal{U}_W := \left\{ [E] \in \Gr(n,V\oplus W) \ | \ E \cap W = 0 \right\} \ 
\] 
via the map
\begin{equation} \label{eqn: open embedding}
\begin{matrix}
\phi : & V^* \otimes W & \to & \calU_W \\
&X &\mapsto & [\Gamma_X]
\end{matrix} 
\end{equation}
where  $\Gamma_X := \left\{ v + Xv  \in V \oplus W | \ v \in V\right\} $ is the graph of the linear map $X : V \to W$. The map $\phi$ is an open embedding of $V^* \otimes W$ into the Grassmannian. In coordinates, fixing bases of $V$ and $W$, one may identify $V^* \otimes W$ with the space of matrices $\Mat_{n \times m}$ of $n \times m$ matrices. In this way $V^* \otimes W$ is identified with the affine chart of $\Gr(n,n+m)$ parameterised by
\begin{equation}\label{eqn: open embedding chart}
\phi: X \mapsto \big[ I_n | X \big] 
\end{equation}
where $I_n$ is the $n\times n$ identity matrix, the (transpose of the) matrix $X \in \Mat_{n \times m}$ is regarded as a map $V \to W$ in the fixed bases and $\Gamma_X$ is identified with the rowspan of $[I_n | X ]$. 

\begin{definition}\label{def: Schubert variety of matrix space}
Let $L\subseteq V^\ast \otimes W$ be a linear subspace. The {\em Schubert variety of} $L$ is the projective variety 
\[ 
\YL \coloneqq \overline{\YL^\circ} \subset \Gr(n,V\oplus W) , 
\]
where $\YL^\circ$ is the image of $L$ in $\Gr(n,V\oplus W)$ via the open embedding $\phi$.
\end{definition}
Since $\phi$ is an open embedding, $\dim \YL = \dim L$ and $\YL$ is a compactification of $L$. In principle, one can define Schubert varieties of linear matrix spaces with respect to any direct sum decomposition of $V \oplus W$.

\begin{remark}
One way to define classical Schubert varieties in a Grassmannian $\Gr(n,N)$ is to fix a complete flag of $\mathbb C^N$ and consider points of $\Gr(n,N)$ which intersect the subspaces of the flag with certain fixed dimensions. Then, all classical Schubert varieties in $\Gr(n,N)$ are obtained by moving the complete flag via the action of $\GL_N$. Similarly, \Cref{def: Schubert variety of matrix space} depends on the direct sum decomposition $V\oplus W$. More generally, one can define Schubert varieties of linear spaces in $\Gr(n,n+m)$ to be those subvarieties $Y$ for which there exists a direct sum decomposition $V \oplus W = \bbC^{n+m}$ and a linear space $L \subseteq V^* \otimes W$ such that $Y = \YL$. 
\end{remark}
The Schubert varieties of linear matrix spaces are a generalisation of the classical and the matroid Schubert varieties. Several other classical constructions in geometry arise as Schubert varieties of linear matrix spaces, see \Cref{subsec:examples}.

\subsection{Classical Schubert varieties}

To see that classical Schubert varieties are special cases of Schubert varieties of linear matrix spaces, we choose, for a given partition $\lambda$, a flag $F_\bullet$ of $V \oplus W$ and verify that the Schubert variety $\Sigma_\lambda$ associated to $\lambda$ coincides with the Schubert variety $Y_L$ of a suitable $L \subseteq \Hom(V,W)$. We present the construction explicitly, following \cite[Sec. 9.4]{FultonYT}.

Let $\lambda = (\lambda_1 \geq \cdots \geq \lambda_n)$ be a partition contained in the $n \times m$ box and define $p_i = m+i-\lambda_i$. Let $v_1 \vvirg v_n$ be a basis of $V$ and $w_1 \vvirg w_m$ a basis of $W$. Fix an ordered basis $u_1 \vvirg u_{n+m}$ of $V \oplus W$ by choosing $u_{p_i} = v_i$ and the remaining $u_j$ to be the vectors $w_m \vvirg w_1$. Let $F_\bullet$ be the flag of $V \oplus W$ with $F_k = \langle u_1 \vvirg u_k\rangle$. In particular, 
\[
F_{p_i} = \langle v_1 \vvirg v_i \rangle + \langle w_m \vvirg w_{\lambda_i + 1}\rangle.
\]
The Schubert variety $\Sigma_\lambda$, with respect to this fixed flag, is given by 
\[
\Sigma_\lambda = \{ [E]\in \Gr(n,V\oplus W) : \dim(E \cap F_{p_i}) \geq i \text{ for $i =1 \vvirg n$}\}.
\]
On the affine chart $\calU_W$, every linear space $E$ is identified with the graph $\Gamma_X$ of a linear map $X : V \to W$. One can verify 
\[
\dim (\Gamma_X \cap F_{p_i}) = \dim \{ v \in  \langle v_1 \vvirg v_i\rangle : Xv \in \langle w_m \vvirg w_{\lambda_i + 1}\rangle \},
\]
therefore $\Gamma_X \in \Sigma_\lambda$ if and only if $X ( \langle v_1 \vvirg v_i\rangle) \subseteq \langle w_m \vvirg w_{\lambda_i + 1}\rangle$. In coordinates, this means that $X$ belongs to the linear space $L = \{ X \in \Mat_{n \times m} : X_{ij} = 0 \text{ for $j > \lambda_i$}\}$, described by a zero pattern with the same shape as the partition $\lambda$.  In turn, this shows $\Sigma_\lambda = Y_L$.

\begin{example}
Let $ n = m = 3$ and $(\lambda_1,\lambda_2,\lambda_3)=(3,1,0)$, so that $(p_1,p_2,p_3) =(1,4,6)$, the ordered basis of $V \oplus W$ is $(v_1, w_3, w_2, v_2, w_1,v_3)$ and the relevant components of the flag described above are
\[
F_1 = \langle v_1\rangle, \quad F_4 = F_1 + \langle w_3,  w_2,  v_2 \rangle, \quad F_6 = V \oplus W.
\]
For a linear space $\Gamma_X \in \Gr(3, 6)$ representing the graph of a linear map $X : V \to W$, the Schubert conditions are 
\begin{align*}
\dim(\Gamma_X \cap F_1) \geq 1 & \quad \Rightarrow \quad Xv_1 = 0, \\
\dim(\Gamma_X \cap F_4) \geq 2 & \quad \Rightarrow \quad X\langle v_1,v_2 \rangle \subseteq \langle w_2,w_3\rangle.
\end{align*}
These conditions define a linear subspace $L \subseteq \Mat_{n \times n} \simeq \Hom(V,W)$. Recall that in coordinates the map $X : V \to W$ is given by the transpose of the matrix on the right hand side of \eqref{eqn: open embedding chart}. Therefore, the open embedding \eqref{eqn: open embedding} of $L$ is given by the rowspan of the matrix
{\small \[
\left[\begin{array}{ccc|ccc}
1 &0 & 0  & 0 & 0 & 0 \\
0 & 1 & 0  & 0 & * & * \\
0 & 0 & 1  & * & * & *
\end{array} \right].
\]}

The zero pattern on the right-hand side has the shape of the Young diagram of the partition $\lambda = (3,1,0)$. The Schubert variety $Y_L$ of the linear space $L$ coincides with the Schubert variety $\Sigma_\lambda$ defined by the partition $\lambda$ and the fixed choice of flag.
\end{example}

\subsection{Matroid Schubert varieties}\label{subsec: matroid schubert varieties}

Matroid Schubert varieties were introduced in \cite{ardila2016closure} as compactifications of linear spaces arising from realisations of matroids. We show that they are special cases of Schubert varieties of linear matrix spaces as well. The vector space $\bbC^n$ has a compactification 
\[ \begin{matrix}
\phi : & \bbC^n & \to & (\bbP^1)^{\times n} \\
    & (x_1,\ldots,x_n) & \mapsto & ([1:x_1],\ldots,[1:x_n]) .
\end{matrix} \]
The \emph{matroid Schubert variety} associated to a linear subspace $L \subseteq \bbC^n$ is the closure
\[
Y_L^{\mathrm{mat}} = \bar{\phi (L)} \subseteq (\bbP^1)^{\times n}.
\]
When $L$ is a realisation of a representable matroid on $[n]$ elements, the intersections of $L$ with the coordinate subspaces of $\bbC^n$, that is the combinatorial structure of the matroid, can be recovered from intersection-theoretic properties of $Y_L^{\mathrm{mat}}$; in turn this provides a geometric framework to study the matroid \cite{ardila2016closure}.

We explain explicitly how this construction is a special case of \Cref{def: Schubert variety of matrix space}. Set
\[
    V= \bbC^n=\langle v_1,\ldots,v_n\rangle,
    \qquad
    W=\bbC^n=\langle w_1,\ldots,w_n\rangle,
\]
and identify $\mathbb C^n$ with the linear subspace
\[
\Diag_n= \left\{ {\scriptsize \begin{pmatrix}
x_1 & &  \\
& \ddots &  \\
&  & x_n
\end{pmatrix} } : (x_1 \vvirg x_n) \in \mathbb C^n \right\} \subseteq V^*\otimes W
\]
of diagonal linear maps in the fixed bases. The Schubert variety $Y_{\Diag_n}$ coincides with a Segre embedding of $(\mathbb P^1)^{\times n}$ into $\Gr(n,V\oplus W) \subseteq \bbP \Wedge^n (V \oplus W)$. Indeed, the open embedding $\phi : \Diag_n \to \Gr(n,V \oplus W)$ extends to a closed embedding 
\[ \begin{matrix}
\iota : & (\mathbb P^1)^{\times n} & \to & \Gr(n,V\oplus W) \\
     & \bigl(\left[\begin{smallmatrix}a_1 \\ b_1 \end{smallmatrix} \right] \vvirg \left[\begin{smallmatrix}a_n \\ b_n \end{smallmatrix} \right]  \bigr)
     & \mapsto &
    {\scriptsize \left[
    \begin{array}{ccc|ccc}
    a_1 &     & & b_1 &  \\
       & \ddots & &  & \ddots &  \\
        &      & a_n &  &     & b_n
    \end{array}
    \right]}
\end{matrix} \]
 defined by the maximal minors of an $n \times 2n$ matrix. On the open set $a_1 \cdots a_n \neq 0$ the map $\iota$ coincides with $\phi$, hence its image is the closure of $\phi(\Diag_n)$, that is $Y_{\Diag_n}$. Thus, after identifying $(\mathbb P^1)^{\times n}$ with its image under $\iota$, matroid Schubert varieties are exactly the Schubert varieties associated to linear spaces of diagonal matrices.

For a representable matroid $M$ on $[n]$ let $L \subseteq \bbC^n$ be a representation of $M$ in the following sense. Let $z_i : L\longrightarrow\mathbb C$ be the coordinate functions of $\bbC^n$ restricted to $L$. Then $z_{i_1} \vvirg z_{i_s}$ are linearly independent as elements of $L^*$ if and only if $\{i_1,\ldots,i_s\}$ is an independent set of $M$.

The combinatorics of the matroid is visible in the equations of $Y_L$. For example, let $c_1x_{i_1}+\cdots+c_sx_{i_s}=0$ be a linear equation of $L$. On the chart $a_i \neq 0$, this defines an equation whose homogenisation is 
\[ \textstyle 
    \sum_{k \in [s]}
    c_k a_{i_1}\cdots a_{i_{k-1}} b_{i_k} a_{i_{k+1}}\cdots a_{i_s} = 0.
\]
In particular, the minimal-support linear relations among the coordinate functions give distinguished equations for $Y_L \subseteq (\bbP^1)^{\times n}$. One of the main results of \cite{ardila2016closure} is that these equations control the whole defining ideal, and in fact form a universal Gr\"obner basis.

The bases of the matroid are recorded by the multidegree. More precisely, one has
\[
    [Y_L^{\mathrm{mat}}]    =
    \sum_{B\in\mathcal B(M)} \prod_{i\notin B} h_i
   \qquad\text{in}\qquad
    \CH^\bullet((\mathbb P^1)^{\times n})
    =
    \frac{\mathbb Z[h_1,\ldots,h_n]}
         {(h_1^2 \vvirg h_n^2)}.
\]
Thus the multidegree determines the set of bases of $M(L)$ and hence the labelled matroid. This is one of the simplest instances of the general principle motivating Schubert varieties of linear matrix spaces: information about the original linear space is converted into intersection-theoretic information on its compactification.

\begin{example}\label{example: matroid Schubert}
Consider the matroid $M$ of rank $2$ on $[4]$ whose set of bases is
\[
\mathcal{B}(M) = \bigl\{
\{1,2\},\{1,3\},\{1,4\},\{2,3\},\{2,4\}
\bigr\}.
\]
Thus the only $2$-subset which is not a basis is $\{3,4\}$. A linear subspace of $\bbC^4$ defining a realisation of $M$ is 
\[
L =
\bigl\{ (x_1,x_2,x_1+x_2,x_1+x_2)
: x_1,x_2\in\mathbb{C} \bigr\}
\subseteq \mathbb{C}^4.
\]
The circuits of $M$ define the minimal linear relations
\[
z_3-z_4=0,
\qquad
z_1+z_2-z_3=0,
\qquad
z_1+z_2-z_4=0.
\]
The open embedding of $L$ into $(\mathbb{P}^1)^{\times 4}$ defines the matroid Schubert variety $Y^{\mathrm{mat}}_L$. The circuits give the equations 
\begin{align*}
F_{34} &= b_3a_4-a_3b_4,\\
F_{123} &= a_1a_2b_3 - b_1a_2a_3 -  a_1b_2a_3 ,\\
F_{124} &= a_1a_2b_4 - b_1a_2a_4 -  a_1b_2a_4.
\end{align*}
The bases of $M$ are encoded in the multidegree of $Y^{\mathrm{mat}}_L$. The class of $Y^{\mathrm{mat}}_L$ in $\CH^\bullet((\mathbb{P}^1)^{\times 4})$ is
\[ \textstyle 
[Y_L^{\mathrm{mat}}] = \sum_{B\in\mathcal{B}(M)}
\prod_{i\notin B}h_i = h_3h_4 + h_2h_4 + h_2h_3 + h_1h_4 + h_1h_3.
\]
\end{example}

\subsection{Pencils of matrices and Kronecker invariants}\label{subsec: pencils}

 A pencil of matrices is a $2$-dimensional linear subspace of matrices $L \subseteq V^* \otimes W$ that can be identified with a tensor $T \in A \otimes V^* \otimes W$ with $\dim A = 2$ via tensor flattening. Let $x_1,x_2$ be a basis of $A$, so that the linear map $T : A^* \to V^* \otimes W$ given by the contraction against the factor $A$ parameterises $L$ as an $n \times m$ matrix of linear forms in $x_1,x_2$.

It is a classical fact due to Weierstrass and Kronecker \cite{Weis67,Kron90} that pencils of matrices have normal forms: in other words, up to the action of $\GL(V) \times \GL(W)$ on $A \otimes V^* \otimes W$, a pencil of matrices is uniquely determined by a finite set of discrete and continuous invariants, today known as Kronecker invariants. We refer to \cite[Ch.~13]{Gant59} for the theory and we record here the result.

There are three types of \emph{indecomposable} pencils of matrices: the left and the right singular pencils
\[
S^{\rmleft}_{p} = {\footnotesize 
\left[\begin{array}{ccccc}
        x_1 & x_2 \\ 
         & \ddots & \ddots \\
         & & x_1 & x_2
        \end{array} \right]_{p\times (p+1)} }, \qquad  
S^{\rmright}_{p} = {\footnotesize
\left[{\begin{array}{cccc}
        x_1 & \\
        x_2 & \ddots  & \\
        & \ddots & x_1 \\
        & &  x_2       \end{array}}\right]_{(p+1)\times p} }
\]
and the Jordan pencils with generalised eigenvalue $\theta = \theta^{(0)}x_1  + \theta^{(1)} x_2 \in A$
\[ 
    J_p(\theta) = {\footnotesize
    \left[{\begin{array}{cccc}
        \theta & x_2 \\
        & \ddots & \ddots \\
        & & \ddots & x_2  \\
        & &  & \theta 
        \end{array}}\right]_{p\times p} }.
        \]
The \emph{block sum} of two pencils $T_1 \in A \otimes \bbC^{n_1} \otimes \bbC^{m_1}$ and $T_2 \in A \otimes \bbC^{n_2} \otimes \bbC^{m_2}$ is the pencil obtained by projecting the direct sum $T_1 \oplus T_2 \in (A^{\oplus 2}) \otimes \bbC^{n_1 + n_2} \otimes \bbC^{m_1 + m_2}$ on the diagonal $A \subseteq A^{\oplus 2}$; in terms of matrices of linear forms, the block sum is 
\[
T_1 \boxplus T_2 = {\footnotesize \begin{bmatrix}
T_1 & 0 \\
0 & T_2
\end{bmatrix} } \in A \otimes (\bbC^{n_1} \oplus \bbC^{n_2})  \otimes (\bbC^{m_1} \oplus \bbC^{m_2}).
\]
The structural result of \cite{Weis67,Kron90} is that every pencil of matrices $T \in A \otimes V^* \otimes W$ is, up to the action of $\GL(V) \times \GL(W)$, a block sum of indecomposable pencils; the sizes of the blocks and the generalised eigenvalues of the Jordan blocks are uniquely determined. One can further partially normalise the generalised eigenvalues using the action of $\GL(A)$ on $A$ \cite{JaJa79}, but we will not employ this normalisation. We point out, however, that if there is a single generalised eigenvalue then such eigenvalue can be assumed to be $\theta = x_1$. Moreover, note that if $L$ is block sum of copies of $J_1(x_1)$ then it does not define a pencil, but rather the single linear dimensional space of multiples of the identity.

We say that a pencil is regular if its decomposition into indecomposable pencils contains no singular blocks. We say that it is singular if it contains singular blocks. We call \textit{left} (resp. \textit{right}) \textit{singular component} the block sum of all the left (resp. right) singular blocks of the pencil. We call \textit{regular component} the block sum of all the Jordan blocks.

Let $L \subseteq V^* \otimes W$ be a pencil. The \emph{generic rank} of $L$ is the rank of its generic point as a linear map $V \to W$. There are finitely many points $X_1,\ldots , X_s\in \mathbb PL$ having rank strictly lower than the generic rank of $L$, which we call \emph{rank degenerate} points. These are given precisely by the roots of the generalised eigenvalues, hence they only depend on the regular part of the pencil;  when $L$ is regular, they are the roots of $\det(L)$. On the other hand, when $L$ is totally singular, there are no rank-degenerate points in $L$: indeed, the singular blocks $L_p$ and $R_p$ always have maximal rank unless $x_1 = x_2 = 0$.

\subsection{Grassmannians and direct sums} \label{subsec: Grassmannians and direct sums}

We give a brief overview of a standard construction which will be useful to \Cref{subsec: restriction of phi to line at infinity} and \Cref{subsec: schubert classes}. The Kronecker--Weierstrass decomposition of a pencil of matrices regards the pencil as a linear subspace of block diagonal matrices. More precisely, if
$L\subseteq V^*\otimes W$, then there are direct sum decompositions
\[
V=V_1\oplus\cdots\oplus V_r,
\qquad
W=W_1\oplus\cdots\oplus W_r,
\]
such that
\[
L\subseteq
V_1^*\otimes W_1\oplus\cdots\oplus V_r^*\otimes W_r
\subseteq V^*\otimes W
\]
with one indecomposable Kronecker block in each $V_i^* \otimes W_i$. The corresponding decomposition on Grassmannians will allow us to study several geometric features of $Y_L$ by reducing the analysis to indecomposable blocks.

Let $U$ be a vector space with a direct sum decomposition
\[
U=U_1\oplus\cdots\oplus U_r.
\]
Fix integers $q_1,\ldots,q_r$ with $q_i\leq \dim U_i$, and set $q_\bullet=q_1+\ldots+q_r$. Let $\Gr_\bullet := \Gr(q_1 , U_1) \times \ldots \times \Gr(q_r, U_r)$. There is a natural direct sum map
\begin{equation}\label{eqn: direct sum grassmannian}
\begin{matrix}
\Delta: & \Gr_\bullet & \to & \Gr(q_\bullet, U)  \\
& (E_1 \vvirg E_r) & \mapsto & E_1 \oplus \ldots \oplus E_r.
\end{matrix}
\end{equation}
The ambient space $\Wedge^{q_\bullet} U$ of the Pl\"ucker embedding of $\Gr(q_\bullet,U)$ has a direct sum decomposition in terms of the multidegrees with respect to each $U_j$; one of the summands is $\Wedge^{q_1}U_1 \ootimes \Wedge^{q_r} U_r$. The map $\Delta$ is the restriction to $\Gr_\bullet$ of the Segre embedding 
\[
\bbP \Wedge^{q_1}U_1 \times \ldots \times \bbP \Wedge^{q_r} U_r \to \bbP \Wedge^{q_\bullet} U.
\]
Let $\calS$ be the tautological bundle on $\Gr(q_\bullet,U)$, $\calS_i$ be the tautological bundle on $\Gr(q_i,U_i)$, and let $\pi_i : \Gr_\bullet \to \Gr(q_i, U)$ be the projection onto the $i$-th factor.

\begin{lemma}\label{prop: grassmannian direct sum}
The map $\Delta$ of \eqref{eqn: direct sum grassmannian} is a closed embedding. The pullback of the tautological bundle splits as
    \[
    \Delta^*\calS \simeq  \textstyle  \bigoplus_{i=1}^r\pi_i^*\calS_i.
    \]
\end{lemma}
\begin{proof}
    The map $\Delta$ is a closed embedding because it is the restriction of the Segre embedding to $\Gr_\bullet$. Let $(E_1,\ldots,E_r) \in \Gr_\bullet$ and let $E = E_1 \ooplus E_r \subseteq U$. By definition 
\[
\calS_E= (\calS_1)_{E_1} \ooplus (\calS_r)_{E_r}
\]
which immediately shows $\Delta^*\calS \simeq   \bigoplus_{i=1}^r \pi_i^*\calS_i$. 
\end{proof}

We record a useful consequence of \Cref{prop: grassmannian direct sum}.
\begin{corollary}\label{cor: sum of rnc}
Let $\gamma_i: \bbP^1 \to \Gr(q_i ,U_i)$ be the embedding of a rational normal curve $C_i$ of degree $\delta_i$ in the Pl\"ucker embedding of $\Gr(q_i, U_i)$. Then 
\[
\bbP^1 \xto{\gamma_1 \ttimes \gamma_r} \Gr(q_1 ,U_1) \ttimes \Gr(q_r, U_r) \xto{\Delta} \Gr(q_\bullet, U)
\]
parameterises a rational normal curve of degree $\sum_{i \in [r]}\delta_i$ in the Pl\"ucker embedding of $\Gr(q_\bullet, U)$.
\end{corollary}
\begin{proof}
    The Segre embedding restricted to the diagonal image of $\bbP^1 \xto{\gamma_1 \ttimes \gamma_r}\Gr(q_1 ,U_1) \ttimes \Gr(q_r, U_r)$ is the Veronese embedding of $\bbP^1$ of degree $\delta_1 + \cdots + \delta_r$. Its image is the desired rational normal curve.
\end{proof}

In the case of matrix pencils $L \subseteq V^* \otimes W$, the Kronecker-Weierstrass normal form gives a direct sum decomposition of $V$ and $W$ so that $L \subseteq (V_1^* \otimes W_1) \oplus \ldots \oplus (V_r^* \otimes W_r)$. It is immediate that the Schubert variety $Y_L$ is contained in the induced subvariety $\Gr_\bullet$. \Cref{prop: grassmannian direct sum} will guarantee that the pullback of the tautological bundle $\calS$ to $Y_L$ can be studied in terms of the sum of the pullback bundles $\calS_i$ to the Schubert varieties of the indecomposable blocks of $L$. This simple fact will be extremely useful in \Cref{subsec: schubert classes}. 

\subsection{A short review of quotient surface singularities}\label{subsec: surface singularities}
We briefly recall terminology and basics concerning surface singularities that will be used in the paper. We refer to \cite[Ch.~4]{Reid97} for details. A \emph{surface singularity} is a pair $(X,p)$ of a quasi-projective surface $X$ and a point $p \in X$: $p$ is usually a singular point of $X$, but this setting allows $p$ to be smooth as well. We consider a geometric notion of equivalence of surface singularities, slightly more restrictive than the usual one found in the literature. We say that $(X,p)$ and $(Y,q)$ are \emph{Zariski locally} isomorphic if there exist open subvarieties $X',Y'$ with $p \in X' \subseteq X$ and $q \in Y' \subseteq Y$ and an isomorphism of varieties $X' \to Y'$ mapping $p$ to $q$. This is equivalent to requiring the local rings $\calO_{X,p}$ and $\calO_{Y,q}$ to be isomorphic as $\bbC$-algebras. In the literature, one usually considers a less restrictive notion of isomorphism, where the \emph{completions} $\hat{\calO}_{X,p}$, $\hat{\calO}_{Y,q}$ are required to be isomorphic. This distinction is irrelevant for our purposes. We say that a surface singularity $(X,p)$ is \emph{normal} if the local ring $\calO_{X,p}$ is integrally closed.

Let $G \subseteq \GL_2$ be a finite subgroup. Then $G$ acts on $\bbA^2$ and the quotient $X_G = \bbA^2 / G$ is a surface. A \emph{quotient singularity} is a surface singularity $(X,p)$ that is isomorphic to $(X_G,0)$ for some finite subgroup $G \subseteq \GL_2$. In particular, quotient singularities are normal. We say that $(X,p)$ is a \emph{cyclic quotient singularity} of type $\frac{1}{r}(a,b)$ if $(X,p)$ is isomorphic to $(X_G,0)$ where $G = \mu_r$ is the multiplicative group of $r$-th roots of $1$ acting by $\zeta \cdot (u_1,u_2) = (\zeta^a u_1, \zeta^b u_2)$, for a primitive root of unity $\zeta$.

Let $(X,p)$ be a normal surface singularity. A resolution of $(X,p)$ is a proper birational morphism $\rho : \tilde{X} \to (X,p)$ such that $\tilde{X}$ is smooth and $\rho: \tilde{X} \setminus \rho^{-1}(p) \to X\setminus p$ is an isomorphism, possibly after replacing $X$ with an open neighborhood of $p$; the fiber $\rho^{-1}(p)$ is called the \emph{exceptional locus} of the resolution. We say that $\tilde{X}$ is minimal if any other resolution $\tilde{X}'$ of $(X,p)$ factors through $\tilde{X}$. Every surface singularity admits a minimal resolution, unique up to isomorphism \cite[Thm. 2.16]{Kol07}.

We say that a normal surface singularity is a \emph{Du Val singularity} if the exceptional locus of a minimal resolution $\phi: \tilde{X} \to (X,p)$ is a union of smooth rational curves of self-intersection $-2$, intersecting transversely, and whose dual graph is a Dynkin diagram of type $\mathsf{A}$, $\mathsf{D}$ or $\mathsf{E}$. There is a characterisation in terms of the group action: a normal surface singularity is Du Val if and only if it is a quotient singularity $\bbA^2 / G$ for a finite subgroup $G \subseteq \SL_2$ \cite[Rmk. 4.21]{KM98}. In particular, a cyclic quotient singularity of type $\frac{1}{r}(a,b)$ is Du Val if and only if $a+b \equiv 0 \bmod r$. The quotient singularities relevant for this work are described in the following example.
\begin{example}
The cyclic quotient singularity of type $\frac{1}{r}(1,-1)$ is a Du Val singularity of type ${\mathsf{A}}_{r-1}$. Its minimal resolution is obtained by successive blow-ups of the singular point. The exceptional locus is a chain of $r-1$ rational curves with self-intersection $-2$. 
On the other hand, the cyclic quotient singularity $\frac{1}{r}(1,1)$ is not Du Val if $r \geq 3$. Its minimal resolution has a single exceptional curve of self-intersection $-r$.
\end{example}

\subsection{Some examples of Schubert surfaces}\label{subsec:examples} 

We conclude this section by presenting a collection of examples of the Schubert varieties occurring in small dimension. These examples anticipate the phenomena studied in the rest of the paper. The very first instance of the construction is the quadric surface of $\bbP^3$.

\begin{example}[Quadric surface]\label{ex:quadric-surface}
Consider the pencil 
\[
L=J_1(x_1)\boxplus J_1(x_2) =
\left\{
{\footnotesize \begin{pmatrix}
x_1&0\\
0&x_2
\end{pmatrix} }
\mid x_1,x_2\in \bbC
\right\} \subseteq V^* \otimes W
\]
with $\dim V = \dim W = 2$. This is the linear space parameterised by a generic tensor in $\bbC^2 \otimes \bbC^2 \otimes \bbC^2$. The open embedding of $L$ in $\Gr(2,V\oplus W)$ is defined by 
\[ (x_1,x_2) \mapsto {\footnotesize \begin{pmatrix}
1 & 0 & x_1&0\\
0 & 1 & 0&x_2
\end{pmatrix}} , 
\] 
which in the Pl\"ucker coordinates is given by to $\phi(x_1,x_2) = [1:x_1:x_2:x_1x_2]$; therefore $Y_L$ is the quadric surface in $\bbP^3$. The map $\phi$ extends to a rational map $\bbP^2 \dashto \bbP^3$ with $\phi([x_0:x_1:x_2]) = [x_0^2 :x_0x_1:x_0x_2:x_1x_2]$ whose indeterminacy locus is the set $Z = \{ [0,1,0], [0,0,1] \}$. The resolution of this map blows-up the two points, making the line spanned by them  into a $(-1)$-curve. The quadric is the result of the contraction of that $(-1)$-curve to a point. This gives the classical isomorphism between the blow-up of $\bbP^2$ at two points and the blow-up of $\bbP^1 \times \bbP^1$ at one point. 

More generally, for the pencil $L=J_1(x_1)^{\boxplus e_1}\boxplus J_1(x_2)^{\boxplus e_2}$, the Schubert variety of $L$ is again $\PP^1\times\PP^1$, and the Pl\"ucker embedding of $\Gr(e_1+e_2, 2(e_1+e_2))$ restricts to the embedding given by $\calO_{\bbP^1 \times \bbP^1} (e_1,e_2)$. In this case $Y_L$ is a Segre-Veronese variety. These varieties will appear in case (1)(ii) of \Cref{prop: smooth YL}.
\end{example}

An interesting variation on \Cref{ex:quadric-surface} is when the indeterminacy locus degenerates to a local scheme of length $2$.

\begin{example}[Quadratic cone]\label{ex: length 2 scheme}
Consider the pencil
\[
L=J_2(x_1)
=
\left\{
{\footnotesize \begin{pmatrix}
x_1 & x_2\\
0   & x_1
\end{pmatrix} }
\mid x_1,x_2\in \bbC
\right\}.
\]
with $\dim V = \dim W = 2$. When regarded as a tensor in $\bbC^2 \otimes \bbC^2 \otimes \bbC^2$, this is a general tangent vector to the Segre variety. The rational map $\phi: \bbP^2 \dashto \bbP^3$ is defined by $\phi([x_0:x_1:x_2]) = [x_0^2 : x_0x_1 : x_0x_2 : x_1^2]$. The indeterminacy locus of $\phi$ is a scheme of length $2$ supported at the point $[0,0,1]$ and contained in the line at infinity $x_0=0$. The Schubert variety $Y_L$ is the quadric cone in $\bbP^3$. The resolution of $\phi$ will be studied, in a more general setting, in \Cref{subsec: smooth resolution}: the exceptional locus is a chain of two curves, with self-intersection $(-2)$ and $(-1)$. The blow-up procedure makes the line at infinity into a $(-1)$-curve. The surface $Y_L$ is the result of the contraction of the exceptional $(-2)$-curve, resulting in the $\mathsf{A}_1$-singularity at the vertex of the cone, and of the line at infinity, resulting in a smooth point.

More generally, if $L = J_k(x_1)$ consists of a single Jordan block, then the indeterminacy locus of the rational map is a local scheme $Z$ with ideal $I(Z) = (x_0,x_1^k)$. In this case, the exceptional locus of the resolution of $\phi$ is a chain of $k$ curves with self-intersection $(-2), \ldots, (-2), (-1)$. The line at infinity lifts to a curve with self-intersection $-(k-1)$. The surface $Y_L$ is obtained by the contraction of all exceptional $(-2)$-curves, resulting in the $\mathsf{A}_{k-1}$-singularity, and the contraction of the strict transform of the line at infinity, resulting in a singularity of type $\frac{1}{k-1}(1,1)$. When $k = 2,3$, this falls into case (1) of \Cref{prop: reconstruct case by case}, and if $k > 3$, it falls into case (1) of \Cref{prop: W non DuVal singularity}.
\end{example}

\begin{example}[Hirzebruch surface]\label{ex: hirzebruch}
Consider the pencil $L=J_1(x_1)\boxplus J_2(x_1)$ with $\dim V = \dim W = 3$. The rational map $\phi: \bbP^2 \dashto \bbP^5$ is given by  
\[
\phi([x_0 : x_1 : x_2]) = [x_0^3 : x_0^2x_1 : x_0^2x_2 :  x_0x_1^2 :x_0x_1x_2:x_1^3].
\]
The indeterminacy locus of $\phi$ is a local scheme $Z$ supported at the point $[0,0,1]$ whose ideal is $(x_0,x_1)(x_0,x_1^2)$. A resolution of $\phi$, as constructed in \Cref{subsec: smooth resolution}, has a chain of two exceptional curves with self-intersection $(-2),(-1)$ respectively. The blow-up makes the line at infinity into a $(-1)$-curve, and $Y_L$ is the result of the contraction of that curve. In particular, $Y_L$ is smooth. In fact, $Y_L$ is the Hirzebruch surface $\bbF_2=\PP(\calO_{\bbP^1} \oplus \calO_{\bbP^1}(2))$. The Pl\"ucker embedding of $\Gr(3,6)$ restricts to the embedding of $\bbF_2$ given by the divisor $E + 3F$, where $E$ is the negative section and $F$ is a fiber of the ruling. The image of this embedding is a rational normal scroll $S(1,3) \subseteq \bbP^5$ of degree $4$. More generally, if $L=J_1(x_1)^{\boxplus e_1}\boxplus J_2(x_1)^{\boxplus e_2}$, the Schubert variety $\YL$ is again the surface $\bbF_2$, embedded by the linear series associated to the divisor $e_2E + (e_1+2e_2)F$. When $e_2 = 1$, the image is the rational normal scroll $S(e_1,e_1+2) \subseteq \bbP^{2e_1 + 3}$. This gives rise to the family of case (1)(iii) in \Cref{prop: smooth YL}.
\end{example}

\section{Schubert varieties of matrix pencils}\label{sec: Schubert surfaces}

Throughout the rest of the paper, we focus on Schubert varieties of dimension $2$, and we refer to them as Schubert surfaces. Fix vector spaces $V,W$ with $\dim V = n$ and $\dim W = m$. Given a matrix pencil $L \subseteq V^* \otimes W$, let $Y_L$ denote the associated Schubert surface. The open embedding $L \to \Gr(n,V \oplus W)$ defines a rational map 
\begin{equation}\label{eq: phi} 
\begin{matrix}
\varphi : &  \mathbb P(\bbC \oplus L) & \dashrightarrow & \Gr(n,V\oplus W) \\
& [1:X] & \mapsto & [\Gamma_X]
\end{matrix} 
\end{equation}
and $\YL = \overline{\im\varphi}$. Fix coordinates $x_0,x_1,x_2$ on $\mathbb P^2 = \bbP (\bbC \oplus L )$ and consider on $\Gr(n, V \oplus W)$ the Pl\"ucker coordinates given by the maximal minors of the matrix $[I | X]$. Then the rational map $\phi$ is defined by the linear series $ D_n\!\left( \big[ x_0I_n | L(x_1,x_2)\big] \right)$, that is the degree $n$ component of the ideal of maximal minors of the matrix of the $n\times (n+m)$ matrix of linear forms $\big[ x_0I_n | L(x_1,x_2)\big]$. As explained in \Cref{subsec: pencils}, a pencil $L$ has a unique decomposition in terms of Kronecker invariants: we record here the notation that we will use throughout the rest of the paper.

\begin{center}

\framebox[15.7cm]{
	\begin{minipage}{15.2cm}
        \vspace{1.5ex}
    \begin{notation}\label{notation:pencils}
    The pencil $L$ has a unique decomposition into a left-singular component, a right-singular component and a regular component: $L =  S^{\rmleft}_\bullet \boxplus S^{\rmright}_\bullet \boxplus L^{\mathrm{reg}}_\bullet$. Let $\epsilon_{\#}$ and $\eta_{\#}$ be the number of left and right singular blocks, respectively. Let $\boldsymbol{\epsilon} = (\epsilon_1 \vvirg \epsilon_{\epsilon_{\#}})$ and $\boldsymbol{\eta}= (\eta_1,\ldots, \eta_{\eta_{\#}})$ be the left and right singular Kronecker invariants of $L$. In particular
    \[
    S^{\rmleft}_\bullet = S^{\rmleft}_{\epsilon_1} \boxplus \ldots \boxplus S^{\rmleft}_{\epsilon_{\epsilon_{\#}}}, \qquad
    S^{\rmright}_\bullet = S^{\rmright}_{\eta_1} \boxplus \ldots \boxplus S^{\rmright}_{\eta_{\eta_{\#}}}.
    \]
    Let $\epsilon_\bullet = \sum \epsilon_i$ be the number of rows of $S^{\rmleft}_{\bullet}$ and $\eta_\bullet = \sum \eta_i$ for the number of columns of $S^{\rmright}_\bullet$. Let $n_0= \epsilon_\bullet + \eta_\bullet + \eta_{\#}$ and $m_0 = \epsilon_\bullet + \epsilon_{\#} + \eta_\bullet$ be the number of rows and columns of the singular components. 
    
    Let $\theta_1, \ldots , \theta_s \in \mathbb C[x_1,x_2]_1$ be the $s$ distinct generalised eigenvalues of $L$: we may assume they have the form $\theta_i = x_1 - \zeta_ix_2$ for distinct $\zeta_i \in \bbC$. For every $i \in [s]$, set:
    \begin{itemize}
    \item $k_{i  1} < \ldots < k_{i  h_i}$ the distinct sizes of the Jordan blocks $J_{k_{i 1}}(\theta_i), \ldots , J_{k_{i \, h_i}}(\theta_i)$;
    \item $e_{i ,j} \geq 1$ the number of blocks $J_{k_{i j}}(\theta_i)$;
    \item $d_i := e_{i 1} + \ldots + e_{i h_i}$ the total number of Jordan blocks with eigenvalue $\theta_i$;
    \item $L_i := J_{k_{i 1}}(\theta_i)^{\boxplus e_{i 1}}\boxplus \ldots \boxplus J_{k_{i \, h_i}}(\theta_i)^{\boxplus e_{i h_i}}$ the sum of Jordan blocks with eigenvalue $\theta_i$;
    \item $n_i := e_{i 1}k_{i 1} + \ldots + e_{i h_i}k_{i \, h_i}$ be the size of $L_{i}$;
    \item $X_i := L(\zeta_i,1)$ the rank degenerate matrix at the value $\zeta_i$ with $\theta_i(\zeta_i,1) = 0$;
    \item $Z_i\subseteq \mathbb P^2=\mathbb P(\bbC \oplus L)$ the $0$-dimensional scheme supported at $[0:X_i]$ defined by the system of maximal minors of the matrix $D_{n_i}([x_0 I_{n_i} | L_i(x_1,x_2)])$.
    \end{itemize}
    \end{notation}
        \vspace{0.5mm}
        \end{minipage}}
    \end{center}

\subsection{Linear series defining $\phi$}\label{subsec: linear series} 
 In this section, we prove several results regarding the linear series $D_n([x_0I | L(x_1,x_2)])$ of maximal minors of a pencil of matrices. In particular, we prove the following result, characterising the indeterminacy locus of the rational map $\phi$ of \eqref{eq: phi}.
 
\begin{theorem}\label{theorem: linear series of YL}
Let $L$ be a matrix pencil of size $n \times m$. Then, the system of maximal minors of $\hat{L} = [x_0 I_n | L(x_1,x_2)]$ is $x_0^{\eta_{\#}} I(Z)_{n-\eta_{\#}}$ where $Z = Z_1 \cup \ldots \cup Z_s$ is the union of the $0$-dimensional scheme defined by each regular component of $L$. In particular, the scheme $Z$ is the indeterminacy locus of the rational map $\phi $ parameterising $Y_L$. Moreover, the regularity of $Z$ is
    \[
    \reg(Z) = n_1 + \ldots + n_s -1  .
    \]
\end{theorem}

The proof of \Cref{theorem: linear series of YL} is based on two technical results which can be seen as a generalisation of \cite[Lemma 3.2]{GeKe:collineations} to matrices of the form $\hat{L}$. In particular, \Cref{lemma: systems of minors} gives an explicit characterisation of the system of minors in terms of certain monomial ideals, and \Cref{lemma: linear series single eigenvalue} shows that such monomial ideals coincide with the ideal of a $0$-dimensional scheme and computes its degree and regularity.
 
\begin{lemma}\label{lemma: systems of minors}  The following hold:
\begin{enumerate}[leftmargin=*]
\item Let $S^{\rmleft}_\bullet$ be a pencil of matrices of size $\epsilon_\bullet \times (\epsilon_\bullet + \epsilon_{\#})$ consisting only of left-singular blocks. Then 
\[
D_{\epsilon_\bullet} \left(\big[x_0I | S^{\rmleft}_\bullet]\right)  = \bbC[x_0,x_1,x_2]_{\epsilon_\bullet}
\]
is the complete linear series of degree $\epsilon_\bullet$ on $\bbP^2$;
\item Let $S^{\rmright}_\bullet$ be a pencil of matrices of size $(\eta_\bullet + \eta_{\#}) \times \eta_\bullet$ consisting only of right-singular blocks. Then 
\[
D_{\eta_\bullet + \eta_{\#}} \left(\big[x_0I | S^{\rmright}_\bullet]\right)  = x_0^{\eta_{\#}}\bbC[x_0,x_1,x_2]_{\eta_\bullet}
\]
is a multiple of the complete system of degree $\eta_\bullet$ on $\bbP^2$;
\item Let $k_1 < \ldots < k_h$ and $e_1 \vvirg e_h$ be positive integers with $n = \sum e_ik_i$. Let $\theta$ be a linear form. Let $L = J_{k_1}(\theta)^{\boxplus e_1} \boxplus \ldots \boxplus J_{k_h}(\theta)^{\boxplus e_h}$ be a regular pencil with unique eigenvalue $\theta$. Then 
\[
D_{n} \left(\big[x_0 I_{n}|L(x_1,x_2)\big]\right) = \bigl((x_0,\theta^{k_{1}})^{e_{1}} \ldots (x_0, \theta^{k_h})^{e_h} \bigr)_{n}.
\]
\end{enumerate}
\end{lemma}
\begin{proof}
   Let $\hat{M} = [x_0 I | M(x_1,x_2)]$, where $M$ is one of the three pencils of the statement. In all three cases, after possibly reordering the columns $\hat{M}$, we may assume that $\hat{M}$ is block diagonal with blocks of the form $[x_0 I | M'(x_1,x_2) ]$ with $M'$ ranging over the blocks of $M$ in its Kronecker decomposition. The ideal of maximal minors of a block diagonal matrix is the product of the ideals of maximal minors of its blocks. Therefore it suffices to prove the statement for indecomposable blocks. Consider the identity
    \begin{equation}\label{eqn: split minors}
   D_{n} ( \hat{M} ) = \textstyle \sum_{\ell = 0}^{n} x_0^{n-\ell} D_\ell(M).
    \end{equation}
    \begin{enumerate}[leftmargin=*]
    \item If $M = S^{\rmleft}_p$, then $\hat{M}$ is a $p \times (2p+1)$ matrix. For every $q = 0 \vvirg p$, the minors of size $q$ of $S^{\rmleft}_p$ define the complete system of degree $q$ in the variables $x_1,x_2$. By \eqref{eqn: split minors}, we conclude that $D_p( \hat{M}) = \bbC[x_0,x_1,x_2]_p$. 
    \item If $M = S^{\rmright}_p$, then $\hat{M}$ is a $(p+1) \times (2p+1)$ matrix. For every $q = 0 \vvirg p$, the minors of size $q$ of $S^{\rmright}_p$ define the complete system of degree $q$ in the variables $x_1,x_2$. As a result, the last term in the summation in \eqref{eqn: split minors} is zero, and we obtain $D_{p+1}( \hat{M}) = x_0 \bbC[x_0,x_1,x_2]_p$.
    \item If $M = J_p(\theta)$, then $\hat{M}$ is a $p \times 2p$ matrix. For every $q < p$, we have $D_q(J_p(\theta)) = \bbC[x_1,x_2]_q$, whereas $D_p(J_p(\theta)) = det(J_p(\theta)) = \theta^p$. Therefore 
    \[
    D_p(\hat{L}) = \textstyle \sum_{\ell = 0}^{p-1} x_0^{p-\ell} \bbC[x_1,x_2]_\ell + \langle \theta^p \rangle = (x_0,\theta^p)_p. \qedhere
    \]
\end{enumerate}
\end{proof}

\begin{lemma}\label{lemma: linear series single eigenvalue}
    Let $\theta \in \bbC[x_1,x_2]$ be a linear form. Let $\tilde{k}_1 \leq \ldots \leq \tilde{k}_d$ be positive integers. Define the ideal
    \[
   J := (x_0,\theta^{\tilde{k}_1}) \cdots (x_0,\theta^{\tilde{k}_d}) \subseteq \bbC[x_0,x_1,x_2].
    \]
    Then $J$ is saturated, and it defines a local $0$-dimensional scheme $Z \subseteq \bbP^2$ satisfying 
    \[
   \deg Z = \textstyle \sum_{j=1}^{d} (d-j+1)\tilde{k}_i, \qquad \reg(Z) = \tilde{k}_1 + \ldots + \tilde{k}_d -1.
    \]
\end{lemma}
\begin{proof}
Assume without loss of generality $\theta = x_1$, so that $J$ is a monomial ideal. Let $n = \tilde{k}_1 + \ldots + \tilde{k}_d$. Let $J^\sat$ be the saturation of $J$; we show $J = J^\sat$. Let $x_0^\alpha x_1^\beta x_2^\gamma \in J^\sat$. Then, for a large enough integer $N$, we have $x_0^\alpha x_1^\beta x_2^\gamma \cdot x_2^N \in J$, therefore $\beta \geq \sum_{j=1}^r \tilde{k}_{i}$ and $\alpha \geq d-r$ for a certain $r$; this implies $x_0^\alpha x_1^\beta x_2^\gamma \in J$. Therefore $J$ is saturated.

Let $Z$ be the scheme defined by $J$. Then $Z$ is $0$-dimensional and supported at $[0,0,1] \in \bbP^2$. First observe $x_1^{n} \in J$ is a minimal generator: indeed, it cannot be generated by any elements of lower degree since any other monomial generator of $J$ is a multiple of $x_0$. This shows  $\reg(Z) \geq n -1$.

Let $\rmHF_{Z}(t)=\codim J_t$ be the Hilbert function of $Z$. We complete the proof by showing  
\begin{equation}\label{eqn: HFZ}
\rmHF_{Z} (n-1) = \rmHF_{Z}(n) =  \textstyle \sum_{j=1}^{d} (d-j+1)\tilde{k}_{j} .
\end{equation}
By \cite[Thm. 1.69]{IarrKan:PowerSumsBook}, this implies $\reg(Z) = n-1$ and $\deg(Z) = \rmHF_{Z}(n)$.

Consider the ideal $J' = J \cap \bbC[x_0,x_1]$: by construction, $J'$ is the ideal of the scheme $Z$ regarded as an affine scheme on the chart $\{ x_2 \neq 0\}$ of $\bbP^2$. In particular, $\rmHF_{Z}(t) = \sum_{t'=0}^t \rmHF_{\bbC[x_0,x_1]/J'}(t')$ and it coincides with the number of standard monomials of $J'$ of degree at most $t$. Therefore \eqref{eqn: HFZ} is equivalent to the fact that $J'$ has $\sum _1^d \tilde{k}_{j}$ standard monomials, all of degree at most $n-1$. 

For $a \geq d$, the monomial $x_0^{a} x_1^{b}$ lies in $J'$. For $a < d$, then $x_0^ax_1^b \notin J'$ if and only if $b < \sum_{j=1}^{d-a} \tilde{k}_{j}$. Therefore there are $\tilde{k}_{1} + \ldots + \tilde{k}_{ d-a}$ standard monomials of the form $x_0^{a} x_1^{b}$. In particular, all standard monomials for $J'$ have degree at most $n-1$, showing $\rmHF_{\bbC[x_0,x_1]/J'}(t') = 0$ if $t \geq n$ and therefore $\reg(Z) = n-1$. Adding up the number of standard monomials for every $a$, we obtain the statement for the degree.
\end{proof}

\begin{proof}[Proof of \Cref{theorem: linear series of YL}]
As in the proof of \Cref{lemma: systems of minors}, write $\hat{L} = [x_0 I | L(x_1,x_2)]$ and, after permuting the columns of $\hat{L}$, suppose that $\hat{L}$ has a block diagonal structure. We obtain 
\[
D_n(\hat{L}) = D_{\epsilon_\bullet}([x_0I | S^{\rmleft}_\bullet ]) \cdot D_{(\eta_\bullet + \eta_{\#})}([x_0I | S^{\rmright}_\bullet ]) \cdot \prod_{1}^s D_{n_i}([x_0I_{n_i}|L_{i}]).
\]
\Cref{lemma: systems of minors} provides the description of each factor in the multiplication above. Let $\tilde{k}_1 \leq \cdots \leq \tilde{k}_{d}$ be the repeated sizes of the Jordan blocks, so that the size $k_j$ occurs $e_j$ times. We apply \Cref{lemma: linear series single eigenvalue} and obtain that the factors arising from the regular blocks coincide with $I(Z_i)_{n_i}$. Therefore, we obtain 
\begin{align*}
D_n(\hat{L}) &=  x_0^{\eta_{\#}} \cdot \bbC[x_0,x_1,x_2]_{\epsilon_\bullet + \eta_\bullet} \cdot I(Z_1)_{n_1} \cdots I(Z_s)_{n_s}  = \left( x_0^{\eta_{\#}} \cdot I(Z_1)_{n_1} \cdots I(Z_s)_{n_s} \right)_n.
\end{align*}
We claim that
\begin{equation}\label{eqn: product graded pieces}
I(Z_1)_{n_1} \cdots I(Z_s)_{n_s} = (I(Z_1) \cdots I(Z_s))_{n_1 + \ldots + n_s}.
\end{equation}
The containment $\subseteq$ is clear. To prove the other containment recall $\reg(Z_i) = n_i-1$ and write 
\[
(I(Z_1) \cdots I(Z_s))_{n_1 + \ldots + n_s} = 
\sum_{\substack{p_1 \vvirg p_s \\ 
    p_1 + \ldots + p_s = n_1 + \ldots + n_s}}
    I(Z_1)_{p_1} \cdots I(Z_s)_{p_s}.
\]
Fix $(p_1 \vvirg p_s)$ such that $p_1 + \ldots + p_s = n_1 + \ldots + n_s$. Let $\bfp_{\geq} = \{ i \in [s] : p_i \geq  n_i \}$ and let $\bfp_{<}$ be its complement in $[s]$. If $i\in \bfp_{\geq}$ then $I(Z)_{p_i} = I(Z_i)_{n_i} \cdot \CC[x_0,x_1,x_2]_{p_i-n_i}$ because $I(Z_i)$ is generated in degree at most $n_i$ by \Cref{lemma: linear series single eigenvalue}. On the other hand, if $i \in \bfp_{<}$, then $I((Z_i)_{p_i})_{n_i} \subseteq I(Z_i)_{n_i}$. We obtain 
\[
I(Z_1)_{p_1} \cdots I(Z_s)_{p_s} \subseteq \prod_{i \in \bfp_{\geq}} I(Z_i)_{n_i} \cdot \prod_{i \in \bfp_\geq} \CC[x_0,x_1,x_2]_{p_i-n_i} \cdot \prod_{i \in \bfp_{<}} I(Z_i)_{p_i} \subseteq I(Z_i)_{n_1} \cdots I(Z_s)_{n_s}.
\]
This shows the inclusion $\supseteq$ in \eqref{eqn: product graded pieces}.

Now, since $Z_1 \vvirg Z_s$ are supported at distinct points, from \cite[Theorem 1.8]{Sid02} we have
\begin{equation}\label{eqn: reg product of ideals}
\reg(I(Z_1)\cdots I(Z_s))\leq \sum_{i=1}^s \reg(I(Z_i)) =n_1 + \ldots + n_s.
\end{equation}
Since ideals agree with their saturation in degree larger than their regularity, and using again that $Z_1 \vvirg Z_s$ are supported at distinct points, setting $n_\bullet = n_1 + \ldots + n_s$, we have
\[
I(Z_1 \cup \ldots \cup Z_s)_{n_\bullet} =  \big( (I(Z_1) \cdots I(Z_s))^{\mathrm{sat}} \big)_{n_\bullet} = \big( I(Z_1) \cdots I(Z_s) \big)_{n_\bullet} = I(Z_1)_{n_1} \cdots I(Z_s)_{n_s}. 
 \]
This completes the proof that $D_n([x_0I_n|L]) = x_0^{\eta_{\#}} \cdot I(Z_1\cup \ldots \cup Z_s)_{n-\eta_{\#}}$. 

Finally, the degree is additive on unions of $0$-dimensional schemes with disjoint supports, which gives the result about $\deg(Z_1\cup \ldots \cup Z_s)$. The bound $\reg(Z_1 \cup \ldots \cup Z_s) \leq n_1 + \ldots n_s-1$ follows from the standard relation $\reg(I) = \reg(S/I) +1$ where $S$ denotes the polynomial ring. Moreover, the bound is sharp because $I(Z_1 \cup \ldots \cup Z_s)$ has a minimal generator of degree $n_1 + \ldots + n_s$.
\end{proof}

As an immediate corollary of \Cref{theorem: linear series of YL}, we obtain:
\begin{corollary}\label{rmk:YL does not see singular invariants}
Let $L,L'$ be two pencils of the same size and having the same regular component. Then $\YL$ and $Y_{L'}$ are projectively equivalent. In particular, the sizes of the individual singular blocks cannot be recovered from the geometry of $\YL$.
\end{corollary}

\begin{corollary}\label{cor: image of phi totally singular}
Let $S^{\rmleft}_\bullet$ and $S^{\rmright}_\bullet$  be pencils consisting only of left or right singular blocks, with singular invariants $\bfepsilon$ and $\bfeta$, respectively. Let $L = S^{\rmleft}_\bullet \boxplus S^{\rmright}_\bullet$. Then the rational map $\phi$ is regular and $\YL$ is a Veronese surface of degree $\epsilon_\bullet + \eta_\bullet$.
\end{corollary}
\begin{proof}
\Cref{theorem: linear series of YL} guarantees that $\phi$ is regular and given by the system $x_0^{\eta_{\#}} \cdot \bbC[x_0,x_1,x_2]_{\epsilon_\bullet + \eta_\bullet}$. Therefore, the image is parameterised by the system $\bbC[x_0,x_1,x_2]_{\epsilon_\bullet + \eta_\bullet}$, which defines a Veronese surface of degree ${\epsilon_\bullet + \eta_\bullet}$.
\end{proof}

\subsection{The restriction of $\phi$ to the line at infinity}\label{subsec: restriction of phi to line at infinity} 
In this section we describe the image of the restriction of the rational map $\phi:\mathbb P^2 \dashto \YL$ to the line at infinity
\[ 
H_\infty:= \{x_0=0\} \subset \mathbb P^2=\mathbb P(\mathbb C \oplus L) .
 \]
 Since this is a line, the restriction $\phi|_{H_\infty}$ is a regular map, and its image is either a point or a rational curve in $\YL$ whose degree is computed as follows.
 
\begin{proposition}\label{prop: image of phi} 
The following holds:
\begin{enumerate}[leftmargin=*]
\item\label{prop: image of phi regular} 
If $L$ is a regular pencil, then the image of $\phi|_{H_\infty}$ is the point $[W]\in \Gr(n,V\oplus W)$. In particular, $[W] \in \YL$.
\item\label{prop: image of phi mixed}
If $L$ is a singular pencil with left and right singular invariants $\bfepsilon$ and $\bfeta$ respectively. Then the image of $\phi|_{H_\infty}: H_\infty \to \YL$ in the Pl\"ucker embedding of $\Gr(n,V \oplus W)$ is a rational normal curve $C_\infty$ of degree $\epsilon_\bullet + \eta_\bullet$.
\end{enumerate}
\end{proposition}

We first characterise the image of $\phi$ at points $[0:X] \in H_\infty$ with $X$ of generic rank in $L$.
\begin{proposition}\label{prop: im plus ker}
Let $X\in L$ be an element of generic rank in $L$. Then 
\[
\phi([0:X])= [\ker(X) \oplus \image(X)] .
\]
\end{proposition}
\begin{proof}
By \cref{theorem: linear series of YL}, the indeterminacy locus of $\phi$ is supported at the rank defective element on the line at infinity. In particular, since $X$ is an element of generic rank, the map $\phi$ is well defined at $[0:X]$. To compute the value $\phi([0:X])$, it suffices to compute the restriction to any line through $[0,X]$. Consider the line parameterised as 
\[ \begin{matrix}
\gamma :& \bbC &\to &\bbP( \bbC \oplus L) \\
& t &\mapsto & [t: X] = [1: t^{-1} X].
\end{matrix}
\]
Let $v_1 \vvirg v_n$ be a basis of $V$ and suppose $v_1 \vvirg v_k$ is a basis of $\ker(X)$.  For $t\neq 0$, we have 
\[
\Gamma_{t^{-1}X} = \langle v_i + t^{-1}Xv_i \mid i \in [n] \rangle = \ker(X) + \langle tv_{i} + X v_{i} \mid i = k+1 \vvirg n\rangle .
\]
We have $\lim_{t\to 0 } [tv_{i} + X v_{i}] = [Xv_i]$ as an element in $\bbP (V \oplus W)$. Therefore, 
\[
\lim_{t \to 0} \Gamma_{t^{-1}X} = \ker(X) + \langle Xv_{k+1} \vvirg Xv_n\rangle = \ker(X) + \image(X) . \qedhere
\]
\end{proof}

We point out that the argument of \Cref{prop: im plus ker} generalises beyond the case of pencils.
\begin{remark}
     We remark that the indeterminacy locus of the map $\phi$ for a general linear space $L\subseteq V^* \otimes W$ has codimension at least $1$ in the hyperplane at infinity $H_\infty \subseteq \bbP( \bbC \oplus L )$. In particular, for $X \in L$ generic, $\phi([0:X])$ is well-defined and the same argument as in the proof of \Cref{prop: im plus ker} shows $\phi([0:X]) = [\ker(X) \oplus \image(X)]$. The difference between the general setting and that of pencils is that \Cref{theorem: linear series of YL} describes precisely the indeterminacy locus of $\phi$, showing that it is supported on the locus of rank-defective matrices in $H_\infty$.
\end{remark}
The proof of \Cref{prop: image of phi} is now straightforward.
\begin{proof}[Proof of \Cref{prop: image of phi}]
\Cref{prop: image of phi regular} follows from \Cref{prop: im plus ker}: if $L$ is regular, then $\ker(X) = 0$ and $\image(X) = W$ for generic $X \in L$. In particular, $\phi([0:X]) = [W]$.

To prove \Cref{prop: image of phi mixed}, we use \Cref{prop: im plus ker} and parameterise the curve defined by 
\[ \begin{matrix}
\gamma : & \bbP L & \to & \Gr(n, V \oplus W) \\
& [X] & \mapsto & [\ker(X) \oplus \image(X)] 
\end{matrix}  \]
for $X$ generic in $L$. By the block structure of $L$, $\gamma$ is realised as a sum, in the sense of \Cref{cor: sum of rnc}, over the indecomposable blocks of $L$. We consider three different cases:

\begin{itemize}
\item If $L = S^{\rmleft}_p$ and $X = L(1,t)$ is generic, then $X$ is surjective and 
\[
\ker X = \langle v_0- tv_1 + t^2v_2 - \ldots + (-t)^p v_{p} \rangle \subseteq \bbC^{p+1}
\]
where $v_0 \vvirg v_p$ is a basis of the domain of the singular block $S^{\rmleft}_p$. Therefore $\ker X$ is a rational normal curve in $\bbP (\bbC^{p+1}) = \Gr(1,\bbC^{p+1})$.
\item If $L = S^{\rmright}_p$ and $X = L(1,t)$ is generic, then $X$ is injective and 
\[
\image X = (\ker X^{T})^\perp = \langle w_0- tw_1 + t^2w_2 - \ldots + (-t)^p w_{p} \rangle^\perp \subseteq \bbP({\bbC^{p+1}})^*
\] 
where $w_0 \vvirg w_p$ is a basis of the dual space of the codomain of the singular block $S^{\rmright}_p$. Therefore $\image(X)$ is a rational normal curve defined under the identification $\Gr(1, ({\bbC^{p+1}})^*) = \Gr(p , \bbC^{p+1})$. 
\item If $L = J_p(\theta)$ and $X = L(1,t)$ is a generic element, then $X$ is bijective.
\end{itemize}

Adding the degrees of the rational normal curves obtained from the left and right singular blocks, by \Cref{cor: sum of rnc}, we conclude.
\end{proof}

We conclude this section with a result concerning Schubert varieties of linear spaces of square matrices of any dimension. Suppose $\dim V = \dim W = n$ and let $L \subseteq V^* \otimes W$ be a linear subspace containing full rank linear maps $V \to W$. In this setting, we have $[W] \in Y_L$ and the singularity $(Y_L , [W])$ can be characterised in terms of the top collineation variety $\scrC_{n-1}(L)$, defined in \cite{GeKe:collineations} as the Zariski closure of the image of the inversion map $\mathbb{P} L \dashrightarrow \mathbb P(W^* \otimes V)$ given by $[X] \mapsto [X^{-1}]$ for any full rank $X\in L$.

\begin{proposition}\label{prop: W singularity general}
Let $n = \dim V = \dim W$. Let $L \subseteq V^* \otimes W$ be an $a$-dimensional linear space whose general element is full rank as a linear map $V \to W$. Then $[W] \in Y_L$ and the singularity $(Y_L, [W])$ is isomorphic to the singularity $(\hat{\scrC}_{n-1}(L),0)$ in the affine cone over the collineation variety $\scrC_{n-1}(L)$. In particular, the projectivised tangent cone $\bbP \TC_{[W]} Y_L$ coincides with $\scrC_{n-1}(L)$, and $[W]$ is a smooth point of $Y_L$ if and only if $\scrC_{n-1}(L)$ is a linear space.
\end{proposition}
\begin{proof}
First, observe that $[W] \in Y_L$, since for $X\in L$ full rank one gets $\lim_{t \to 0} \Gamma_{t^{-1}X} = [W]$. 

Let $\calU_V \subseteq \Gr(n, V \oplus W)$ be the open subset defined by $\calU_V = \{ [E] : V \cap E = \emptyset\}$ and let $U = \calU_V \cap Y_L$ which is an open affine neighborhood of $[W] \in U$. We show that $(U,[W]) \simeq (\hat{\scrC}_{n-1}(L),0)$ as an affine variety. 

Consider the graph map analogous to the map $\phi$ defining $Y_L$, and given by 
\[ \begin{matrix}
\phi' : & \Hom(W,V) & \to & \Gr(n,V \oplus W) \\
& X & \mapsto & \Gamma'_X 
\end{matrix} \]
where $\Gamma'_X$ denotes the graph of the linear map $X : W \to V$. The map $\phi'$ is an open embedding: let ${\phi'}^{-1}$ denotes its inverse on $\calU_V$. For any invertible $X \in L$, we have ${\phi'}^{-1} \phi(X) = X^{-1}$. In particular, ${\phi'}^{-1} : U \to \Hom(W,V)$ maps $[\Gamma_X]$ to $X^{-1}$ and parameterises an open subset of $\hat{\scrC}_{n-1}(L)$. Since ${\phi'}^{-1}$ is invertible, taking closures in $\calU_V$, we obtain an isomorphism $U \simeq \hat{\scrC}_{n-1}(L)$ mapping $[W]$ to $0$, as desired.

Since $\phi'$ is locally an isomorphism, its differential defines a linear isomorphism from the tangent space to $\Hom(W,V)$ at $0$ to the tangent space to $\Gr(n, V \oplus W)$ at $[W]$. Such differential maps $\scrC_{n-1}(L) = \bbP \TC_{0}\hat{\scrC}_{n-1}(L)$ to $\bbP \TC_{[W]} Y_L$. This concludes the proof.
\end{proof}

We record a consequence of \Cref{prop: W singularity general} in the case of matrix pencils. Define the map
\[ 
\begin{matrix}
\psi : & H_\infty & \to & \bbP T_{[W]} Y_L  \\
 & [0:X]  & \mapsto & \textstyle \frac{d}{du}|_{u=0} \phi([u:X])
\end{matrix}
\]
mapping a point $[0:X] \in H_\infty$ with generic $X$ to the tangent vector of the image of a line passing through $[0:X]$. 

\begin{corollary}\label{prop: W singularity}
Let $L=L_{1}\boxplus \ldots \boxplus L_{s}$ be a regular pencil of size $n$, with $L_{i}$ of size $n_i=\sum_{j \in [h_i]} e_{ij}k_{ij}$. Let $\kappa := \sum_{i \in [s]} k_{i \, h_{i}}$ be the sum of the maximal sizes of Jordan blocks, one for each eigenvalue. Then, the following holds: 
\begin{enumerate}
    \item[$(i)$] the tangent cone $\TC_{[W]}Y_L$ is a cone over a rational normal curve of degree $\kappa - 1$; \vspace{1ex}
    \item[$(ii)$] $[W]$ is a smooth point of $Y_L$ if and only if $\kappa = 2$; \vspace{1ex}
    \item[$(iii)$] if $\kappa \geq 3$, then $Y_L$ has a quotient singularity at $[W]$ of type $\frac{1}{\kappa -1}(1,1)$; \vspace{1ex}
    \item[$(iv)$] the map $\psi$ is well defined, regular and it coincides with the $(\kappa-1)$-th Veronese embedding of $H_\infty$ onto $\bbP T_{[W]} Y_L$.
\end{enumerate}
\end{corollary}
\begin{proof}
    By \Cref{prop: W singularity general}, the tangent cone $\TC_{[W]}Y_L$ is a cone over $\scrC_{n-1} (L)$. By \cite[Thm. 3.1]{GeKe:collineations}, $\scrC_{n-1} (L)$ is a rational normal curve, of degree $n-1-p$ where $p = \deg ( \bbP L \cap \sigma_{n-2})$ is the degree of the intersection of $L$ with the variety $\sigma_{n-2}$ of linear maps of rank at most $n-2$. The degree of this $0$-dimensional scheme can be computed explicitly observing that 
    \[
\textstyle    D_{n-1} (L) = \sum_{i \in [s]} D_{n_i-1}(L_i) \prod_{i' \neq i} \det(L_{i'}) = \left(\prod_{i \in [s]} \theta_i^{n_i - k_{i \, h_i}}\right) \cdot \bbC[x_1,x_2]_{\kappa-1}.
    \]
We obtain $p = n-1 - \sum_{i \in [s]} (n_i - k_{i \, h_i}) = \kappa -1$ as desired. This proves (i).

Items (ii) and (iii) are immediate consequences of (i). Indeed, the vertex of the cone of a rational normal curve of degree $\kappa-1$ is a quotient singularity of type $\frac{1}{\kappa -1} (1,1)$.

Finally, to prove (iv), observe that the map $\psi$ is the collineation map itself: fix $X \in L$ generic and let $\gamma_X(t) = \phi([t:X]) = \Gamma_{t^{-1}X}$. By genericity, $X$ is invertible and $\Gamma_{t^{-1}X} = \Gamma'_{tX^{-1}}$. The derivative of $\gamma_X$ at $t = 0$ is therefore given by $X^{-1}$ regarded as an element of $\bbP (W^* \otimes V) = \bbP T_{[W]} \Gr(n, V \oplus W)$. This shows that $\psi: H_\infty \to \bbP \TC_{[W]} Y_L$ is defined by mapping $[X] \in H_\infty$ to $[X^{-1}] \in \bbP (W^* \otimes V)$, which coincides with the collineation map.
\end{proof}

\subsection{Schubert class of Schubert surfaces}\label{subsec: schubert classes}
In this section, we determine the expression of the class $[Y_L]$ of a pencil $L$ in the Chow ring of the Grassmannian $\Gr(n , V\oplus W)$, establishing Item 3 of \Cref{thm: main forward}. The Chow ring of $\Gr(n,V \oplus W)$ is generated by Schubert classes $\sigma_\lambda$, indexed by partitions $\lambda$ contained in the $n \times m$ rectangle, where $n = \dim V$ and $m = \dim W$. Since $Y_L$ is a surface, its class $[Y_L]$ is a $\bbZ$-linear combination of the two classes $\sigma_{(1,1)^c}$ and $\sigma_{(2)^c}$ where $\lambda^c$ is the complement of the partition $\lambda$ in the $n \times m$ rectangle. The coefficients can be expressed in terms of the sizes $\eps_\bullet$ and $\eta_\bullet$ of the singular components, and of the degree of $Y_L$ in the Pl\"ucker embedding, which is computed in \Cref{theorem: degree of YL}.
\begin{theorem}\label{thm: schubert class}
    Let $L$ be a pencil of size $n \times m$. Let $\eps_\bullet$ and $\eta_\bullet$ be the sum of the sizes of the left singular and right singular blocks of $L$. Then the class of the Schubert surface $Y_L$ in the Chow ring of $\Gr(n,V \oplus W)$ is given by
    \[
[Y_L] =  \frac{1}{2} \left( \deg Y_L + \eps_\bullet - \eta_\bullet \right) \sigma_{(2)^c} + \frac{1}{2} \left( \deg Y_L - \eps_\bullet + \eta_\bullet \right) \sigma_{(1,1)^c} .
    \]
\end{theorem}

The proof of \Cref{thm: schubert class} is based on the fact that the sum of the two coefficients is determined by $\deg Y_L$ as observed in \Cref{lemma: sum coefficients}. To determine the difference of the two coefficients, we study the pullback Chern class of the tautological bundle $\calS$ of $\Gr(n,V \oplus W)$ via the parameterisation map $\phi$ of the Schubert variety of $L$. The computation can be split across the different blocks using \Cref{prop: grassmannian direct sum}. This reduces the computation of the difference $a-b$ to the case of indecomposable pencils and will complete the proof of \Cref{thm: schubert class}.

\begin{lemma}\label{lemma: sum coefficients}
Let $[Y_L] = a \sigma_{(2)^c} + b \sigma_{(1,1)^c}$ be the class of the Schubert surface $Y_L$ in the Chow ring of $\Gr(n,V \oplus W)$. Then $a+b = \deg Y_L$.
\end{lemma}
\begin{proof}
Let $h$ be the hyperplane class of the Pl\"ucker embedding. Then the degree of $Y_L$ is the product of the class of $[Y_L]$ with $h^2$. Since the pullback to the Grassmannian of the hyperplane class of the Pl\"ucker embedding is $\sigma_1$, we obtain  
\[
\deg Y_L = [Y_L] \cdot \sigma_1^2 = (a \sigma_{(2)^c} + b \sigma_{(1,1)^c}) \cdot (\sigma_{(2)} + \sigma_{(1,1)}) = a + b
\]
where we use Pieri's rule to compute $\sigma_1^2 = \sigma_{(2)} + \sigma_{(1,1)}$ and the fact that if $|\lambda| + |\mu| = nm$ then $\sigma_\lambda \cdot \sigma_\mu = 1$ if $\lambda^c = \mu$ and $0$ otherwise.
\end{proof}

The same argument as in \Cref{lemma: sum coefficients} characterises the difference of the two coefficients as the product $[Y_L] \cdot (\sigma_{(2)} - \sigma_{(1,1)})$. The upshot is that the difference $\sigma_{(2)} - \sigma_{(1,1)}$ is, in a way, additive across different Kronecker blocks. 

\begin{proof}[Proof of \Cref{thm: schubert class}]
   Let $\calS$ be the tautological bundle of $\Gr(n,V \oplus W)$ and let $\calQ$ be the quotient bundle. By Whitney's formula \cite[Thm. 5.3]{EisHar:3264}, we have 
$c(\calS) \cdot c(\calQ) = 1$, where $c(-)$ denotes the total Chern class. Since $\sigma_{(2)} = c_2(\calQ)$ and $\sigma_{(1,1)} = c_2(\calS)$, we obtain 
\[
\sigma_{(2)} - \sigma_{(1,1)} = c_2(\calQ) - c_2(\calS) = c_1(\calS)^2 -2c_2(\calS).
\]
The latter is up to scaling the second \emph{Chern character} of $\calS$, which is additive on direct sums \cite[Sec. 5.5.2]{EisHar:3264}. By \Cref{prop: grassmannian direct sum}, we obtain 
\begin{align*}
    (a-b) = &[Y_L] \cdot (\sigma_{(2)} - \sigma_{(1,1)}) = [Y_L] \cdot (c_1(\calS)^2 -2c_2(\calS)) = \\
    &(\Delta^* [Y_L]) \cdot \Delta^* (c_1(\calS)^2 -2c_2(\calS)) = (\Delta^* [Y_L]) \cdot \textstyle \sum_{i=1}^r  (c_1(\calS_i)^2 -2c_2(\calS_i));
\end{align*}
here we use that the map $\Delta$ of \Cref{prop: grassmannian direct sum} is an embedding, together with additivity of the Chern character. Let $M_1 \vvirg M_r$ be the indecomposable Kronecker blocks of $L$, as in \Cref{prop: grassmannian direct sum}. The $i$-th summand $(\Delta^* [Y_L] ) (c_1(\calS_i)^2 -2c_2(\calS_i))$ coincides with $[Y_{M_i}] \cdot (c_1(\calS_i)^2 -2c_2(\calS_i))$. Therefore, the calculation is reduced to computing the differences $a-b$ for indecomposable pencils. 

Let $M$ be left singular of size $p \times (p+1)$. We prove $[Y_M](\sigma_2 - \sigma_{(1,1)})= p$. In this case the map $\phi: \bbP^2 \to \Gr(p, 2p+1)$ is regular; we are going to show $\phi^* \calS = \calO_{\bbP^2}(-1)^{\oplus p}$. Let $\hat{M} = [x_0 I | M]$, which has rank $p$ everywhere on $\bbP^2$. In particular, (the transpose of) $\hat{M}$ is everywhere injective, so it defines a map of vector bundles $\hat{M}^\bft : \calO_{\bbP^2}(-1)^{\oplus p} \to \calO_{\bbP^2}^{\oplus(2p+1)}$ whose image is a subbundle of rank $p$ of $ \calO_{\bbP^2}^{\oplus(2p+1)} = (V \oplus W) \otimes \calO_{\bbP^2}$. The fiber of this subbundle at $x \in \bbP^2$ is the rowspan of $\hat{M}$, that is exactly the fiber of $\calS$ at $\phi(x)$. This shows $\phi^* \calS = \calO_{\bbP^2}(-1)^{\oplus p}$, giving the total Chern class $c(\phi^* \calS) = (1-h)^p$ where $h$ is the hyperplane class of $\bbP^2$. In particular, $c_1(\phi^*\calS)  = -ph$ and $c_2(\phi^*\calS) = \binom{p}{2} h^2$. Therefore 
\[
[Y_M] (c_1(\calS)^2 -2c_2(\calS)) = (c_1(\phi^* \calS)^2 - 2 c_2(\phi^* \calS)) = p^2 - 2\binom{p}{2} = p.
\]

Let $M$ be right singular of size $(p+1) \times p$. We prove $[Y_M] (\sigma_2 - \sigma_{(1,1)}) = -p$. The same argument as above shows that the bundle map $\hat{M}: \calO_{\bbP^2}^{\oplus(2p+1)} \to \calO_{\bbP^2}(1)^{\oplus p}$ is surjective with kernel $\phi^*\calS$. We have the exact sequence $0 \to \phi^*\calS \to \calO_{\bbP^2}^{\oplus(2p+1)} \to \calO_{\bbP^2}(1)^{\oplus p} \to 0$ and by Whitney's formula $c(\calO_{\bbP^2}^{\oplus(2p+1)} ) = c(\phi^*\calS) \cdot c(\calO_{\bbP^2}(1)^{\oplus p})$. In particular, we obtain $c(\phi^*\calS) = (1+h)^{-p}$, giving $c_1(\phi^*\calS) = -ph$ and $c_2(\phi^*\calS) = \binom{p+1}{2} h^2$. Therefore
\[
[Y_M] (c_1(\calS)^2 -2c_2(\calS)) = (c_1(\phi^* \calS)^2 - 2 c_2(\phi^* \calS)) = p^2 - 2\binom{p+1}{2} = -p.
\]

Let $M$ be a Jordan pencil of size $p$; without loss of generality assume its eigenvalue is $x_1$. We prove $[Y_M] (\sigma_2 - \sigma_{(1,1)}) = 0$. Fix bases on $V$ and $W$ and use them to have an identification $V^* \simeq W$ and $W^* \simeq V$. In this way, one defines an isomorphism $\tau : \Gr(p, V \oplus W) \to \Gr(p , V^* \oplus W^*)  \simeq \Gr(p ,W \oplus V)$ by sending a subspace $E$ to its annihilator $E^\perp$. If $E = \Gamma_X$ is the graph of a map $X$ then $E^\perp$ is the graph of the map $-X^\bft$.  $\tau(J_p(x_1))$ is a Jordan pencil as $J_p(x_1)$, and in particular $[Y_M] = [Y_{\tau(M)}] = \tau^* [Y_M]$. On the other hand, $\tau^* \calS = \calQ^*$. This shows $\tau^* \sigma_{(1,1)} = \sigma_{(2)}$ and $\tau^* \sigma_{(2)} = \sigma_{(1,1)}$. Therefore,
\[
[Y_M] (\sigma_2 - \sigma_{(1,1)}) = \tau^* [Y_M] \tau^*(\sigma_2 - \sigma_{(1,1)}) = [Y_M] (\sigma_{(1,1)} - \sigma_{(2)})
\]
implying $[Y_M] (\sigma_2 - \sigma_{(1,1)}) = 0$. Adding the three contributions together, we conclude.
\end{proof}

\section{Intersection theory in Schubert surfaces}\label{sec:intersection theory}

In this section, we provide a smooth resolution of the indeterminacy of the parameterisation map $\phi : \bbP^2 \dashto Y_L$. This resolution is constructed via iterated smooth blow-ups at the reduced points supporting the indeterminacy locus $Z_1 \cup \ldots \cup Z_s$ described in \Cref{theorem: linear series of YL}. The exceptional locus of the blow-up will be a union of rational curves: their intersection products and their degree will reflect the invariants of the pencil $L$, and in turn will be the core of the reconstruction statement of \Cref{thm: reconstruction of invariants}.

\subsection{A smooth resolution of $\phi$}\label{subsec: smooth resolution}

We construct a smooth variety $\tilde{X}$ endowed with a regular map $\pi : \tilde{X} \twoheadrightarrow \bbP^2$ birational onto $\bbP^2 \setminus Z$ inducing a natural extension $\tilde{\phi} : \tilde{X} \twoheadrightarrow Y_L$, giving the commuting diagram: 
\[
\xymatrix{
\tilde{X} \ar[d]_\pi \ar[dr]^{\tilde{\phi}} & \\
\bbP^2 \ar@{-->}[r]_\phi & Y_L .
}
\]
Recall that the indeterminacy locus of $\phi$ is the zero-dimensional scheme $Z = Z_1 \cup \ldots \cup Z_s$, supported at the points $\{ x_0 = \theta_i = 0\}$ defined by the eigenvalues of the pencil. The map is given by the linear series $I(Z)_{n - \eta_{\#}}$. The variety $\tilde{X}$ is constructed as a sequence of smooth blow-ups $\tilde{X} \to \bbP^2$ after which the pullback of the linear series $I(Z)_{n - \eta_{\#}}$ becomes invertible. The construction is local over the support of $Z$, and the blow-up chains over the distinct points $\{x_0=\theta_i = 0\}$ are independent.

For every $i \in [s]$, let $k_{i1} < \ldots < k_{i \, h_i}$ be the distinct sizes of Jordan blocks with eigenvalue $\theta_i$ and let $e_{i1} \vvirg e_{i h_i}$ be their multiplicities. Recall that $I(Z) = (x_0 ,\theta_i^{k_{i1}})^{e_{i1}} \cdots (x_0,\theta_i^{k_{i \, h_i}})^{e_{i h_{i}}}$.

Starting from the point $\{ x_0 = \theta_i = 0\}$, the sequence of blow-ups $\tilde{X}$ is obtained by successive blow-ups, $k_{i \, h_i}$ times, each time at the intersection of the strict transform of the line at infinity $H_\infty$ with the most recently created exceptional divisor. 

We describe the construction in detail for a local scheme $Z$ supported at $z_0$. Set $X_0 := \bbP^2$ and $H^{(0)} :=  H_\infty$, so that $z_0 \in H_\infty$. Consider the blow-up $\pi_1:X_1:=\Bl_{z_0}X_0 \twoheadrightarrow X_0$ and denote by $A_1^{(1)}:= \pi_1^{-1}(z_0)$ the exceptional divisor, and by $H^{(1)}:=\overline{\pi_1^{-1}(H^{(0)}\setminus z_0)}$ the strict transform of the line at infinity. Set $z_1:= H^{(1)}\cap A_1^{(1)}$.

Recursively define $\pi_{i+1}: X_{i+1}:= \Bl_{z_{i}}X_{i} \twoheadrightarrow X_{i}$ to be the blow-up of $X_{i}$ with center the point $z_{i}:= H^{(i)}\cap A_i^{(i)}$, where $A_i^{(i)}:=\pi_i^{-1}(z_{i-1}) \subset X_i$ is the exceptional divisor of $\pi_i$, and $H^{(i)} := \overline{\pi_i^{-1}(H^{(i-1)}\setminus z_{i-1})}$ is the strict transform of $H^{(i-1)}$. Observe that the point $z_{i} \in X_i$ lies on $A_i^{(i)}$ but does \textit{not} lie on the strict transform $A_{i-1}^{(i)}:=\overline{\pi_i^{-1}(A_{i-1}^{(i-1)}\setminus z_{i-1})}$. 

For $i' < i-1$, let $A^{(i)}_{i'} \subset X_i$ be the strict transform of $A^{(i')}_{i'}$ via the concatenation of blow ups $\pi_i\circ \ldots \circ \pi_{i'+1}$. In particular, $z_{i-1} \notin A^{(i-1)}_{i'}$ if $i' < i-1$, therefore $\pi_i : X_i \to X_{i-1}$ restricts to an isomorphism from $A^{(i)}_{i'}$ to $A^{(i-1)}_{i'}$ whenever $i' < i-1$. \Cref{fig: iterated blowup} represents the first steps of this process.

\begin{figure}[h]
\begin{tikzpicture}[scale=0.85]
    \pgfmathsetmacro{\xbase}{10}
    
    \draw[thick] (\xbase,-2) -- (\xbase, 1.5) node[pos=1, right=3pt] {$X_0 = \bbP^2$};
    
    \draw[thick] (\xbase,-2) -- (\xbase, -1.7) node[font=\footnotesize,pos=1, right=3pt] {$H^{(0)} = H_\infty$};
    
    \node[anchor=west, font=\scriptsize, fill=white] at (\xbase,0) {$z_0$};
    \node at (\xbase,0) [circle,fill,inner sep=1pt]{};

    \pgfmathsetmacro{\xbase}{7}
    \draw[thick] (\xbase,-2) -- (\xbase, 1.5) node[pos=1, right=3pt] {$X_1$};
    \draw[thick] (\xbase,-2) -- (\xbase, -1.7) node[font=\footnotesize,pos=1, right=3pt] {$H^{(1)}$};
    \node[anchor=north west, yshift=5pt, font=\scriptsize, fill=white] at (\xbase,0) {$z_1$};
    \draw[thick, rotate around={30:(\xbase-0.2,0)},color=red]
        (\xbase+0.2, 0) -- (\xbase-0.6, 0)
        node[pos=0.7, font=\tiny, above] {$A^{(1)}_1$};
    \node at (\xbase,0.12) [circle,fill,inner sep=1pt]{};
    \draw[->, thick] (\xbase+.7,-1) -- (\xbase+2.3,-1);
    \node at (\xbase+1.4,-0.8) [inner sep=1pt]{$\pi_1$};
    
    \pgfmathsetmacro{\xbase}{4}
    \draw[thick] (\xbase,-2) -- (\xbase, 1.5) node[pos=1, right=3pt] {$X_2$};
    \draw[thick] (\xbase,-2) -- (\xbase, -1.7) node[font=\footnotesize,pos=1, right=3pt] {$H^{(2)}$};
    \node[anchor=north west, yshift=2pt, font=\scriptsize, fill=white] at (\xbase,0) {$z_2$};
    \draw[thick, rotate around={-30:(\xbase-0.2,0)},color=blue]
        (\xbase+0.2, 0) -- (\xbase-0.6, 0)
        node[pos=0.7, font=\tiny, below] {$A^{(2)}_2$};
    \draw[thick, rotate around={30:(\xbase-0.7,0)},color=red]
        (\xbase-0.3,0) -- (\xbase - 1.1, 0)
        node[pos=0.8, font=\tiny, above] {$A^{(2)}_1$};
    \node at (\xbase,-0.12) [circle,fill,inner sep=1pt]{};
    \draw[->, thick] (\xbase+.7,-1) -- (\xbase+2.3,-1);
    \node at (\xbase+1.4,-0.8) [inner sep=1pt]{$\pi_2$};

    \pgfmathsetmacro{\xbase}{1}
    \draw[thick] (\xbase,-2) -- (\xbase, 1.5) node[pos=1, right=3pt] {$X_3$};
    \draw[thick] (\xbase,-2) -- (\xbase, -1.7) node[font=\footnotesize,pos=1, right=3pt] {$H^{(3)}$};
    \node[anchor=north west, yshift=4pt, font=\scriptsize, fill=white] at (\xbase,0) {$z_3$};
    \node at (\xbase,0.12) [circle,fill,inner sep=1pt]{};
    \draw[thick, rotate around={30:(\xbase-0.2,0)}]
        (\xbase+0.2, 0) -- (\xbase-0.6, 0)
        node[pos=0.6, font=\tiny, above] {$A^{(3)}_3$};
    \draw[thick, rotate around={-30:(\xbase-0.7,0)},color=blue]
        (\xbase-0.3,0) -- (\xbase - 1.1, 0)
        node[pos=0.8, font=\tiny, below] {$A^{(3)}_2$};
    \draw[thick, rotate around={30:(\xbase-1.2,0)},color=red]
        (\xbase-0.8,0) -- (\xbase - 1.6, 0)
        node[pos=0.7, font=\tiny, above] {$A^{(3)}_1$};
    \draw[->, thick] (\xbase+.5,-1) -- (\xbase+1.8,-1);
    \node at (\xbase+1,-0.8) [inner sep=1pt]{$\pi_3$};

    \pgfmathsetmacro{\xbase}{-2}
    \draw[->, thick] (\xbase,-1) -- (\xbase+1.3,-1);
    \node at (\xbase+.6,-0.8) [inner sep=1pt]{$\pi_4$};
    \node at (\xbase-1,-1) [inner sep=1pt]{$\cdots$};

\end{tikzpicture}
\caption{The iterated blow-up $X_{m} \to X_{m-1} \to \cdots \to X_1 \to \bbP^2$.}
\label{fig: iterated blowup}
\end{figure}

Let $\tilde{X}$ be the variety obtained after performing, for every $i$, $k_{i \, h_i}$ blow-ups at the point $\{ x_0 = \theta_i = 0\}$ in the support of the indeterminacy locus. We use the notation $A_{ij}$ and $E_{ij}$, for the strict transform and the total transform in $\tilde{X}$ of the exceptional divisor of the $j$-th blow-up of the chain corresponding to eigenvalue $i$, respectively. Moreover, we denote by $\tilde H\subset \tilde X$ the pullback of the generic hyperplane $H \subset \mathbb P^2$ via the blow-up map.

The following result provides the desired smooth resolution $\tilde X$ of $\phi$.

\begin{proposition}\label{prop: resolution of phi}
The extension $\tilde{\phi} : \tilde{X} \to Y_L$ is regular. It is given by the complete linear series of the divisor 
\[ 
D := (n - \eta_{\#}) \tilde H - \sum_{i \in [s]} \sum_{j\in [k_{i \, h_i}]} \mu_{ij} E_{ij}
\quad \text{ where } \quad 
\mu_{ij} := \sum_{\{ \ell \mid k_{i \ell} \geq j \} } e_{\ell} .  
\]
In particular, $\YL$ is linearly normal.
\end{proposition}

Since the blow-up is a local operation, the proof of \Cref{prop: resolution of phi} is reduced to a local calculation along each of the exceptional chains. More precisely, we give a recursive description of the pullback of the linear series $I(Z)_{n - \eta_{\#}}$. We recall the notion of \emph{weak transform} of a scheme under a blow-up, following \cite[p.29]{HauserBlowups}. 

Let $X$ be a smooth variety, $Z \subseteq X$ be a (possibly non-reduced) scheme and let $z \in Z$ be a reduced point. Let $\pi : \Bl_z X \to X$ be the blow-up of $X$ with center $z$. The total transform of $Z$ in $\Bl_z X$ is a scheme $\pi^{-1}(Z)$ having the exceptional divisor $E$ of $\pi$ in its support. Since $E$ is Cartier, locally it has a principal ideal: the \emph{weak transform} of $Z$ in $\Bl_zX$ is the scheme $Z' \subseteq \Bl_z X$ defined as the residual scheme of $\pi^{-1}(Z)$ with respect to $E$, namely the scheme that, locally, has ideal $I(\pi^{-1}(Z)) : I(E)$.

If $Z$ is reduced, then the weak transform coincides with the strict transform of $Z$: more precisely, the strict transform is the scheme that, locally, has ideal $I(\pi^{-1}(Z)) : I(E)^{\infty}$, which coincides with $I(\pi^{-1}(Z)) : I(E)$ if $Z$ is reduced. However, if $Z$ is non-reduced, then the total transform $\pi^{-1}(Z)$ may be non-reduced as well, and in this case the weak transform retains the information of this non-reduced structure \emph{near} $E$. In our case, $Z$ is local and supported at the point $z$, and the total transform of $Z$ will consist of the exceptional divisor $E$ with an embedded prime at a specific point $z' \in E$. Keeping track of the behaviour of such embedded prime allows us to prove \Cref{prop: resolution of phi}.

The local description is given in the following result.

\begin{lemma}\label{lemma: transform one point}
    Let $I = \prod_{i \in [h]} (x_0,x_1^{k_i})^{e_i}$. Then, the weak transform of $I$ after the $\ell$-th blow-up is supported at $z_\ell$ and it is the ideal
    \[
   I_\ell = \prod_{\{ j \mid k_j > \ell \}}(\tilde{x}_0 , \tilde{x}_1^{k_j-\ell})^{e_j}
    \]
    where $\tilde{x}_0$ and $\tilde{x}_1$ are the local coordinates at $z_\ell$ of the strict transform of $H_\infty$ and the exceptional divisor of the $\ell$-th blow-up, respectively. In particular, the order of vanishing of $I_\ell$ at $z_\ell$ is $\mu_{\ell+1}$, and if $\ell \geq k_h$, then $I_\ell$ is the unit ideal.
\end{lemma}

\begin{proof}
We apply induction on $\ell$. For $\ell = 0$, we have $I = I_0$ and there is nothing to prove. Assume $\ell \geq 1$. Let $h'$ be such that $k_{h'} \leq \ell \leq k_{h'+1}$ for some $h' < h$, with possibly $h' = 0$ in which case assume $k_0=0$ or $h' = h$ in which case assume $k_{h'+1} = \infty$. Let $U\subseteq X_{\ell-1}$ be an affine open set containing $z_{\ell-1}$, and let $\bar{x}_0, \bar{x}_1$ be local coordinates on $U$ such that $H^{(\ell-1)}\cap U = \{\bar{x}_0 = 0\}$ and $A_{\ell-1}^{(\ell-1)}=\{\bar{x}_1 = 0\}$. Similarly, let $\tilde{U}\subseteq X_\ell$ be an affine open set containing $z_\ell$, and let $\tilde{x}_0, \tilde{x}_1$ be local coordinates on $\tilde{U}$ such that $H^{(\ell)}\cap \tilde U = \{\tilde{x}_0 = 0\}$ and $A_{\ell}^{(\ell)}\cap \tilde U = \{\tilde{x}_1 = 0\}$.

Let $\ell > 0$. By the induction hypothesis, the weak transform $Z_{\ell-1}$ is supported at $z_{\ell-1}$ and it has ideal
\[
I(Z_{\ell-1}) = \prod_{i =h''+1}^h (\bar{x}_1^{k_i - (\ell-1)}, \bar{x}_0)^{e_i}
\]
where $h'' = h'-1$ if $\ell = k_{h'}$ and $h'' = h'$ otherwise. By definition, $\Bl_{z_{\ell-1}} U$ is cut out by the equation $\bar{x}_0 u_1 - \bar{x}_1 u_0 = 0$ in $U\times \PP^1$ and it has two affine charts $B_i = \Bl_{z_{\ell-1}} U \cap \{ u_i \neq 0\}$ for $i=0,1$. Here $u_0,u_1$ are coordinates on $\bbP^1$. The exceptional divisor $A^{(\ell)}_\ell$ has equation $\bar{x_i} = 0$ on $B_i$.

The equation of the strict transform $H^{(\ell)}$ of $H^{(\ell-1)}$ is computed on the two charts as follows. On $B_0$, the equation of the total transform of $H^{(\ell)}$ is $\bar{x}_1 = 0$ so that removing the factor $\bar{x}_1 = 0$ of the exceptional divisor, we have $H^{(\ell)} \cap B_0 = \emptyset$. On $B_1$, the equation of the total transform is $\bar{x}_1 = 0$, which implies $\bar{x}_0 u_1 = 0$ and by removing the factor $\bar{x}_0 = 0$ of the exceptional divisor, we conclude $H^{(\ell)} \cap B_1 = \{ u_1 = 0 \}$. In particular, observe that $\bar{x}_1 = 0$ as well, since on $B_1$ we have $\bar{x}_1 = \bar{x}_0 u_1$.

The weak transform of the scheme $Z_{\ell-1}$ is a scheme $Z_\ell$ computed similarly. On $B_0$, setting $u_0 = 1$, we have $\bar{x}_1 = u_1\bar{x}_0$, so the ideal of the total transform of $Z_{\ell-1}$ is 
\[
\prod_{i =h''+1}^h (\bar{x}_1^{k_i - (\ell-1)}, \bar{x}_0)^{e_i} = \prod_{i =h''+1}^h ((u_1\bar{x}_0)^{k_i - (\ell-1)}, \bar{x}_0)^{e_i}.
\]
By removing the factor $\bar{x}_0 = 0$ of the exceptional divisor, we have $Z_\ell \cap B_0 = \emptyset$. On the open set $B_1$, setting $u_1 = 1$, we have $\bar{x}_0 = u_0\bar{x}_1$, so the ideal of the total transform of $Z_{\ell-1}$ is
\[
\prod_{i =h''+1}^h (\bar{x}_1^{k_i - (\ell-1)}, \bar{x}_0)^{e_i} =  \prod_{i =h''+1}^h (\bar{x}_1^{k_i - (\ell-1)}, u_0\bar{x}_1)^{e_i} .
\]
By removing the factor $\bar{x}_1 = 0$ of the exceptional divisor, we obtain 
\begin{equation}\label{eqn: weak transform of Z}
I(Z_{\ell} \cap B_1 )  = \prod_{i =h''+1}^h (\bar{x}_1^{k_i - \ell}, u_0)^{e_i}.
\end{equation}
If $\ell= k_{h''+1}$, then the first factor in \eqref{eqn: weak transform of Z} is the trivial ideal $(1)$ and the ideal in \eqref{eqn: weak transform of Z} is
\[
I(Z_{\ell} \cap B_1 ) = \prod_{i =h'+1}^h (\bar{x}_1^{k_i - \ell}, u_0)^{e_i}.
\]
Here $\bar{x}_1 = 0$ is the local equation of the exceptional divisor $A^{(\ell)}_\ell$, and $u_0 = 0$ is the local equation of the strict transform $H^{(\ell)}$. In particular, $Z_{\ell}$ is supported at $z_\ell = A^{(\ell)}_\ell \cap H^{(\ell)}$. Rename these coordinates as $\tilde{x}_1$ and $\tilde{x}_0$, respectively, to obtain the desired expression of the ideal of $Z_\ell$ in $\tilde{U}$.

The computation of the order of vanishing is straightforward: 
\[
\max \{ q \mid (\tilde{x}_0, \tilde{x}_1)^q \supseteq I(Z_{\ell} \cap B_1) \} = \sum_{i =h''+1}^h e_i = \mu_{\ell+1}.
\]
Finally, if $\ell = k_h$, then $Z_{k_h}$ is empty and the ideal $I(Z_{k_h} \cap B_1)$ is the unit ideal.
\end{proof}

The proof of \Cref{prop: resolution of phi} is obtained by pulling back the linear series defining $\phi$. We recall a general fact, see e.g. \cite[Ex. IV.2.(ii)]{Beau96}: \textit{Let $X$ be a smooth variety, $z \in X$ a point and $\beta: X' \to X$ the blow-up at $z$ with exceptional divisor $E$. If $f: X \dashto Y$ is given by a linear series $U \subseteq H^0(X , \calO_X(F))$ for a divisor $F$, then the extension $f' : X' \dashto Y$ is given by a linear series $U' \subseteq H^0(X' , \calO_{X'}(\beta^*F - qE))$, where $q$ is the multiplicity of $U$ at $z$.}

\begin{proof}[Proof of \Cref{prop: resolution of phi}]
  Applying \Cref{lemma: transform one point} at every base point of the blow-up, we obtain that $\tilde{\phi}$ is a regular extension of $\phi$. 

 We apply \cite[Ex. IV.2.(ii)]{Beau96}. In our case, $\beta: \tilde X \to X$ is $\pi: \tilde X \to \mathbb P^2$, and $F$ is the generic hyperplane $H\subset \mathbb P^2$. We use induction on the number of blow-ups. From \Cref{theorem: linear series of YL} we know $n-\eta_\# \geq \reg(Z)$, hence the ideal sheaf $\mathcal{I}_Z(n-\eta_\#)$ on $\mathbb P^2$ is globally generated, and $\varphi$ is defined by the linear system $I(Z)_{n-\eta_\#}$ inside the complete linear series $(n-\eta_\#)H$. Working locally, assume there is a single chain of blow-ups and let $D_\ell$ be the divisor defining the extension after $\ell$ blow-ups. The multiplicity is computed in \Cref{lemma: transform one point}, and we obtain
  \[
 D_{\ell+1} = D_\ell - \mu_{\ell+1} E_{\ell+1}
  \]
  where we omit the pullback via the blow-up map from the notation. 
  Over a fixed chain, and after $k_h$ blow-ups, we deduce $D_{k_h} = D_0 - \sum_{\ell=1}^{k_h} \mu_\ell E_\ell$. Adding the contribution over all chains, we obtain the desired expression 
  \[
 D = (n-\eta_{\#})\tilde H - \sum_{i \in [s]}\sum_{j \in [k_{i \, h_i}]} \mu_{ij} E_{ij}.
  \]
This shows that $\pi^\ast \mathcal{I}_Z = \mathcal{O}_{\tilde X}(-\sum \mu_{ij}E_{ij})$ and $\tilde \phi$ is given by a linear series $U' \subseteq H^0(\tilde X, \mathcal{O}_{\tilde X}(D))^\ast$.  It remains to show that the above containment is actually an equality.  By construction, $U' = H^0(\mathbb P^2, \mathcal{I}_Z(n-\eta_\#))$, identified as a subspace of $H^0(\tilde{X}, \calO_{\tilde{X}}(D))$ via the transform of the divisors in $I(Z)_{n-\eta_\#}$. 

To conclude, it suffices to show $\pi_\ast \mathcal{O}_{\tilde X}(-\sum \mu_{ij}E_{ij}) = \mathcal{I}_Z$, which would imply 
\begin{align*}
H^0(\tilde X, \mathcal{O}_{\tilde X}(D)) = H^0(\mathbb P^2, \pi_\ast\mathcal{O}_{\tilde X}(D)) = H^0(\mathbb P^2, \mathcal I_Z(n-\eta_\#)) .
\end{align*} 
To prove $\pi_\ast \mathcal{O}_{\tilde X}(-\sum \mu_{ij}E_{ij}) = \mathcal{I}_Z$, we use the general fact that $\pi_\ast \mathcal{O}_{\tilde X}(-\sum \mu_{ij}E_{ij})$ is the integral closure of the ideal sheaf $\mathcal{I}_Z$ \cite[Remark 9.6.4]{lazarsfeld2017positivity}, and we observe that $\calI_Z$ is integrally closed. This follows because locally, at each support point of $Z$, $\calI_Z$ is product of integrally closed monomial ideals in a two dimensional regular local ring; under this assumption, $\calI_Z$ is integrally closed \cite[Appendix 5, Thm. 2']{zariski2013commutative}. 
\end{proof}

We complete our study of the smooth resolution $\tilde X$ with the description of the non-trivial multiplications in the Chow ring $\CH^\bullet (\tilde{X})$. Denote by $\tilde{\zeta} = [\tilde H]$ the class of pullback of the generic hyperplane of $\mathbb P^2$ and by $[\tilde H_\infty]$ the class of the strict transform of the line at infinity $H_\infty \subset \mathbb P^2$. Denote by $\alpha_{ij}:=[A_{ij}]$ the classes of the exceptional curves, for $i \in [s]$ and $j\in [k_{i \, h_i}]$. Observe that the total transform $E_{ij}$ is the sum $E_{ij}= A_{ij} + \ldots + A_{i \, k_{i \, h_i}}$.

\begin{lemma}\label{lemma: intersections upstairs}
The Chow ring $\CH^\bullet(\tilde X)$ is generated by the classes $\tilde\zeta$ and $\alpha_{ij}$ for all $i \in [s]$ and $j \in [k_{i \, h_i}]$. Their product is given by  
\[
\alpha_{ij} \cdot \alpha_{i'j'} = \left\{\begin{array}{ll} 
-\delta_{ii'} & \text{if } j = j' = k_h \\
-2\delta_{ii'} & \text{if } j = j' < k_h \\
\delta_{ii'} & \text{if } |j-j'| = 1 \\
0 & \text{otherwise}
\end{array}\right. ,  \qquad \alpha_{ij} \cdot \tilde{\zeta} = 0, \qquad \tilde{\zeta}^2 = 1.
\]
In particular, $[E_{ij}] \cdot [E_{i'j'}] = -\delta_{ii'}\delta_{jj'}$. Moreover, $A_{ij}$ is the unique representative of its class $\alpha_{ij} \in \CH^\bullet(\tilde X)$. Finally, the class of the strict transform of the line at infinity $H_\infty$ is given by
\[ 
[\tilde H_\infty] = \tilde{\zeta} - \sum_{i \in [s]}\sum_{j \in [k_{i \, h_i}]} j \alpha_{ij} = \tilde{\zeta} - \sum_{i \in [s]}\sum_{j \in [k_{i \, h_i}]}[E_{ij}].
\]
\end{lemma}
\begin{proof}
It is enough to consider the local case. The proof can be obtained by standard arguments similar to the one of \cite[Prop. 2.13]{EisHar:3264}. A more general statement is in \cite{PortelaChow}.
\end{proof}

Given $\kappa := \sum_{i \in [s]} k_{i \, h_{i}}$, the expression of $[\tilde H_\infty]$ in \Cref{lemma: intersections upstairs} yields the identities
\begin{equation}\label{eqn: intersection against H}
[\tilde H_\infty]^2 = 1-\kappa \quad \quad , \quad \quad \tilde{\zeta} \cdot [\tilde H_\infty] = 1 \quad  \quad , \quad \quad \alpha_{ij} \cdot [\tilde H_\infty] = \left\{\begin{array}{ll}
1 & \text{if $j = k_{i \, h_i}$} \\
0 & \text{otherwise}.
\end{array} \right.
\end{equation}

A consequence of \Cref{prop: resolution of phi} and \Cref{lemma: intersections upstairs} is the following result, which proves Item 2 of \Cref{thm: main forward} about the degree of $Y_L$ in Pl\"ucker embedding and yields closed formulas for the coefficients of $[Y_L]$ in \Cref{thm: schubert class}.

\begin{theorem}\label{theorem: degree of YL}
The degree of the Schubert surface $\YL \subset \Gr(n,V\oplus W)$ is 
\[
\deg(\YL) = ({n-\eta_{\#}})^2 - \sum_{i \in [s]} \sum_{j \in [k_{i \, h_i}]} \mu_{ij}^2 \ . 
\]
\end{theorem}
\begin{proof}
   Since the smooth resolution $\tilde\phi$ is defined by the complete linear system of the divisor $D$, the degree of $\YL$ coincides with the self-intersection of $D$:
   \[ \textstyle [D]^2 = (n-\eta_\#)^2[\tilde H]^2 - 2 \sum_i \sum_j (n-\eta_\#)\mu_{ij}[\tilde H] [E_{ij}] + \sum_{i_1,i_2}\sum_{j_1,j_2} \mu_{i_1 j_1}\mu_{i_2 j_2} [E_{i_1 j_1}][E_{i_2 j_2}] . \]
   The statement follows from the intersection products in \Cref{lemma: intersections upstairs}.
\end{proof}

\subsection{Exceptional curves and rational singularities}\label{subsec: exceptional curves}

In this section, we determine the image in $Y_L$ of the exceptional curves of the blow-up $\tilde{X}$ constructed in the previous section via $\tilde{\phi}: \tilde{X} \to Y_L$ and we characterise the singularities arising from this map. For $i\in [s]$ and $\ell \in [k_{i h_i}]$, recall that $A_{i\ell}$ denotes the exceptional divisor in $\tilde X$ of the $\ell$-th blow-up over the $i$-th irreducible component of $Z$.
\begin{proposition}\label{prop: contraction and embeddings}
    For $i\in [s]$, let $\ell \leq k_{h_i}$ be an integer. 
    \begin{itemize}
    \item If $\ell = k_{i j}$ for some $j$, then $\tilde{\phi}$ maps the exceptional divisor $A_{i \ell}$ to a rational normal curve of degree $e_{i j}$.
    \item If $k_{i j} < \ell < k_{i \, j+1}$, then $\tilde{\phi}$ contracts $A_{i \ell}$ to a point.
    \end{itemize}
\end{proposition}

\begin{proof}
Suppress $i$ from the notation. Let $k_1 \vvirg k_h$, $e_1 \vvirg e_h$ be the relevant regular invariants. Let $n' = n - \eta_{\#}$. For $\ell \in [k_h]$, let $E_\ell$ be the total transform in $\tilde{X}$ of the $\ell$-th exceptional divisor; in particular $E_\ell = A_\ell + \cdots + A_{k_h}$.
     
    By \Cref{prop: resolution of phi}, the map $\tilde{\phi}$ is given by the complete linear series of the divisor $D = (n-\eta_\#)\tilde H - \sum_{j \in [k_h]} \mu_{j} E_j$. From the intersection products in \Cref{lemma: intersections upstairs} we have
    \begin{align*}
        [A_{\ell}]\cdot [D] & = - \sum_{j \in [k_h]}  \mu_j [A_\ell]\cdot [E_j] = \mu_\ell - \mu_{\ell+1} = \begin{cases}
           e_{h'} &  \text{if } \ell = k_{h'} \text{ for some $h'$} \\
            0 & \text{otherwise}.
        \end{cases}
    \end{align*}
 Therefore, the image of $\tilde{\phi}|_{A_\ell}$ is a point if $k_{h'}< \ell < k_{h'+1}$, and it is a curve of degree $e_{j}$ if $\ell = k_j$. We show that in the latter case it is the rational normal curve of degree $e_j$. This is equivalent to showing that $A_{k_j}$ is mapped by the complete linear series of degree $e_j$. 

Let $X$ and $X'$ be respectively the $(\ell-1)$-th and the $\ell$-th blow-up in the considered chain. Let $z \in X$ be the center of the blow-up, and $A'$ the exceptional divisor. So $\pi: X' \to X$ maps $A'$ to $z$ and $A_\ell$ is the strict transform of $A'$ after the subsequent blow-ups. Abusing notation, let $\phi : X \dashto Y_L$ be the rational map extended to $X$ and $\phi' : X' \dashto Y_L$ its extension to $X'$. 

The base ideal of $\phi$ is $I = \prod_{\{j' \mid \ell \leq k_{j'} \}} (x_0,x_1^{k_{j'} - \ell+1})^{e_{j'}}$. By \Cref{lemma: transform one point}, $I$ has multiplicity $\mu_\ell$ at $z$: therefore, the linear series of $\phi'$ restricted to $A'$ is given by $(I \cap \frakm^{\mu_\ell}) / (I \cap \frakm^{\mu_\ell+1})$ where $\frakm$ is the maximal ideal supported at $z$. If $\ell = k_j$, then $I= (x_0,x_1)^{e_j} \cdot \prod_{\{j' \mid \ell < k_{j'}\}} (x_0,x_1^{k_{j'} - \ell+1})^{e_{j'}}$. So the component of degree $\mu_\ell$ is $x_0^{\mu_{\ell+1}} \cdot \bbC[x_0,x_1]_{e_j}$, which is the complete linear series of degree $e_j$ on $A' \simeq \bbP^1$.
\end{proof}

For $i \in [s]$ and $j \in [h_i]$, let
\begin{equation}\label{eqn: CijCurves}
    C_{ij}:=\tilde\varphi(A_{i \, k_{ij}})
\end{equation}
be the rational normal curves in $\YL$ of degree $e_{ij}$ from \Cref{prop: contraction and embeddings}. A representation of the mapping $\tilde{\phi}$ is given in \Cref{fig: mapping of exceptional divisors}.

\begin{figure}[H]
	\centering
	\begin{tikzpicture}[yscale=0.75]
		\draw[thick] (-1, 0.5) -- (1, 0.5) node[pos=0, below=5pt, font=\scriptsize] {$\tilde{H}_\infty$};
		\drawchain{0}{4}
		{{"B","G","R","K"}}
		{{"$A_4$","$A_3$","$A_2$","$A_1$"}}
		{{"G","","",""}}
		{{"0","0","0","0"}}{C}
		{{"0.75","0.75","0.75","0.75"}}
		
		\draw[->, thick] (1.5,2) -- (3.5,2);
		\node at (2.5,2.5) [inner sep=1pt]{$\tilde{\phi}$};
		
		\draw[thick] (4, 0.5) -- (6, 0.5) node[pos=0, below=5pt, font=\scriptsize] {$C_\infty$};
		\def\drawchainEvenYshift{-3pt}
		\drawchain{5}{3}
		{{"B","G","K"}}
		{{"$C_3$","$C_2$","$C_1$"}}
		{{"$p$","$q$","",""}}
		{{"0",".34","0"}}{C}
		{{"0.75","0.75",".9"}}
		
		\node at (5.35,2.35) [circle,inner sep=2pt, color={red!80!black}]{$\bullet$};
		\node at (5, 4.5) [font=\small\bfseries] {Singular};
	
		\node (w_reg) at (7.8, 0.5) [circle, fill=black, inner sep=1.5pt, label=below:{\scriptsize $[W]$}] {};
		\drawchain{7.5}{3}
		{{"B","G","K"}}
		{{"$C_3$","$C_2$","$C_1$"}}
		{{"$p'$","$q'$","",""}}
		{{"0",".34","0"}}{C}
		{{"0.75","0.75",".9"}}
		
		\node at (7.85,2.35) [circle,inner sep=2pt, color={red!80!black}]{$\bullet$};
		\node at (7.5, 4.5) [font=\small\bfseries] {Regular};
		\def\drawchainEvenYshift{0pt}
    \end{tikzpicture} 
	\caption{
		The map $\tilde{\phi}$ with indeterminacy locus $Z$, defined by the pencil invariants $(k_1,k_2,k_3) = (1,3,4)$, $(e_1,e_2,e_3) = (1,3,1)$. Here $\sum e_i k_i = 14$: $\tilde H_\infty$ is embedded into as $C_\infty$ if $L$ is singular, and it is and contracted to the point $[W]$ if $L$ is regular; $A_2$ is contracted to a point; $A_3$ is embedded as a rational normal curve of degree $3$.
	}\label{fig: mapping of exceptional divisors}
\end{figure}
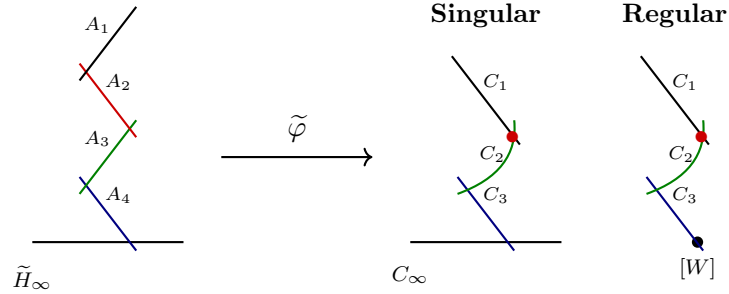

The contraction of the $(-2)$-curves described in \Cref{prop: contraction and embeddings} gives rise to singular points in $Y_L$. The next result characterises these singularities and proves Item 1 of \Cref{thm: main forward}.

\begin{theorem}\label{thm: rational singularities}
For a matrix pencil $L$, the associated Schubert surface $Y_L$ has only Du Val singularities of type $\mathsf{A}$ and, if $L$ is regular, possibly a quotient singularity at $[W]$ of type $\frac{1}{\kappa -1}(1,1)$. In particular, $Y_L$ is $\bbQ$-factorial and normal.
\end{theorem}

The proof of \Cref{thm: rational singularities} relies on \Cref{prop: W singularity} for the characterisation of the singularity at $[W]$, and on the following result, which provides the local model of the more general construction used in the proof. 
\begin{lemma}\label{lemma: single step singularity}
    Let $k,e$ be positive integers. $I = (x_0,x_1^k)^e \subseteq \bbC[x_0,x_1,x_2]$, $d \geq ke$ and let $N+1 = \dim I_d$. The rational map $\phi: \bbP^2 \dashto \bbP^N$ given by the linear series of $I_d$ is resolved by $k$ successive blow-ups, and the resolution $\tilde {\phi} : \tilde{X} \to \bbP^N$ contracts the chain of $k-1$ curves into an $\mathsf{A}_{k-1}$-singularity.
\end{lemma}
\begin{proof}
The map $\phi$ is a monomial map, therefore the closure of its image in $\bbP^N$ is a toric variety. On the open chart $\{x _2 \neq 0\}$, the map is given by all monomials in a suitable $2$-dimensional polytope. Since all two-dimensional polytopes are normal \cite[Cor. 2.2.19]{CLS:ToricVars}, we deduce that the image is normal. By \cite[Thm. 11.4.2]{CLS:ToricVars}, the singularity is rational. The discussion of \cite[Example 10.1.5]{CLS:ToricVars} shows that it is an $\mathsf{A}_{k-1}$-singularity.
\end{proof}

\begin{proof}[Proof of \Cref{thm: rational singularities}]
    Let $\phi : \bbP^2 \dashto Y_L$ be the rational map defined by the linear series $I(Z)_{n-\eta_{\#}}$, where $Z = Z_1 \cup \ldots \cup Z_s$ is the indeterminacy locus as in \Cref{theorem: linear series of YL}. Let $\pi : \tilde{X} \to \bbP^2$ be the resolution defined in \Cref{subsec: smooth resolution} and let $\tilde{\phi} : \tilde{X} \to Y_L$ be the regular map resolving $\phi$.

    By \Cref{prop: contraction and embeddings}, the map $\tilde{\phi}$ contracts subchains of exceptional curves to points. Fix an eigenvalue $\theta_i$, and suppress $i$ from the notation. Let $k_{1} < \cdots < k_{h}$ be the distinct sizes of the Jordan blocks, and let $A_{1} \vvirg A_{k_{h}}$ be the exceptional divisors corresponding to those blocks in $\tilde{X}$. Let $j$ be such that $k_{j} < k_{j+1}-1$, so that the subchain $A_{k_{j}+1} \vvirg A_{k_{j+1}-1}$ is contracted to a point: we show that such a point is a cyclic singularity of type $\mathsf{A}_{k_{j+1}-k_{j}-1}$. 

    To do this, we consider the parameterisation map extended to the $k_{j}$-th blow-up. Let $X'$ be the intermediate partial resolution obtained after the $k_{j}$-th blow-up. Use coordinates $x'_0,x'_1$ for the strict transform of $H_\infty$ and the exceptional divisor of $X'$ respectively. By \Cref{lemma: transform one point}, after the $k_{j}$-th blow-up, locally around the next center, the base ideal is 
    \[
    I = \prod_{j' > j} (x'_0,{x'}_1^{k_{j'}-k_j})^{e_{j'}} = (x'_0,{x'_1}^{k_{j+1}-k_j})^{e_j} \prod_{j' > j+1}(x'_0,{x'_1}^{k_{j'}-k_j})^{e_{j'}}.
    \]
    On the chart parameterised by $x'_0,x'_1$, the map $\phi' : X' \dashto Y_L$ is given by monomials in $I$ of degree at most $\delta$, where $\delta$ depends on the original degree $n - \eta_{\#}$ and on the multiplicities $\mu_1 \vvirg \mu_{k_j}$ from \Cref{prop: resolution of phi}. A direct calculation shows $\delta$ is at least as large as the largest degree of a minimal generator of $I$. We prove that the singularity of interest is equivalent to the singularity obtained by the linear series of $J = ({x'}_0,{x'}_1^{k_{j+1}-k_j})^{e_j}$. 

    To see this, let $X''$ be the result of the blow-up construction of $\bbP^2$ for $k_{j+1}$-times. Then the residual ideal $I''$ has support disjoint from the chain of rational curves that are contracted. In other words, locally around the chain of contracted curves, $X''$ is isomorphic to the variety obtained by resolving the linear series of $J$. Moreover, for a section $s$ on $X'$, the transform of $s$ on $X''$ only depends on the equations of the intermediate exceptional divisors, which are the same as those obtained from resolving $J$. 
    
    This shows that the singularity arising from the contraction is isomorphic to the singularity obtained by the ideal $J$. By \Cref{lemma: single step singularity}, this is a rational singularity of type $\mathsf{A}_{k_{j+1}-k_{j}-1}$. This shows that all singularities obtained from the contracted chains are Du Val, and in particular they are normal.

    \Cref{prop: W singularity} further shows that the singularity at $[W]$ is rational and this concludes the proof.
\end{proof}

\subsection{Intersection theory in Schubert surfaces}\label{subsec: intersection numbers}

In this section, we characterise the intersection product in the surface $Y_L$. By \Cref{thm: rational singularities}, $Y_L$ is $\bbQ$-factorial, and in particular there is a well-defined pairing on the group $\rmN_\bbQ^1(Y_L)$ of Weil divisors modulo numerical equivalence. In fact, in our setting, this group coincides with the middle cohomology group $H^1(Y_L,\bbQ)$ of $Y_L$. In particular, the results of this section characterise the intersection product in cohomology, yielding Item 4 of \Cref{thm: main forward}.

In $\rmN_\bbQ^1(\YL)$, denote by $\zeta := \tilde\varphi_\ast(\tilde \zeta)$ the image of the hyperplane class of $\mathbb P^2$, and by $\gamma_{ij}:=[C_{ij}]$ the class of the exceptional curve in $\YL$ obtained as image of the exceptional divisor $A_{i \, k_{ij}}\subset \tilde X$ from \Cref{prop: contraction and embeddings}.

\begin{theorem}\label{thm: intersections general}
Let $L = L_{\mathrm{sing}} \boxplus L_1 \boxplus \ldots \boxplus L_s$ be a pencil with $s$ regular components with distinct eigenvalues $\theta_1 \vvirg \theta_s$. For $i \in [s]$, let $k_{i 1} < \ldots < k_{i {h_i}}$ be the distinct sizes of the Jordan blocks with eigenvalue $\theta_i$. Then $\rmN^1_\bbQ(Y_L)$ is generated by the classes $\zeta$ and $\gamma_{ij}$ for $i \in [s]$ and $j \in [h_i]$. The intersection pairing is given as follows:
\begin{enumerate}
\item\label{item: products between different chains} for any two $i_1,i_2$ distinct, with $(j_1,j_2) \neq (h_{i_1},h_{i_2})$, we have
\[
\gamma_{i_1 j_1} \cdot \gamma_{i_2 j_2} = 0.
\]
\item\label{item: products along a chain} for every $i = 1 \vvirg s$ and every $(j_1,j_2) \neq (h_i,h_i)$, we have 
    \begin{align*}
    \zeta \cdot \gamma_{ij} & = 0 \text{ for every $j < h_i$} \\
    \gamma_{i j_1} \cdot \gamma_{i j_2} & = \begin{cases}
        - \frac{1}{k_{i j_1} - k_{i  j_1-1}} - \frac{1}{k_{i  j_1+1} - k_{i \, j_1}} & \text{if  $j_1=j_2$}  \\
        \frac{1}{k_{i \, j_1+1} - k_{i \, j_1}} & \text{if $j_2 = j_1+1$}\\
        0 & \text{if  $|j_1-j_2| > 1$}
        \end{cases}
\end{align*}
\item\label{item: products singular case} If $L$ is a singular pencil, then 
\[
\begin{array}{l}
    \zeta^2 = 1 \\
    \zeta \cdot \gamma_{i \, h_i} = 0
\end{array} \qquad \qquad 
\begin{array}{ll}
 \gamma_{i \, h_i}^2 =  - \frac{1}{k_{i \, h_i} - k_{i \, h_i-1}} \\
 \gamma_{i_1 h_{i_1}} \cdot \gamma_{i_2 h_{i_2}} = 0 & \text{ if $i_1 \neq i_2$}
\end{array}
\]
and
\[
[C_\infty] = \zeta - \textstyle  \sum_{i\in [s]} \sum_{j \in [h_i]} k_{i j} \gamma_{ij}.
\]
\item\label{item: products regular case} If $L$ is a regular pencil, given $\kappa := \sum_{i=1}^s k_{i \, h_i}$, then 
\[
\begin{array}{l}
    \zeta^2 = \frac{\kappa}{\kappa - 1} \\
    \zeta \cdot \gamma_{i \, h_i} = \frac{1}{\kappa -1 }
\end{array} \qquad \qquad 
\begin{array}{ll}
 \gamma_{i \, h_i}^2 = - \frac{1}{k_{i \, h_i} - k_{i \, h_i-1}} + \frac{1}{\kappa -1} \\
 \gamma_{i_1 h_{i_1}} \cdot \gamma_{i_2 h_{i_2}} = \frac{1}{\kappa -1} & \text{ if $i_1 \neq i_2$}
\end{array}
\]
and
\[
\zeta= \textstyle  \sum_{i\in [s]} \sum_{j \in [h_i]} k_{i j} \gamma_{ij}.
\]
\end{enumerate}
\end{theorem}

The proof relies on the fact that, since $\YL$ is $\bbQ$-factorial, the pullback of a Weil divisor is well defined. A curve $C \subseteq Y_L$ pulls back to a divisor  $\tilde\varphi^\ast C = \tilde C + \sum_i c_i F_i$ where $\tilde C$ is the strict transform of $C$ and $F_i$ are the irreducible components of the locus contracted by $\tilde{\phi}$. In particular, the push-pull formula \cite[Thm. 1.23]{EisHar:3264} provides the orthogonality condition
\[
\tilde\varphi^\ast C \cdot F_i = C \cdot \phi_* F_i = C \cdot 0 = 0  \quad \text{ for all $i$}. 
\]

The following two results will be useful to the proof of \Cref{thm: intersections general}.

\begin{lemma}[{\cite[Sec.\ 2.5, Eq.\ (4)]{martin2023normal}}]\label{lemma: product at quotient singularity}
Let $S$ be a $\bbQ$-factorial surface with a cyclic quotient singularity $p \in S$ of type $\frac{1}{d}(a,b)$, that is $(S,p)$ is locally isomorphic to $(\mathbb A^2/\mu_d, 0)$. Then, for any two $\bbQ$-divisors $D_1,D_2\subset S$ passing through $p$ and locally defined by two $\mu_d$-invariant polynomials $f_1,f_2 \in \bbC[\bbA^2]^{\mu_d}$, their intersection product at $p$ is the rational number
\[
    \textstyle (D_1 \cdot D_2)_p = \frac{1}{d}\dim_\bbC \frac{\mathbb C[[x,y]]}{(f_1,f_2)} \in \bbQ .
\]
\end{lemma}

We record a technical result which gives the fundamental step in the calculation required in the proof of \Cref{thm: intersections general}.

\begin{lemma}\label{lemma: mumford pullback of type A singularity}
Let $\tilde S$ be a smooth surface, let $S$ be a normal surface and let $f:\tilde S \to S$ be the contraction of a chain of $(-2)$-curves $E_1 \cup \ldots \cup E_{d-1}$ to a point $p \in S$. Let $C \subset S$ be a smooth curve passing through $p$ such that its strict transform $\tilde C = \overline{f^{-1}(C\setminus \{p\})}$ intersects either $E_1$ or $E_{d-1}$. Then the coefficients of the pullback ${\textstyle f^\ast C = \tilde C + \sum_{i \in [d-1]} c_i E_i}$ are
\begin{align*}
    c_i = \begin{cases}
    \frac{d-i}{d} & \text{if $\tilde C$  intersects $E_1$} \\
    \frac{i}{d} & \text{if $\tilde C$ intersects $E_{d-1}$} .
    \end{cases}
\end{align*}
\end{lemma}
\begin{proof}
Suppose $\tilde{C}$ intersects the chain at $E_1$. The case of $E_{d-1}$ is similar. Since $C$ is smooth, $f: \tilde{C} \to C$ is an isomorphism and we have $\tilde C \cdot E_i=\delta_{1i}$. We compute the coefficients in the pullback $f^\ast C = \tilde C + \sum c_i E_i$ by applying the orthogonality condition:
\[
    0 = f^\ast C \cdot E_\ell = \begin{cases}
    1-2c_1 + c_2 &  \text{if } \ell = 1 \\
    c_{\ell -1} -2 c_\ell + c_{\ell +1} &  \text{if } 2 \leq \ell \leq d-2 \\
    c_{d-2} - 2c_{d-1} & \text{if } \ell = d-1 .
    \end{cases}
\]
This yields an affine linear system, whose unique solution gives the values in the statement. More precisely, the second set of vanishings implies $c_{\ell} - c_{\ell-1} = c_{\ell+1}-c_\ell = c_1 -1$. Solving the recursion, starting from $c_1$, we obtain $c_{d-1} = c_1 + (d-2)(c_1-1)$. On the other hand, starting from $c_{d-1}$, we obtain $c_{d-1} = 1-c_1$. This implies $c_1 = \frac{d-1}{d}$ and applying the recursion we conclude $c_\ell = c_1 + (\ell-1)(c_1-1) = \frac{d-\ell}{d}$. 
\end{proof}

\begin{proof}[Proof of \Cref{thm: intersections general}]
    Let $\tilde{\phi}: \tilde{X} \to Y_L$ be the resolution of the rational map $\phi: \bbP^2 \dashto Y_L$ obtained by iterated blow-ups. We provide the expression of the pullback $\tilde{\phi}^* \gamma_{ij}$ in terms of the classes of the exceptional divisors of $\tilde{X}$. The generation statement follows from the fact that $\tilde\zeta$ and $\alpha_{ij}$ for $i\in [s],j\in[k_{i \, h_i}]$ generate $\rmN^1_\bbQ(\tilde X) = \CH^1_\bbQ(\tilde X)$, from \Cref{lemma: intersections upstairs}. 

    \Cref{item: products between different chains} is immediate: either two distinct chains do not intersect (in the singular case) or they only intersect at $[W]$ along the irreducible components $C_{i \, h_1}$ and $C_{i' \, h_{i'}}$ (regular case). Therefore $C_{ij} \cdot C_{i'j'} = 0$ if $i \neq i'$ and $(j,j') \neq (h_i,h_i')$

    The intersection of consecutive curves in the same chain is either smooth or a cyclic quotient singularity arising from the contraction of a chain of $(-2)$-curves, as in \Cref{lemma: mumford pullback of type A singularity}. Suppress $i$ from the notation. The multiplicity of the singularity is computed in the proof of \Cref{thm: rational singularities}: for two sizes $k_{j},k_{j+1}$ of the Jordan blocks, the point of intersection $p_j =  C_{j} \cap C_{j+1}$ is a singularity of $Y_L$ of type $\frac{1}{k_{j+1} - k_j}(1,-1)$, to be intended to be smooth if $k_{j+1} - k_j = 1$.  In particular, we have 
    \[
    \gamma_{j_1}\cdot \gamma_{j_2} = \left\{ \begin{array}{ll}
    0 & \text{if  $|j_1 - j_2| >1$} \\
    \frac{1}{k_{j_2} - k_{j_1}} &\text{if  $j_2 = j_1+1$}
    \end{array}\right.  
    \]
    where the value for $j_2 = j_1 + 1$ is computed using \Cref{lemma: product at quotient singularity}. This proves the first part of \Cref{item: products along a chain}.
    
    The self-intersections in \Cref{item: products along a chain} are computed by pulling back the classes of the exceptional curves to $\tilde{X}$. For $j < h$, each curve $C_j$ is smooth, and passes through at most two singular points of $Y_L$, which are cyclic quotient singularities arising from two disjoint contracted chains as described in \Cref{lemma: mumford pullback of type A singularity}. Since the strict transform of $C_j$ is $A_{k_j}$, we obtain  
    \[
\tilde{\phi}^* C_j = A_{k_j} + \sum_{\ell=k_{j-1}+1}^{k_j - 1} c_{j\ell} A_\ell + \sum_{\ell=k_{j}+1}^{k_{j+1} - 1} c_{j\ell} A_\ell 
    \]
    where the coefficients $c_{j\ell}$ are computed via \Cref{lemma: mumford pullback of type A singularity}.
    Applying the orthogonality condition and the intersection products in \Cref{lemma: intersections upstairs}, we obtain
    \begin{align*} 
        \gamma_{j}^2 & = [\tilde\phi^\ast C_{j}] \cdot [\tilde\phi^\ast C_{j}] = [\tilde\phi^\ast C_{j}] \cdot \alpha_{k_{j}} \\
        & = \alpha_{k_{j}}^2
        + {\textstyle \frac{k_{j}-k_{ j-1} - 1}{k_{j}-k_{ j-1}}} \alpha_{ k_{j}-1} \alpha_{ k_{j}}
        + {\textstyle \frac{k_{j+1}-k_{j}-1}{k_{ j+1}-k_{j}}} \alpha_{k_{j}+1}\alpha_{ k_{j}} = {\textstyle - \frac{1}{k_{j}-k_{j-1}} - \frac{1}{k_{j+1}-k_{j}}} . 
    \end{align*}
    We complete \Cref{item: products along a chain} observing that the condition $\zeta \cdot \gamma_{j}=0$ for all $j <h$ follows from the blow-up construction.

    Next, we address \Cref{item: products singular case}: in this case the pencil is singular. We apply \Cref{prop: image of phi}: $\tilde \varphi$ maps the divisor $\tilde H_\infty$ to a rational normal curve $C_\infty$, and the $s$ chains of exceptional curves $C_{i1}\cup \ldots \cup C_{i h_i}$ for $i \in [s]$ intersect $C_\infty$ at $s$ distinct points $C_{i,h_i} \cap C_\infty$. Therefore, 
    \[ 
    \gamma_{i_1  h_{i_1}}\cdot \gamma_{i_2 h_{i_2}}=0 \quad \forall i_1 \neq i_2 .
    \]
    To compute the self-intersections of the irreducible exceptional curves $C_{i h_i}$, we apply again \Cref{lemma: mumford pullback of type A singularity}. The curve $C_{i,h_i}$ has (at most) one singularity of $Y_L$ that is the intersection point with $C_{i \, h_i-1}$. Applying \Cref{lemma: mumford pullback of type A singularity}, and computing the intersection product, we obtain the desired value of $\gamma_{i,h_i}^2$.
    We complete \Cref{item: products singular case} by observing that the equations involving $\zeta$ are straightforward, and the expression of $C_\infty$ follows from the one of $\tilde H_\infty$ in \Cref{lemma: intersections upstairs} and from \Cref{prop: contraction and embeddings}:
    \[
    \textstyle [C_\infty] = \tilde\varphi_\ast[\tilde H_\infty] = \zeta - \sum_{i \in [s]} \sum_{j \in [k_{i h_i}]} j \tilde\varphi_\ast(\alpha_{ij}) = \zeta - \sum_{i \in [s]}\sum_{j \in [h_i]} k_{ij} \gamma_{ij} .
    \]
    
    Finally, we prove \Cref{item: products regular case}: in this case the pencil is regular. \Cref{prop: image of phi} shows that $\tilde\varphi$ contracts $\tilde H_\infty$ to the point $[W]$: in particular, all exceptional curves $C_{1 h_1}, \ldots, C_{s h_s}$ intersect at $[W]$. Therefore, for every $i \in [s]$ the pullback of $C_{i h_i}$ in $\tilde X$ is of the form
    \[ \textstyle 
    \tilde\varphi^\ast C_{i h_i} = A_{i \, k_{ih_i}} + \sum_{\ell =1}^{k_{ih_i} - k_{i \, h_i-1} -1} \frac{\ell}{k_{ih_i} - k_{i \, h_i-1}} A_{i , k_{i \, h_i-1}+\ell} + c_{i \infty }\tilde H_{\infty} .
    \]
    From the orthogonality condition and the intersection products in \eqref{eqn: intersection against H}, we obtain 
    \[ 0 = \tilde\phi^\ast C_{i h_i}\cdot \tilde H_\infty = A_{i \, k_{ih_i}}\cdot \tilde H_\infty + c_{i \infty}\tilde H_\infty ^2 = 1 + c_{i \infty}(1 - \kappa), \] 
    and therefore $\textstyle c_{i \infty} = \frac{1}{\kappa -1}$.

    The same computations as in the singular case show 
    \begin{align*}
        \gamma_{i \, h_i}^2 
        & = \alpha_{i \, k_{i h_i}}^2 + {\textstyle \frac{k_{ih_i} - k_{i \, h_i-1} -1}{k_{ih_i} - k_{i \, h_i-1}}} \alpha_{i , k_{ih_i}-1}\alpha_{i \, k_{i h_i}} + {\textstyle \frac{1}{\kappa -1}}[\tilde H_\infty]\cdot \alpha_{i \, k_{i h_i}} = - {\textstyle \frac{1}{k_{ih_i} - k_{i \, h_i-1}} + \frac{1}{\kappa -1} } .
    \end{align*}
    In a similar way, for $i_1 \neq i_2$ we have
    \begin{align*}
        \gamma_{i_1 h_{i_1}}\cdot \gamma_{i_2 h_{i_2}} & = {\textstyle \frac{1}{\kappa -1}} [\tilde H_\infty]\cdot \alpha_{i_1 , k_{i_1 h_{i_1}}} +
        {\textstyle \frac{1}{\kappa -1}} [\tilde H_\infty]\cdot \alpha_{i_2 , k_{i_2 h_{i_2}}} +
        {\textstyle \frac{1}{(\kappa -1)^2}} [\tilde H_\infty]^2 = {\textstyle \frac{1}{\kappa -1}} ,\\
        \zeta \cdot \gamma_{i h_i} & = \tilde \zeta \cdot [\tilde\phi^\ast C_{i h_i}] = {\textstyle \frac{1}{\kappa -1}}[\tilde H_\infty]\cdot \tilde \zeta = {\textstyle \frac{1}{\kappa -1}} .
    \end{align*}
    The expression of $\zeta$ in terms of the classes $\gamma_{ij}$'s is obtained by the pushforward of the expression of $[\tilde H_\infty]$ in \Cref{lemma: intersections upstairs}, which vanishes since $\tilde H_\infty$ contracts to $[W]$. Finally, we have $\tilde\varphi^\ast(\zeta)= \tilde \zeta + \frac{1}{\kappa -1}[\tilde H_\infty]$ and 
    \[ \zeta^2 = \tilde\varphi_\ast(\tilde \zeta)\cdot \zeta = \tilde \zeta \cdot \tilde\phi^\ast(\tilde\zeta) = 1 + {\textstyle \frac{1}{\kappa -1} = \frac{\kappa}{\kappa -1} } . \qedhere
    \]
    \end{proof}

A fundamental step toward the proof of \Cref{thm: reconstruction of invariants} is that the rational normal curves $C_{ij} \subseteq Y_L$ are, essentially, uniquely determined by their classes $\gamma_{ij} \in \rmN^1_{\bbQ}(Y_L)$:
\begin{corollary}\label{corol: uniqueness exceptional curves}
Let $C_{ij} \subseteq Y_L$ be the rational normal curve as defined in \ref{eqn: CijCurves}. Then the following holds:
\begin{itemize}
\item If $L$ is a singular pencil, then $C_{ij}$ is the unique irreducible curve representing the class $\gamma_{ij} = [C_{ij}] \in \rmN^1_{\bbQ}(Y_L)$.
\item If $L$ is a regular pencil and $j < h_i$, then $C_{ij}$ is the unique irreducible curve representing the class $\gamma_{ij} = [C_{ij}] \in \rmN^1_{\bbQ}(Y_L)$. 
\item If $L$ is regular and $j = h_i$, then $C_{i \, h_i}$ is the unique irreducible curve representing the class $\gamma_{i \, h_i} = [C_{i \, h_i}] \in \rmN^1_\bbQ(Y_L)$ and passing through $[W]$.
\end{itemize}
\end{corollary}
\begin{proof}
The proof is immediate whenever $\gamma_{ij}^2 < 0$. Indeed, in this case, if $D$ is any other effective representative of the class $\gamma_{ij}$, then 
\[
0> \gamma_{ij}^2 = [D] \cdot [C_{ij}] = \deg (D \cap C) \geq 0
\]
showing that $D$ cannot exist. This proves the first and second statement.

If $D_1,D_2$ are two distinct irreducible curves passing through $[W]$, then $(D_1 \cdot D_2)_{[W]} \geq \frac{1}{\kappa - 1}$: this is a consequence of \Cref{lemma: product at quotient singularity} and \Cref{prop: W singularity}. Since $\gamma_{i \, h_i}^2 = \frac{1}{\kappa -1} - \frac{1}{k_{i \, h_i} - k_{i \, h_i-1}} < \frac{1}{\kappa - 1}$, we conclude that $C_{i \, h_i}$ is the unique irreducible curve representing the class $\gamma_{i \, h_i}$ and passing through $[W]$.
\end{proof}

\subsection{Normality properties in higher dimensional Schubert varieties}\label{sec: higher normality}
\Cref{prop: resolution of phi} and \Cref{thm: rational singularities} show that Schubert varieties of matrix pencils are linearly normal in the Pl\"ucker embedding and normal as projective varieties. The proofs fundamentally use the fact that $L$ is a matrix pencil and, in particular, that one can completely control the indeterminacy locus of the rational map $\phi$. Hence, it is natural to ask whether these normality properties extend to higher dimensional Schubert varieties of linear matrix spaces. For instance, classical Schubert varieties are normal \cite{MS87} and projectively normal in their Pl\"ucker embedding \cite{RR85}. The same two properties hold for matroid Schubert varieties: the multidegree characterisation of \cite{ardila2016closure} shows that matroid Schubert varieties are multiplicity free in the Segre variety $(\mathbb{P}^1)^{\times n}$, in the sense of \cite{Brion03}. In this setting, \cite[Thm.\ 1]{Brion03} guarantees that they are normal, Cohen-Macaulay, and projectively normal in the Segre embedding.

Instead, higher dimensional Schubert varieties of linear matrix spaces are, in general, neither normal nor linearly normal. We show this with two examples.

\begin{example}\label{example: counterexample normality}
Let $L$ be a general $4$-dimensional linear subspace of $V^* \otimes W$ for $\dim V = \dim W = 3$. The Schubert variety of $L$ is a $4$-fold $Y_L \subseteq \Gr(3, V\oplus W)$. The genericity assumption guarantees $L$ contains full rank matrices, so $[W] \in Y_L$. We show that $[W]$ is a non-normal singularity. 

By \Cref{prop: W singularity}, the singularity is isomorphic to the singularity of the affine cone over the collineation variety $\scrC_2(L)$. From \cite[Example 6.1]{GeKe:collineations} we know that in this case $\scrC_2(L)$ is a generic projection of the Veronese variety $v_2(\bbP^3)$; in particular, it is not linearly normal.

The singularity at the vertex of a cone is normal if and only if the projectivised cone is projectively normal as a projective variety \cite[Ex. 5.14]{Hartshorne}. In this case, the projectivised cone of $\YL$ at $[W]$ is not linearly normal, hence not projectively normal and we obtain that $[W]$ is a non-normal point of $Y_L$.
\end{example}

\begin{example}\label{example: counterexample linear normality}
    Let $L \subseteq V^* \otimes W$ be a general $6$-dimensional linear space, with $\dim V = 3$ and $\dim W = 4$. Then $Y_L$ is a $6$-dimensional subvariety of $\Gr(3, V\oplus W)$. We show that it is not linearly normal. 
    
    The linear series defining $\phi: \bbP (\bbC \oplus L) \dashto \Gr(3, V\oplus W) \subseteq \bbP \Wedge^3 (V \oplus W)$ is 
    \begin{equation}\label{eqn: J3 as minors}
    J_3 = D_3(L) + x_0 D_2(L) + x_0^2 D_1(L) + \langle x_0^3 \rangle
    \end{equation}    
    which has dimension (at most) $\dim J_3 = 4 + \binom{3}{2}\binom{4}{2} + 6 + 1 = 29$.

    The variety of rank one elements in $V^* \otimes W$ has codimension $6$, so the indeterminacy locus of $\phi$ is the zero locus $B = \calZ( D_3(L)) \subseteq \bbP L$: it is smooth and only consists of elements of rank $2$. In particular, $\tilde{X} = \Bl_B \bbP(\bbC \oplus L)$ is a resolution of the indeterminacy of $\phi$. The linear series $J_3$ pulls back to the divisor $3 H - E_B$ in $\tilde X$, where $H$ is the pull-back of $\mathcal{O}_{\PP^2}(1)$ and $E_B$ is the exceptional divisor over $B$. Its complete linear series is $H^0(\tilde{X}, 3H- E_B) = H^0(\bbP(\bbC \oplus L), \calI_B(3)) = I(B)_3$, and we have 
    \[
   I(B)_3 = (x_0, D_3(L)) = x_0 \cdot \bbC[x_0, \ldots, x_5]_2 + D_3(L) 
    \]
    which has dimension $32$. This shows that the regular map $\tilde{\phi}$ is defined by a $29$-dimensional subseries of a $32$-dimensional complete linear series of the divisor $3H -E_B$. We argue that $H^0(\calO_{Y_L}(1))$ coincides with this linear series as well; in particular, we will obtain that $Y_L$ is not linearly normal. To see this, one needs to show that all elements of $I(B)_3$ define a section of $\calO_{Y_L}(1)$. This amounts to showing that $I(B)_3 \bbC[J_3] \subseteq \bbC[J_3]$, where $\bbC[J_3]$ denotes the algebra generated by $J_3$. This can be verified with an explicit calculation: for every $q \in \bbC[x_1 \vvirg x_5]$, and every $f \in J_3$, one can show that $x_0qf \in J_3^2$. The calculation is a case-by-case analysis depending on which of the four summands in \eqref{eqn: J3 as minors} $f$ belongs. We omit the details of the proof, which is lengthy but not difficult. For any given explicit example, all details can be verified computationally via a computer algebra software. 
\end{example}

\subsection{The eigenvalue configuration}\label{subsec: eigenvalue configuration}

In this last section, we record a consequence of \Cref{prop: W singularity}. The support of the indeterminacy locus $Z$ defines, on $H_\infty = \bbP L \subseteq \bbP^2$, a configuration of $s$ points $[0:X_i]$. This configuration carries over to the surface $Y_L$: we record two results showing this fact. The situation for singular pencils follows directly from \Cref{prop: image of phi}.
\begin{proposition}\label{prop: eigenvalue configuration singular}
Let $L$ be a singular pencil and let $\theta_1 \vvirg \theta_s$ be its distinct generalised eigenvalues. Let $w_i \in \bbP^2$ be the point satisfying $x_0=\theta_i = 0$. Then the regular restriction $\phi|_{H_\infty} : H_\infty \to Y_L$ surjects onto $C_\infty$ and the points $\phi|_{H_\infty}(w_1) \vvirg \phi|_{H_\infty}(w_s)$ are the intersection points of $C_\infty$ with the curves $C_{i \, h_i}$ for $i \in [s]$.
\end{proposition}
\begin{proof}
This is straightforward. The restriction of the resolution $\tilde{\phi} : \tilde{X} \to Y_L$ to the strict transform $\tilde{H}_\infty$ of $H_\infty$ coincides with the natural extension of the restriction of $\phi$ to $H_\infty$, as described in \Cref{prop: image of phi}. In particular, the intersection point between the exceptional divisor $A_{i \, h_i}$ and the strict transform $\tilde{H}_\infty$ is mapped to the same point as $w_i$. Since $C_{i \, h_i} = \tilde{\phi}(A_{i \, h_i})$ and $C_\infty = \tilde{\phi}(\tilde{H}_\infty)$, the result follows.
\end{proof}

The case of regular pencils is slightly more subtle and makes use of \Cref{prop: W singularity} and the map $\psi : H_\infty \to \bbP T_{[W]} Y_L$ defining the $(\kappa-1)$-th Veronese embedding of $H_\infty$ onto the projectivised tangent cone $\bbP \TC_{[W]}Y_L$.
\begin{proposition}\label{prop: eigenvalue configuration regular}
Let $L$ be a regular pencil and let $\theta_1 \vvirg \theta_s$ be its distinct generalised eigenvalues. Let $w_i \in \bbP^2$ be the point satisfying $x_0=\theta_i = 0$. Then the map $\psi : H_\infty \to \bbP T_{[W]} Y_L$ maps  $w_i$ to the point defined by the tangent line $\bbP (T_{[W]} C_{i \, h_i})$. 
\end{proposition}
\begin{proof}
By \Cref{prop: W singularity}, the map $\psi$ is the $(\kappa-1)$-th Veronese embedding of $H_\infty$ and it maps elements $[0:X] \in H_\infty$ to a point of $\bbP \TC_{[W]} Y_L$ uniquely determined by the tangent vector at $[W]$ of a curve $\phi(D) \subseteq Y_L$ with  $D$ transverse to $H_\infty$ and passing through $[0:X]$. Let $\pi : \tilde{X} \to \bbP^2$ be the blow-up and let $\tilde{\psi} = \psi \circ \pi|_{\tilde{H}}$. The same argument as in \Cref{prop: W singularity} guarantees $\tilde{\psi}$ maps the point of intersection $\tilde{w}_i = \tilde{H}_\infty \cap A_{i \, k_{i \, h_i}}$ to the point of $\bbP T_{[W]} Y_L$ defined by the tangent line of any curve $\tilde{\phi}(D)$ where $D$ is a curve in $\tilde{X}$ transverse to $\tilde{H}_\infty$ and passing through $\tilde{w}_i$. Choosing $D = A_{i \, k_{i \, h_i}}$ yields the desired result.
\end{proof}

\section{Recovering pencil invariants from the Schubert surface}\label{sec: recovering}

This section is devoted to the proof of \Cref{thm: reconstruction of invariants}. We prove that the regular component of a pencil $L$ is uniquely determined by the Schubert variety $Y_L$. In fact, we provide a reconstruction procedure which can be turned into an actual algorithm, using the ideas of \cite{Lazic24}. The procedure is based on the fact that the point $[W] \in Y_L$ in the regular case or the rational normal curve $C_\infty$ in the singular case can be characterised uniquely in $Y_L$, unless there are very few generalised eigenvalues. If there are few generalised eigenvalues, then we perform a case-by-case analysis based on the configuration of exceptional curves and singularities. An outline of the procedure is given in \Cref{fig: reconstruction-flowchart}.

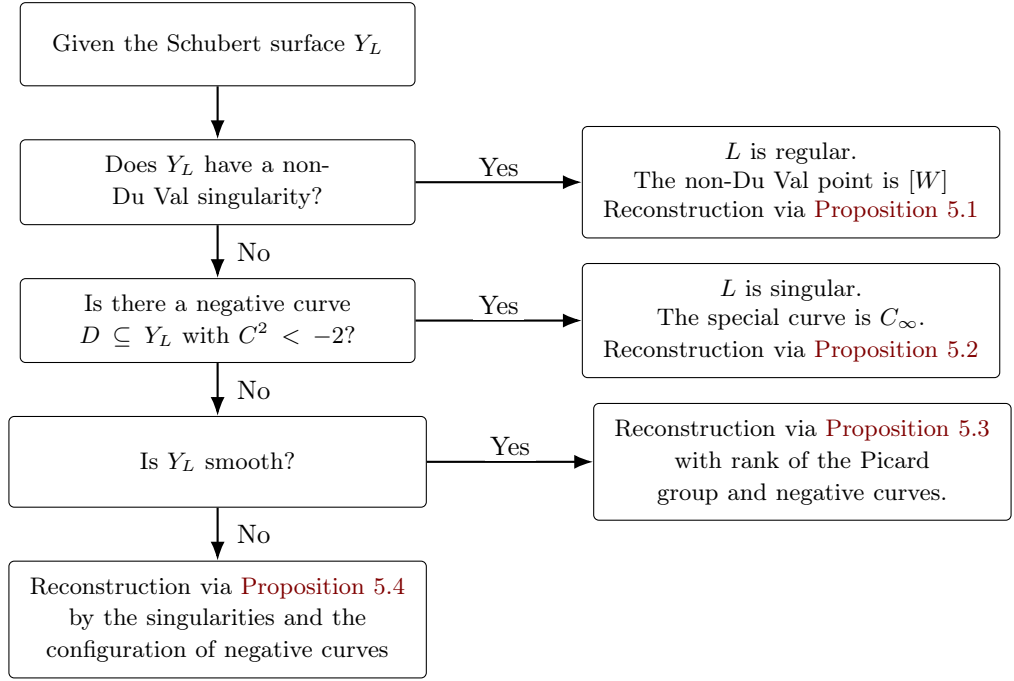
\begin{figure}[H]
\centering
\begin{tikzpicture}[
    node distance=0.7cm and 2.2cm,
    every node/.style={
        font=\small,
        align=center
    },
    box/.style={
        rectangle,
        draw,
        rounded corners=2pt,
        minimum width=5.2cm,
        minimum height=1.1cm,
        text width=4.8cm,
        inner sep=6pt
    },
    result/.style={
        rectangle,
        draw,
        rounded corners=2pt,
        minimum width=5.5cm,
        minimum height=1.2cm,
        text width=5.1cm,
        inner sep=6pt
    },
    arrow/.style={
        -{Latex[length=2.5mm]},
        thick
    },
    label/.style={
        font=\small,
        fill=white,
        inner sep=2pt
    },
    scale=0.25
]

\node[box] (start) {\footnotesize
    Given the Schubert surface $Y_L$
};

\node[box, below=of start] (nonduval) {\footnotesize
    Does $Y_L$ have a non-Du Val singularity?
};

\node[result, right=of nonduval] (reconstructW) {\footnotesize
    $L$ is regular.\\
    The non-Du Val point is $[W]$\\
    Reconstruction via \Cref{prop: W non DuVal singularity}\\
};

\node[box, below=of nonduval] (smallcurve) {\footnotesize
    Is there a negative curve $D\subseteq Y_L$ with $C^2  <-2$?
};

\node[result, right=of smallcurve] (reconstructCinfty) {\footnotesize
    $L$ is singular.\\
    The special curve is $C_\infty$.\\
    Reconstruction via \Cref{prop: Cinfty very negative}
};

\node[result, below=of smallcurve] (smoothquestion) {\footnotesize
    Is $Y_L$ smooth?
};

\node[result, right=of smoothquestion] (smoothanswer) {\footnotesize
Reconstruction via \Cref{prop: smooth YL} \\
with rank of the Picard group and negative curves.
};

\node[result, below=of smoothquestion] (cases) {\footnotesize
Reconstruction via \Cref{prop: reconstruct case by case} \\
by the singularities and the configuration of negative curves
};

\draw[arrow] (start) -- (nonduval);

\draw[arrow] (nonduval) -- node[label, above] {Yes} (reconstructW);
\draw[arrow] (nonduval) -- node[label, right=5pt] {No} (smallcurve);

\draw[arrow] (smallcurve) -- node[label, above] {Yes} (reconstructCinfty);
\draw[arrow] (smallcurve) -- node[label, right=5pt] {No} (smoothquestion);

\draw[arrow] (smoothquestion) -- node[label, above] {Yes} (smoothanswer);
\draw[arrow] (smoothquestion) -- node[label, right=5pt] {No} (cases);

\end{tikzpicture}

\caption{Flowchart of the procedure to reconstruct the regular part $L_\bullet^{\text{reg}}$ from $\YL$.}
\label{fig: reconstruction-flowchart}
\end{figure}

The first two results of this section uniquely reconstruct the pencil $L$ under suitable assumption on the singularity at $[W]$ or on the self-intersection of the rational normal curve $C_\infty$. Throughout this section, we say that a curve $C \subseteq Y_L$ is \emph{negative} if $C^2 < 0$ and it is \emph{non-positive} if $C^2 \leq 0$ in $\rmN^1_{\bbQ}(Y_L)$. By \Cref{prop: image of phi} and \Cref{prop: contraction and embeddings}, all special curves we will consider in this section are rational normal curves in the Pl\"ucker embedding of $Y_L$; when we refer to their degree, we intend the degree in this embedding.

We recall that $\kappa := \sum_{i \in [s]} k_{i \, h_{i}}$ is the sum of the maximal sizes of Jordan blocks, one for each eigenvalue.

\begin{proposition}\label{prop: W non DuVal singularity}
Let $L$ be a pencil of size $n$. Assume $Y_L$ has a singular point $q_W$ that is a quotient singularity of type $\frac{1}{\kappa -1} (1,1)$ not of Du Val type. Then $\kappa \geq 4$, $L$ is regular and one of the following holds: \vspace{1ex}
\begin{enumerate}[leftmargin=*]
\item There are no non-positive curves passing through $q_W$. Then $L$ has a single eigenvalue, and there is an integer $e$ such that $\kappa e = n$ and $L = J_\kappa(x_1)^{\boxplus e}$. \vspace{1ex}
\item There are exactly $s$ non-positive curves $D_1 \vvirg D_s$ passing through $q_W$. In this case $L$ has exactly $s$ distinct eigenvalues $\theta_1 \vvirg \theta_s$; up to reordering, the configuration $(\theta_1 \vvirg \theta_s)$ in $\bbP^1$ is $\GL_2$-equivalent to the configuration of the tangent lines $T_{q_W}D_1 \vvirg T_{q_W}D_s$ in $\bbP(\TC_{[W]}Y_L)$, which is a rational normal curve of degree $\kappa-1$. Moreover, for every $i \in [s]$, $D_i$ is connected to a chain of $h_i-1$ other negative curves, $C_{i1}, \ldots, C_{i \, h_i-1}, C_{i \, h_i}=D_i$. Let $e_{ij} = \deg C_{ij}$ and $c_{ij} = [C_{ij}]^2$. Define recursively 
\begin{align*}
u_{i \, h_i} := {\textstyle \frac{1}{\kappa-1}} - c_{i \, h_i}  \quad , \quad u_{ij} := - c_{ij} - u_{i \, j+1} 
\end{align*}
and set $k_{i j} := \sum_{\ell=1}^{j} \frac{1}{u_{i \ell}}$. Then $k_{i j}$ are integers and for $L = L_1 \boxplus \ldots \boxplus L_s$, we have
\[
L_i = J_{k_{i 1}}(\theta_i)^{\boxplus e_{i 1}} \boxplus J_{k_{i 2}}(\theta_i)^{\boxplus e_{i 2}} \boxplus \ldots \boxplus J_{k_{i \, h_i}}(\theta_i)^{\boxplus e_{i h_i}}.
\]
\end{enumerate} 
\end{proposition}
\begin{proof}
By \Cref{thm: rational singularities}, all singularities of $Y_L$ are of Du Val type with the only possible exception of $[W]$ which is a quotient singularity of type $\frac{1}{\kappa-1}(1,1)$, as in \Cref{prop: W singularity}. Therefore, the existence of a non-Du Val singularity guarantees that $L$ is regular and that $\kappa \geq 4$. Moreover, $[W] = q_W$ and the projectivised tangent cone $\bbP \TC_{[W]}Y_L$ is a rational normal curve of degree $\kappa-1$.

By \Cref{thm: intersections general}, if there are at least two distinct eigenvalues, at least one of the divisors corresponding to Jordan blocks of maximal size is a non-positive curve passing through $[W]$. Similarly, if there is a single eigenvalue, the divisor of the maximal size Jordan block is negative unless either there is a unique block size, that is $h=1$, in which case the self-intersection is positive, or $h=2$ and $k_1=1$, in which case the self-intersection is $0$. This guarantees that if there are no non-positive curves passing through $[W]$ then all Jordan blocks have the same size and Case (1) holds.

If $L$ has more than one generalised eigenvalue or if there is one eigenvalue with multiple block sizes, then all exceptional curves passing through $[W]$ are non-positive by \Cref{thm: intersections general} and they are unique representatives of their own class by \Cref{corol: uniqueness exceptional curves}. The configuration of eigenvalues is determined by the tangent lines of such curves by \Cref{prop: eigenvalue configuration regular}. All curves along each chain are the unique representatives of their own classes by \Cref{corol: uniqueness exceptional curves}. The degrees of these curves determine the multiplicities of the Jordan blocks by \Cref{prop: contraction and embeddings}. The sizes of the blocks are determined by inverting the formulas for the intersection numbers and the result of such inversion is as described in the statement.
\end{proof}

\Cref{prop: W non DuVal singularity} has a counterpart for singular pencils, as follows.

\begin{proposition}\label{prop: Cinfty very negative}
Let $L$ be a pencil of size $n$. Assume $Y_L$ has a negative curve $C_\infty$ with self-intersection $1-\kappa < -2$. Then $\kappa \geq 4$, and $L$ is singular. There are exactly $s$ negative curves $D_1 \vvirg D_s$ intersecting $C_\infty$. In this case the regular component $L^{\reg}$ of $L$ has exactly $s$ distinct eigenvalues $\theta_1 \vvirg \theta_s$. The configuration $(\theta_1 \vvirg \theta_s)$ in $\bbP^1$ is $\GL_2$-equivalent to the configuration of intersection points $(D_i \cap C_\infty : i \in [s])$ on $C_\infty$.  Moreover, for every $i \in [s]$, $D_i$ is connected to a chain of $h_i-1$ other negative curves, $C_{i1}, \ldots , C_{i \, h_i-1}, C_{i \, h_i} = D_i$. Let $e_{ij} = \deg C_{ij}$ and $c_{ij} = [C_{ij}]^2$. Define recursively 
\begin{align*}
u_{i \, h_i} :=  
- c_{i \, h_i}  \quad , \quad u_{ij} := - c_{ij} - u_{i \, j+1} 
\end{align*}
and set $k_{i j} := \sum_{\ell=1}^{j} \frac{1}{u_{i \ell}}$. Then $k_{i j}$ are integer and for $L^{\reg} = L_1 \boxplus \ldots \boxplus L_s$, we have
\[
L_i = J_{k_{i 1}}(\theta_i)^{\boxplus e_{i 1}} \boxplus J_{k_{i 2}}(\theta_i)^{\boxplus e_{i 2}} \boxplus \ldots \boxplus J_{k_{i \, h_i}}(\theta_i)^{\boxplus e_{i h_i}}.
\]
\end{proposition}
\begin{proof}
By \Cref{thm: intersections general}, all the negative curves arising from the chains of exceptional curves have self-intersection $C_{ij}^2 \geq -2$. Therefore, $C_\infty$ is the only negative curve in $Y_L$ with self-intersection $C_\infty^2 = 1-\kappa < -2$. 

The negative curves intersecting $C_\infty$ are uniquely determined by \Cref{corol: uniqueness exceptional curves} and the intersection points of the negative curves with $C_\infty$ determine the eigenvalue configuration of $L^{\reg}$ by \Cref{prop: eigenvalue configuration singular}. The degrees of the curves determine the multiplicities of the Jordan blocks by \Cref{prop: contraction and embeddings}. The sizes of the blocks are determined by inverting the formulas for the intersection numbers and the result of such inversion is as described in the statement.
\end{proof}

We are left to analyze the cases where $Y_L$ has at worst Du Val singularities and negative curves with self-intersection at least $-2$. In this case, we perform a case-by-case analysis. The configuration of exceptional curves and the singularities of $Y_L$ completely determine the regular component of $L$. However, the classification when $Y_L$ is smooth becomes quite involved and we treat it separately.

\begin{proposition}\label{prop: smooth YL}
Let $L$ be a pencil of size $n$. If $Y_L$ is smooth then one of the following occurs:
\begin{enumerate}[leftmargin=*]
    \item $Y_L$ has no $(-1)$-curves. In this case, one of the following occurs:
\begin{enumerate}[label=(\roman*),leftmargin=*]
    \item $Y_L$ is abstractly isomorphic to $\bbP^2$. Then $L$ has no regular component and $Y_L$ is a Veronese variety. In this case, $Y_L$ has no non-positive curves.
    \item $Y_L$ is abstractly isomorphic to $\bbP^1 \times \bbP^1$. Then $L$ is regular with exactly two distinct eigenvalues and blocks of size $1$. The surface $Y_L$ is embedded by $\calO_{\bbP^1 \times \bbP^1}(e_1,e_2)$ if and only if $L = J_{1}(x_1)^{\boxplus e_1} \boxplus J_{1}(x_2)^{\boxplus e_2}$. In this case $Y_L$ has no negative curves.
    \item $Y_L$ is abstractly isomorphic to the Hirzebruch surface $\bbF_2 = \bbP(\calO_{\bbP^1} \oplus \calO_{\bbP^1}(2))$. Then $L$ is regular with exactly one eigenvalue and blocks of size $1$ and $2$. Let $E$ be the class of the base of $\bbF_2$ and $F$ the class of a fiber: the surface $Y_L$ is embedded by $(e_1 + 2e_2)F + e_2E$ if and only if $L = J_1(x_1)^{\boxplus e_1} \boxplus J_2(x_1)^{\boxplus e_2}$. In this case $E$ is the only negative curve of $Y_L$.
\end{enumerate} 
\item $Y_L$ is abstractly isomorphic to an iterated blow-up of $\bbP^2$ at $s$ $0$-dimensional local schemes supported on the same line. It has $s$ pairwise disjoint curves with self-intersection $-1$, $D_1 \vvirg D_s$. For every $i \in [s]$, $D_i$ is connected to a chain of $h_i$ $(-2)$-curves, $D_i = C_{i \, h_i} \vvirg C_{i1}$. Let $e_{ij} = \deg C_{ij}$. Let $\kappa := \sum_{i \in [s]} h_i$. The curves $D_i$ all intersect a curve $D_\infty$, and $D_\infty^2 = 1-\kappa$. In this case $L$ is singular, its regular component $L^{\reg} = L_1 \boxplus \ldots \boxplus L_s$ has exactly $s$ distinct eigenvalues $\theta_1 \vvirg \theta_s$, and we have
\[
L_i = J_{1}(\theta_i)^{\boxplus e_{i1}} \boxplus J_{2}(\theta_i)^{\boxplus e_{i2}} \boxplus \ldots \boxplus J_{h_i}(\theta_i)^{\boxplus e_{ih_i}}. 
\] 
Moreover, if $\kappa \geq 2$, the configuration on $D_\infty$ of the points of intersection with the $(-1)$-curves $D_1 \vvirg D_s$ is $\GL_2$-equivalent to the eigenvalue configuration $(\theta_1 \vvirg \theta_s)$ in $\bbP^1$.\\

\end{enumerate}
\end{proposition}
\begin{proof}
Let $\kappa$ be the sum of the maximal sizes of the Jordan blocks. If $L$ is regular, \Cref{prop: W singularity} guarantees that $\kappa \leq 2$, otherwise $[W]$ is a singular point of $Y_L$. Moreover, \Cref{thm: rational singularities} guarantees that for each eigenvalue, all integers from $1$ to $h_i$ occur as sizes of a Jordan block, otherwise the contraction of a $(-2)$-curve would give rise to a singularity. 

Suppose $L$ is regular. If $L$ has two distinct eigenvalues, then the discussion above guarantees 
\[
L = J_1(x_1)^{\boxplus e_{1 , 1}} \boxplus J_1(x_2)^{\boxplus e_{2 , 1}}
\]
 with $e_{1,1},e_{2,1} > 0$. This gives rise to the case (ii) of part 1. If $L$ has a single eigenvalue, the discussion above guarantees 
 \[
 L = J_1(x_1)^{\boxplus e_1} \boxplus J_2(x_1)^{\boxplus e_2}
 \]
  with $e_1,e_2 > 0$. This gives rise to the case (iii) of part 1. 

If $L$ is singular with trivial regular component, then $Y_L$ is a Veronese surface. This is case (i) of part 1. 

Finally, if $L$ is singular and it has a nontrivial regular component, then by the smoothness assumption of $Y_L$, the extension $\tilde{\phi} : \tilde{X} \to Y_L$ constructed in \Cref{subsec: intersection numbers} is an isomorphism. By \Cref{lemma: intersections upstairs}, there are as many $(-1)$-curves as the number of distinct eigenvalues of $L$, the eigenvalue configuration defines the intersection points on $C_\infty$ as explained in \Cref{prop: eigenvalue configuration singular} and the degrees of the exceptional curves coincide with the multiplicities of the corresponding Jordan blocks as explained in \Cref{prop: contraction and embeddings}. This is case (2) of the statement.
\end{proof}

We point out that there is some overlap between the cases analyzed in \Cref{prop: Cinfty very negative} and the ones in part (2) of \Cref{prop: smooth YL}. We are left to analyze the case where $Y_L$ is singular, with only Du Val singularities and no curves with self-intersection smaller than $-2$. In this case, we perform a case-by-case analysis. Given two integers $a,b$, we use the notation $(a^b)$ to indicate a sequence $(a \vvirg a)$ with $a$ occurring $b$ times; if $b = b_1/b_2$ is a fraction, we assume $b_1$ is divisible by $b_2$.

\begin{proposition}\label{prop: reconstruct case by case}
Let $L$ be a pencil of size $n$. Assume $Y_L$ is singular with only Du Val singularities and it has no curves $C$ with $C^2 < -2$. Then $Y_L$ has at most three negative curves and one of the following occurs:
\begin{enumerate}[leftmargin=*]
\item $Y_L$ has no negative curves. Then $L$ is regular, has a unique eigenvalue, and the singularities determine the structure of the Jordan blocks:
\begin{center}
\vspace{1ex}
{\small \begin{tabular}{c|c}
    singularities &  Jordan sizes and multiplicities \\ \hline
    one $\mathsf{A}_1$ point &   $(k_1^{e_1}) = (2^{n/2})$ \\ 
    one $\mathsf{A}_1$ point and one $\mathsf{A}_2$ point  &  $(k_1^{e_1}) = (3^{n/3})$
\end{tabular} }
\vspace{1ex}
\end{center}
~
\item $Y_L$ has exactly one negative curve $D$. Then $L$ is regular, and the self-intersection of the unique negative curve determines the structure of the Jordan blocks as follows
\begin{center}
\vspace{1ex}
{\small \begin{tabular}{c|c|c}
    $D^2$ &  $\deg D$ & Jordan sizes and multiplicities \\ \hline
    $-3/2$ &  $e_1$ & $(k_1^{e_1},k_2^{e_2}) = (2^{e_1}, 3^{\frac{1}{3} (n - e_1)})$ \\
    $-1/2$ &  $e_{1,1}$ & $[(k_{1,1}^{e_{1,1}}), (k_{2,1}^{e_{2,1}})] = [(2^{e_{1,1}}), (1^{n-2e_{1,1}})]$ 
\end{tabular} }
\vspace{1ex}
\end{center}
~
\item $Y_L$ has exactly two negative curves $D_1,D_2$. Then the self-intersections and the degrees of the negative curves determine the structure of the Jordan blocks as follows
\begin{center}
\vspace{1ex}
{\small \begin{tabular}{c|c|c|c}
    $(D_1^2,D_2^2)$ &  $(\deg D_1,\deg D_2)$ & type & Jordan sizes and multiplicities \\ \hline
    $(-1/2, -1)$ & $(e,e')$ & singular & $(k_{1}^{e_1}) = (2^{e})$ \\
    $(-1/3, -2)$ & $(e,e')$ & singular & $(k_{1}^{e_1}) = (3^{e})$ \\
    $(-3/2, -1/2)$ & $(e_1,e_2)$ & regular& $ (k_1^{e_1},k_2^{e_2}) = (2^{e_1}, 3^{e_2})$ 
\end{tabular} }
\vspace{1ex}
\end{center}
~
\item $Y_L$ has exactly three negative curves $D_1,D_2,D_3$ forming a chain $(D_1,D_2,D_3)$. Then $L$ has exactly two eigenvalues and the self-intersections and the degrees of the negative curves determine the structure of the Jordan blocks as follows:
\begin{center}
\vspace{1ex}
{\small \begin{tabular}{c|c|c|c}
    $(D_1^2,D_2^2,D_3^2)$ &  $(\deg D_i, i \in [3])$ & type & Jordan sizes and multiplicities \\ \hline
    $(-3/2, -1/2, -2)$ & $(e_1,e_2,e')$ & singular & $(k_1^{e_1},k_2^{e_2}) = (1^{e_1}, 2^{e_2})$ \\ 
    $(-3/2, -1, -2)$ & $(e_1,e_2,e')$ & singular & $(k_1^{e_1},k_2^{e_2}) = (2^{e_1}, 3^{e_2})$ \\
    $(-1/2, -2, -1)$ & $(e_{1,1},e_{2,1},e')$ & singular & $[(k_{1,1}^{e_{1,1}}), (k_{2,1}^{e_{2,1}})] = [(2^{e_{1,1}}) , (1^{e_{2,1}})]$ \\
    $(-2, -2, -1)$ & $(e_1,e_2,e_3)$ & regular & $(k_1^{e_1},k_2^{e_2},k_3^{e_3}) = (1^{e_1}, 2^{e_2},3^{e_3})$ \\
    $(-2, -1/2, -1/2)$ & $(e_{1,1},e_{1,2},e_{2,1})$ & regular & $[(k_{1,1}^{e_{1,1}}, k_{1,2}^{e_{1,2}}), (k_{2,1}^{e_{2,1}})] = [(1^{e_{1,1}}, 2^{e_{1,2}}), (1^{e_{2,1}})]$ 
\end{tabular} }
\vspace{1ex}
\end{center}
~
\item $Y_L$ has exactly three negative curves $D_1,D_2,D_3$ all intersecting at the same point. Let $e_i = \deg D_i$. Then $L$ is regular with 
\[
L = J_1(x_1)^{\boxplus e_1} \boxplus J_1(x_1+x_2)^{\boxplus e_2} \boxplus J_1(x_1 + 2x_2)^{\boxplus e_3}.
\]
\end{enumerate}
\end{proposition}
\begin{proof}
By the same argument as in the proof of \Cref{prop: W non DuVal singularity} and \Cref{prop: Cinfty very negative}, we have the sum of the sizes of the maximal Jordan blocks is $\kappa \leq 3$, otherwise it would have either a non-Du Val singularity or a negative curve with self-intersection less than $-2$. Therefore, the sum of sizes of the maximal Jordan blocks is either $0,1,2,3$. 

If $\kappa = 0$, $L$ has trivial regular component and $Y_L$ is smooth.

If $\kappa = 1$, then $L$ has a single eigenvalue with maximal Jordan block of size $1$. This falls in case (2) of \Cref{prop: smooth YL}.

If $\kappa = 2$, then either $L$ is smooth or it has a single eigenvalue with only Jordan blocks of size $2$. This is the first entry in case (1) of the statement.

If $\kappa = 3$, the analysis is more complex. The sizes of the maximal Jordan blocks define one of the three possible partitions of $3$, namely $(1,1,1)$, $(2,1)$ and $(3)$. The partition $(1,1,1)$ falls in case (5) of the statement. The partition $(2,1)$ defines either the second entry of case (2) or the third entry of case (4), depending on whether the pencil is regular or singular. The partition $(3)$ defines all other cases, which have different characteristics depending on whether the pencil is regular or singular, and depending on the presence of blocks of sizes $1$ or $2$. The analysis is straightforward.

In these cases, there is no eigenvalue reconstruction, since the number of eigenvalues is at most $3$, and there is only one configuration of three points on $\bbP^1$ up to $\GL_2$-equivalence.
\end{proof}

\section{Quadratic generation of the ideal of Schubert surfaces}\label{sec: quadrics}

The ideal of classical Schubert varieties is generated by linear forms and quadrics in the Pl\"ucker embedding of the Grassmannian. In this section, we address the question whether such property holds for Schubert varieties of linear matrix spaces as well. As far as we are aware, the standard proofs of quadratic generation for classical Schubert varieties are either combinatorial \cite[Sec.~12.5]{LBflagbook} or representation-theoretic \cite[Thm.~3.11]{Ram87}. Neither settings extend to the case of Schubert varieties of linear matrix spaces, even just in the case of surfaces. 

A related framework is that of rational surfaces obtained via rational maps $\mathbb P^2 \dashto \bbP^N$ defined by a linear system of plane curves with sufficiently generic base locus. The defining ideals of such surfaces were studied in \cite{GG91, GL93} for blow-ups of $\bbP^2$ at generic configurations of points, providing sharp bounds for quadratic generation. In \cite{Ber08} a more precise description of the defining equations was given, using determinantal methods coming from the geometry of tensors. In \cite{AlzRus02}, normality and homological properties of these surfaces were studied, whereas \cite{CavCon13} investigated the role of certain genericity assumptions in the choice of the base locus. Schubert varieties of matrix pencils fit naturally into this framework, since the map $\phi : \bbP^2 \dashto Y_L$ characterises $Y_L$ exactly as the image of a rational map having the scheme $Z$ as base locus. However, the schemes arising from matrix pencils are highly special and the results developed in these works do not apply directly.

Nevertheless, we prove the following results on quadratic generation for Schubert surfaces. We say that an ideal $I$ in the polynomial ring is \textit{generated in degree $d$} if it admits a minimal system of generators of degree at most $d$.
\begin{theorem}\label{thm: quadratic generation}
Let $Y_L \subseteq \Gr(n,n+m)$ be the Schubert surface of a pencil $L$ of size $n\times m$, and $I(Y_L)$ its ideal in the Pl\"ucker embedding $\bbP \Wedge^n \bbC^{n+m}$. Then $I(Y_L)$ is generated in degree $2$.
\end{theorem}

The proof is built on the fact that, in the case of pencils, we can control the divisor embedding the blow-up $\tilde{X}$ of $\bbP^2$, as described in \Cref{prop: resolution of phi}. More precisely, the proof is based on the following two facts:
\begin{itemize}
\item a lifting lemma, based on \cite[Lemma 1.6]{SaintD74}: \textit{Let $Y \subseteq \bbP^N$ be a linearly normal smooth variety and let $H \subseteq \bbP^N$ be a generic hyperplane. If the ideal of $Y \cap H$ is generated by quadrics in $H$, then the ideal of $Y$ is generated by quadrics.} 
\item Green's Theorem \cite[Thm. 4.a.1]{Green84}: \textit{Let $C$ be a smooth curve of genus $g$. Let $\calL$ be a very ample line bundle on $C$ of degree $\deg \calL \geq 2g+2$. Then the ideal of the image of $C$ in the complete linear series of $\calL$ is generated by quadrics.}
\end{itemize}
In fact, the proof yields a stronger statement than \Cref{thm: quadratic generation}: we will show that $Y_L$ satisfies Green's property $N_p$ for a certain range of $p$, see \Cref{corol: Np property}.

The lifting lemma from \cite[Lemma 1.6]{SaintD74} is stated for smooth varieties but the argument is elementary and it applies in higher generality, as provided in \Cref{lemma: lifting lemma}. 
\begin{lemma}\label{lemma: lifting lemma}
Let $Y \subseteq \bbP ^N$ be an irreducible, linearly non-degenerate, linearly normal variety with the property that, for a general hyperplane $H \subseteq \bbP^N$, the ideal of the hyperplane section $I(H \cap Y)$ is generated in degree $2$. Then the ideal of $Y$ is generated in degree $2$.
\end{lemma}
\begin{proof}
Let $S$ be the homogeneous coordinate ring of $\bbP^N$, and let $\ell$ be the defining equation of a general hyperplane $H$, so that $S/(\ell)$ is the coordinate ring of $H$. The assumption is that $I(Y \cap H)$ in $S/(\ell)$ is generated in degree $2$. The intersection $Y \cap H$ is linearly non-degenerate in $H$, so its ideal does not contain linear forms. Let $g_1 \vvirg g_r \in S/(\ell)$ be a system quadratic generators of $I(Y \cap H)$. The restriction exact sequence of $H$ gives 
\[
0 \to \calI_Y (1) \xto{\cdot \ell} \calI_Y(2) \to \calI_{Y \cap H} (2) \to 0.
\]
Since $Y$ is linearly normal, $H^1( \calI_Y(1)) = 0$, so the restriction map $H^0(\calI_Y(2)) \to H^0(\calI_{Y \cap H}(2))$ is surjective. Since $H^0(\calI_Y(2)) = I(Y)_2$ and similarly for $Y \cap H$, we have that the quadrics $g_1 \vvirg g_r$ lift to quadrics $f_1 \vvirg f_r \in I(Y)_2$. We show that these quadrics generate $I(Y)$. 

 Let $h \in I(Y)_k$. We use induction on $k$ to show that $h$ is generated by $f_1 \vvirg f_r$. If $k \leq 1$ there is nothing to prove since $Y$ is linearly non-degenerate. If $k \geq 2$, passing modulo $\ell$, we have $\bar{h} \in I(Y \cap H)$, so $\bar{h} = \bar{h}_1 g_1 + \cdots + \bar{h}_r g_r$ for certain $\bar{h}_j \in S/(\ell)$. Let $h_j \in S$ such that $\bar{h}_j = h_j \bmod (\ell)$. Then
 \[ 
 h = h_1 f_1 + \cdots + h_r f_r + \ell h'
 \]
 for some $h'$ of degree $k-1$. Since $h, f_1 \vvirg f_r \in I(Y)$, we obtain $\ell h' \in I(Y)$. Since $Y$ is irreducible and linearly non-degenerate, we obtain $h' \in I(Y)$. The induction hypothesis guarantees $h'$ is generated by $f_1 \vvirg f_r$, implying that $h$ is generated by $f_1 \vvirg f_r$ as well.
\end{proof}

The next result characterises the hyperplane sections of Schubert varieties of matrix pencils. 

\begin{proposition}\label{prop: hyperplane cut YL}
A generic hyperplane section of $Y_L$ is an irreducible, linearly normal, smooth curve $C$ with 
\begin{align*}
\deg(C) = \textstyle (n-\eta_{\#})^2 - \sum_{i,j} \mu_{ij}^2 \quad \text{ and } \quad g(C) = \textstyle  \binom{n-\eta_{\#}-1}{2} - \sum_{i,j} \binom{\mu_{ij}}{2}.
\end{align*}
\end{proposition}
\begin{proof}
Let $C$ be the generic hyperplane section of $Y_L$.
Since $Y_L$ is linearly normal, $C$ is linearly normal as well. Moreover, by \Cref{thm: rational singularities}, the singularities of $Y_L$ are $0$-dimensional, so a generic hyperplane avoids them, and by Bertini's Theorem $C$ is smooth. 

The degree of $C$ equals the degree of $Y_L$, and it was computed in \Cref{theorem: degree of YL}. 

The genus of $C$ coincides with the genus of its preimage in $\tilde{X}$. Since $\tilde{\phi} : \tilde{X} \to Y_L$ is the map given by the divisor $D$ of \Cref{prop: resolution of phi}, the preimage of $C$ is the generic member of the linear series defined by $D$.
We compute the genus $g(C)$ via the adjunction formula $2g(D)-2 = D (D + K_{\tilde X})$, where $K_{\tilde X}$ is the canonical divisor of $\tilde X$, see \cite[Sec. 2.4.1]{EisHar:3264}. The canonical divisor is computed via \cite[Prop. 2.20]{EisHar:3264}: it is $K_{\tilde X} = -3 H + \sum_{i \in [s]} \sum_{j \in [k_{i \, h_i}]} E_{ij}$. Since $H \cdot E_{ij} = 0$ and $E_{ij} \cdot E_{i'j'} = -\delta_{ii'}\delta_{jj'}$, 
we obtain
\[ \textstyle
D^2 = (n-\eta_\#)^2 - \sum_{i,j} \mu_{ij}^2 \quad , \quad 
D \cdot K_{\tilde X}= -3(n-\eta_\#) + \sum_{i ,j} \mu_{ij} .
\]
Thus, $$2g(C) = (n-\eta_\#-1)(n-\eta_\#-2) -  \textstyle \sum_{ij} \mu_{ij}(\mu_{ij}-1),$$
that is $g(C) = \binom{n-\eta_\# - 1}{2} - \sum_{ij}\binom{\mu_{ij}}{2}$.
\end{proof}

We now have all the ingredients to complete the proof of \Cref{thm: quadratic generation}.

\begin{proof}[Proof of \Cref{thm: quadratic generation}]
By \Cref{prop: resolution of phi}, the Schubert surface $Y_L$ is linearly normal. Therefore, by \Cref{lemma: lifting lemma}, it suffices to prove that the ideal of its general hyperplane section is generated by quadrics. 

\Cref{prop: hyperplane cut YL} characterises the degree and the genus of the hyperplane section $C$. We prove that the ideal of $C$ is generated by quadrics by applying Green's Theorem. To do so, we need to verify that $\deg(C) \geq 2 g(C) + 2$. By \Cref{prop: hyperplane cut YL}, this is equivalent to
\[
0 \leq \deg(C) - (2g(C) +2) = 3(n-\eta_{\#}) - \textstyle \sum_{ij} \mu_{ij} -4
\]
For every fixed $i$, we have $\sum_j \mu_{ij} = n_i$ coincides with the size of the regular component with eigenvalue $\theta_i$. Therefore, Green's condition reduces to 
\begin{equation}\label{eqn: Green condition reduction}
2n_{\reg} + 3(\eps_\bullet + \eta_\bullet) -4 \geq 0,
\end{equation}
where $n_\reg$ is the size of the regular component. If $n_\reg = 0$, then $Y_L$ is a Veronese variety by \Cref{cor: image of phi totally singular}, and its ideal is generated by quadrics.  If $n_\reg = 1$, then either $\eps_\bullet$ or $\eta_\bullet$ must be positive, otherwise $\dim L$ and $L$ contains a single matrix up to scaling: in either case \eqref{eqn: Green condition reduction} is satisfied and Green's Theorem guarantees that the ideal of $C$ is generated by quadrics. If $n_{\reg} \geq 2$, \eqref{eqn: Green condition reduction} is straightforward, and again Green's Theorem guarantees that the ideal of $C$ is generated by quadrics. This concludes the proof. 
\end{proof} 

\subsection{Property $N_p$ for Schubert surfaces}
The argument in the proof of \Cref{thm: quadratic generation} provides a result concerning Green's property $N_p$. Recall that, for a variety $Y \subseteq \bbP^N$, we say:
\begin{itemize}
    \item $Y$ satisfies property $N_0$ if it is linearly normal in $\bbP^N$;
    \item $Y$ satisfies property $N_1$ if it is linearly normal and $I(Y)$ is generated in degree $2$;
    \item for $p \geq 2$, $Y$ satisfies property $N_p$ if it satisfies property $N_{p-1}$ and the $(p-1)$-th syzygy of $I(Y)$ is linear.
\end{itemize}
In particular, \Cref{thm: quadratic generation} guarantees that every Schubert surface $Y_L$ satisfies property $N_1$. The same proof has the following consequence.
\begin{corollary}\label{corol: Np property}
Let $L$ be a pencil of matrices, and let $p \leq 2 n_{\reg} + 3(\epsilon_\bullet + \eta_\bullet) - 3$. Then the Schubert surface $Y_L$ satisfies Green's property $N_p$.
\end{corollary}
\begin{proof}
As in the proof of \Cref{thm: quadratic generation}, let $C = Y_L \cap H$ be the generic hyperplane section of $Y_L$, with $H = \{\ell = 0\}$ a generic hyperplane. The same calculation as in the proof of \Cref{thm: quadratic generation} shows that 
\[
\deg(C) \geq 2g(C)+1+p.
\]
In this range, Green's Theorem \cite[Thm.\ 4.a.1]{Green84} guarantees that $C$ satisfies property $N_p$. We prove that, in our setting, this property lifts to $Y_L$. To see this, first observe that $I(C) = I(Y_L) + (\ell)$: indeed, equality holds in degree $2$ and both ideals are generated in degree $2$ by \cref{thm: quadratic generation}. Let $S_{Y_L} = S/I(Y_L)$ be the homogeneous coordinate ring of $Y_L$ and $S/I(C)$ the homogeneous coordinate ring of $C$.

Let $F_\bullet$ be a minimal free resolution of $Y_L$, that is an exact sequence of $S$-modules
\[
0 \leftarrow S/I(Y_L) \leftarrow S \leftarrow F_1 \leftarrow F_2 \leftarrow \cdots ,
\]
with $F_j$ free modules. Since $\ell$ is a non-zero divisor in $S/I(Y_L)$, tensoring by $S/(\ell)$ preserves exactness \cite[Lemma 3.15]{EisSyz}. We obtain an exact sequence of $S/(\ell)$-modules
\begin{equation}\label{eqn: min free res IC}
0 \leftarrow S/(\ell) \otimes_S S/I(Y_L) \leftarrow S/(\ell) \otimes_S S \leftarrow S/(\ell) \otimes_S F_1 \leftarrow S/(\ell) \otimes_S F_2 \leftarrow \cdots ,
\end{equation}
where
\[
S/(\ell) \otimes_S S/I(Y_L) = S / (I(Y_L) + (\ell)) = S/I(C).
\]
Therefore, \eqref{eqn: min free res IC} is a free resolution of $C$. Moreover, it is minimal: the differentials are those in the minimal free resolution of $S/I(Y_L)$ modulo $\ell$, so they all have positive degree.

As observed above, $C$ has property $N_p$, therefore the resolution in \eqref{eqn: min free res IC} is linear up to homological degree $p-1$. The same holds for the minimal free resolution of $I(Y_L)$, showing that $Y_L$ has property $N_p$.
\end{proof}

\subsection{Quadratic generation in higher dimensional cases}\label{sec: higher quadratic}

Similarly to the normality properties discussed in \Cref{sec: higher normality}, it is natural to ask whether quadratic generation, or higher Green's properties, holds for the ideal of Schubert varieties of linear spaces of dimension larger than $2$. As mentioned before, this holds for classical Schubert varieties.

Quadratic generation holds for the ideal of matroid Schubert varieties in the Segre embedding, which coincides with the Pl\"ucker embedding of the corresponding Schubert varieties of the linear space of diagonal matrices. We give an outline of the proof.
\begin{proposition}
    Let $Y^{\mathrm{mat}} \subseteq (\bbP^1)^{\times n}$ be a matroid Schubert variety and let $Y_L \subseteq \bbP^{2^n-1}$ be its image in the Segre embedding. Then $I(Y_L)$ is generated in degree $2$.
\end{proposition}
\begin{proof}
    In \cite{ardila2016closure} the multihomogeneous ideal of the matroid Schubert varieties in the multigraded coordinate ring of $(\mathbb{P}^1)^{\times n}$ is given in terms of the circuits of the matroid: let $J_{\mathrm{circuit}} \subseteq \bbC[(\bbP^1)^{\times n}]$ be this ideal. Here $\bbC[(\bbP^1)^{\times n}]$ denotes the homogeneous coordinate ring of $(\bbP^1)^{\times n}$, that is a polynomial ring in $2n$ variables with a $\bbZ^n$-multigrading. The generators of $J_{\mathrm{circuit}}$ are multilinear: their degree in the multigrading of $\bbC[(\bbP^1)^{\times n}]$ is a vector of $0$'s and $1$'s. Let $N := 2^n-1$ and let $\bbC[\bbP^N]$ be the homogeneous coordinate ring of the ambient space of the Segre embedding of $(\mathbb{P}^1)^{\times n}$. Let $I_\Seg$ be the ideal of the Segre variety, which is generated by quadrics. 

    The Segre embedding $(\mathbb{P}^1)^{\times n} \to \bbP^N$ defines an isomorphism $S / I_\Seg \to \bbC[(\bbP^1)^{\times n}]_{\bbZ(1 \vvirg 1)}$, where $\bbC[(\bbP^1)^{\times n}]_{\bbZ(1 \vvirg 1)} = \bigoplus_{m \in \bbZ} \bbC[(\bbP^1)^{\times n}]_{(m \vvirg m)}$ is the \emph{diagonal subalgebra} of $\bbC[(\bbP^1)^{\times n}]$. By construction, $I(Y_L) / I_\Seg$ is the preimage of $J_{\mathrm{circuit}} \cap \bbC[(\bbP^1)^{\times n}]_{\bbZ(1 \vvirg 1)}$.

    Given a multilinear generator $f \in J_{\mathrm{circuit}}$ of degree $\bfe \in \bbZ^n$, consider the $(2^{n- |\bfe|})$-dimensional subspace $f \cdot \bbC[(\bbP^1)^{\times n}]_{(1 \vvirg 1) - \bfe} \subseteq J_{\mathrm{circuit}} \cap \bbC[(\bbP^1)^{\times n}]_{(1 \vvirg 1)}$. It is easy to prove that $J_{\mathrm{circuit}} \cap \bbC[(\bbP^1)^{\times n}]_{\bbZ(1 \vvirg 1)}$ is generated by all elements constructed in this way. Since they have multidegree $(1 \vvirg 1)$, these equations all pullback in $S / I_\Seg$ to linear equations. This shows $I(Y_L) = I_\Seg + (I(Y_L)_1)$, thus it is generated in degree $2$.
\end{proof}
In general, the ideals of higher dimensional Schubert varieties of linear matrix spaces are not generated by quadrics. One can verify that already the ideal of the Schubert variety in \Cref{example: counterexample linear normality} has minimal generators in degree $3$. However, in that case the Schubert variety is cut out by quadrics scheme-theoretically. In general, even set-theoretically quadratic generation can fail as shown in the following example.
\begin{example}\label{example: counterexample quadratic ideal}
Let $\dim V = \dim W = 4$ and let $L \subseteq V^* \otimes W$ be a general $7$-dimensional linear space. We show that $Y_L \subseteq \Gr(4, V \oplus W)$ is not set-theoretically cut out by quadrics. To see this, we determine a linear subspace $\Sigma \subseteq \bbP \Wedge^4( V \oplus W)$ such that $Y_L \cap \Sigma$ is a cubic surface contained in a $\bbP^3$: if $Y_L$ was cut out by quadrics, then its linear sections would be cut out by quadrics as well, yielding a contradiction. 

The linear space $\Sigma$ is described by a Schubert condition on $\Gr(4, V \oplus W)$. Let $U \subseteq W$ be a general $3$-dimensional subspace and let $\Sigma_U = \{ [E] \in \Gr(4, V \oplus W) : U \subseteq E \subseteq U \oplus V\}$. This is a classical Schubert variety whose associated partition is $(4,4,4,1)$: in particular, it defines a $\bbP^3 \subseteq \Gr(4, V \oplus W)$ parameterised by $\bbP V \to \Sigma_U$ sending $v \mapsto U \oplus \langle v \rangle$. 

We prove that $Y_L \cap \Sigma_U$ is a cubic surface. First, we show $Y_L \cap \Sigma_U$ contains a cubic surface. Set $L_U = L \cap (V^* \otimes U)$: by genericity $\dim L_U = 3$. For every $A \in L_U$, we have $\rank A \leq 3$: if equality holds, then $\lim_{t \to \infty} \Gamma_{tA} = U \oplus \ker A \in Y_L \cap \Sigma_U$ as in \Cref{prop: im plus ker}. This parameterises a surface 
\[
S_U := \overline{ \{ U \oplus \ker A : A \in L, \im A = U\} }\subseteq Y_L \cap \Sigma_U . \]
Observe that $\deg S_U = 3$ in the Pl\"ucker embedding of $\Gr(4,V \oplus W)$: for a fixed basis $A_1,A_2,A_3$ of $L_U$, the surface $S_U$ is parameterised by $v\in \bbP V$ such that $A_1v,A_2v,A_3v$ are linearly dependent elements of $U$. Since $\dim L_U = \dim U = 3$, this is equivalent to the vanishing of a $3 \times 3$ determinant. This shows $S_U \subseteq Y_L \cap \Sigma_U$.

For the inclusion $Y_L \cap \Sigma_U \subseteq S_U$, consider $[E] \in Y_L \cap \Sigma_U$ and write $E = U \oplus \langle v \rangle$. Since $[E] \in Y_L$, there is a curve $A(t)$ in $L$ with $E = \lim_{t \to \infty} \Gamma_{A(t)}$. Let $[0:A_0] \in \bbP(\bbC \oplus L)$ be the limit point $[0:A_0] = \lim_{t \to \infty} [1 : A(t)]$. Therefore, $A(t) = t^k A_0 +  o(t^{k})$; since $A(t) \in L$, we have $A_0 \in L$ as well. The same calculation as above shows $\im A_0 \subseteq U$ and $v \in \ker A_0$. This shows $E = U \oplus \ker A_0$, so $[E] \in S_U$ and $\YL \cap \Sigma_U=S_U$.
\end{example}

{\small \textbf{Acknowledgements.} 
We warmly thank Mateusz Micha{\l}ek for the many valuable and inspiring discussions that have accompanied this work. 
We thank Enrica Floris for suggestions on the geometry of blow-ups, Simon Telen for conversations on toric methods, Laurent Bus\'e, Luca Chiantini, Giorgio Ottaviani and Kristian Ranestad for references and discussions on defining equations of images of blow-ups, and Viktoriia Borovik for discussions on Gr\"obner bases and triangulations.
Part of this work was realized during the \textit{Research in Pairs program at CIRM Trento in July 2025: we thank CIRM and FBK for financial support and for providing an excellent working environment during our research visit}. 
V.G. and F.G. thank MPI-MiS Leipzig for the hospitality during the Chow Lectures 2023, where discussions with Mateusz Micha{\l}ek and June E. Huh further inspired this work.
V.G. and F.G. benefited from mutual visits at the MPI-CBG Dresden and the IMT in Toulouse, whom they thank for the support and hospitality.
V.G. is member of the Italian group GNSAGA-INdAM. }

{
\bibliographystyle{alphaurl}
\bibliography{Schubert_pencil.bib}
}

\end{document}